[1994/12/01]
\documentclass[11pt]{article}
\usepackage[margin=1.08in]{geometry}
\usepackage{cite}
\makeatletter
\let\NAT@parse\undefined
\makeatother
\usepackage{hyperref}
\hypersetup{final}
\usepackage{amsmath,enumerate,amsfonts,amssymb,color,graphicx,amsthm,stmaryrd}
\usepackage{amsmath,amsthm}

\makeatletter
\newtheorem{claim}{Claim}[section]
\newdimen\bibspace
\renewenvironment{thebibliography}[1]{%
	\section*{\refname 
		\@mkboth{\MakeUppercase\refname}{\MakeUppercase\refname}}%
	\list{\@biblabel{\@arabic\c@enumiv}}%
	{\settowidth\labelwidth{\@biblabel{#1}}%
		\leftmargin\labelwidth
		\advance\leftmargin\labelsep
		\itemsep\bibspace
		\parsep\z@skip     %
		\@openbib@code
		\usecounter{enumiv}%
		\let\p@enumiv\@empty
		\renewcommand\theenumiv{\@arabic\c@enumiv}}%
	\sloppy\clubpenalty4000\widowpenalty4000%
	\sfcode`\.\@m}
{\def\@noitemerr
	{\@latex@warning{Empty `thebibliography' environment}}%
	\endlist}

\makeatother

\chardef\bslash=`\\ 

\newtheorem{thm}{Theorem}[section]
\newtheorem{cor}[thm]{Corollary}
\newtheorem{lem}[thm]{Lemma}
\newtheorem{prop}[thm]{Proposition}

\theoremstyle{definition}

\theoremstyle{remark}
\newtheorem{rem}{Remark}[section]

\newcommand{\eval}[2][\right]{\relax
	\ifx#1\right\relax \left.\fi#2#1\rvert}

\begin{document}
	\title{Isolated Singularities of Solutions to the Yamabe Equation with boundary in Dimension 3 and 4}
	\author{Yuxuan Liao, Yuexiao Ma}
	\date{}
	\maketitle
	\begin{abstract}
This paper studies the asymptotic behavior of positive solutions to the boundary Yamabe equation near an isolated singularity  when the metric is not conformally flat. In dimensions 3 and 4,  we establishes the sharp upper bound and, in the non-removable case, the matching lower bound, which also gives a necessary and sufficient condition for removability. Moreover, every solution with a non-removable singularity is shown to be asymptotically cylindrically symmetric, without relying on a global classification of Fowler-type singular solutions. These  results aslo extend the flat half-space theory of Caffarelli–Jin–Sire–Xiong (2014) to non-flat boundary geometries and provide boundary analogues of the interior theories developed by Marques (2008) and Xiong–Zhang (2022).
	\end{abstract}

	\renewcommand{\sectionmark}[1]{}

	\section{Introduction}
	This paper investigates the local behavior of positive solutions to the boundary Yamabe equation near an isolated singularity on the boundary. More precisely, let \(\mathcal{B}_1^+\subset \mathbb R^n\), $n\geq 3$, be a half-ball equipped with a smooth Riemannian metric \(g\). We consider
	\begin{equation}\label{eq:Main_PDE}
		\begin{cases}
			-L_g u = 0,                & \text{in } \mathcal{B}_1^+,                         \\
			B_g u  = u^{\frac{n}{n-2}}, & \text{on } \partial'\mathcal{B}_1^+\setminus\{0\},
		\end{cases}
	\end{equation}
	where
	\[
	L_g=\Delta_g-\frac{n-2}{4(n-1)}R_g,
	\qquad
	B_g u=\frac{\partial u}{\partial \nu_g}+\frac{n-2}{2}h_g u
	\]
	are the conformal Laplacian and the conformal boundary operator, respectively. Here, \(\Delta_g\) denotes the Laplace--Beltrami operator, \(R_g\) is the scalar curvature of \(g\), \(h_g\) is the mean curvature of the boundary, and \(\nu_g\) is the outward unit normal. Moreover, \(\mathcal{B}_r\) denotes the Euclidean ball of radius \(r\) in \(\mathbb{R}^n\), \(B_r\) denotes the Euclidean ball of radius \(r\) in \(\mathbb{R}^{n-1}\), the superscript \(+\) denotes the upper half-ball, and \(\partial'\Omega:=\partial\Omega\cap\{x_n=0\}\). Thus \(\partial'\mathcal{B}_r^+=B_r\). By elliptic regularity, solutions are smooth away from the singular point. The main objective is to determine their behavior as \(x\to0\).

	On a compact smooth Riemannian manifold \((M,g)\) of dimension \(n \ge 3\), the Yamabe problem, which concerns the existence of constant scalar curvature metrics in the conformal class of \(g\), was solved affirmatively through Yamabe \cite{yamabeDeformationRiemannianStructures1960}, Trudinger \cite{trudingerRemarksConcerningConformal1968}, Aubin \cite{aubinEquationsDifferentiellesNon1976}, and Schoen \cite{schoenConformalDeformationRiemannian1984}. The problem is equivalent to solving the Yamabe equation
	\[
	-L_g u = n(n-2)\operatorname{sign}(\lambda_1)u^{\frac{n+2}{n-2}}
	\quad \text{on } M,\qquad u>0,
	\]
	where \(\operatorname{sign}(\lambda_1)\in\{-1,0,1\}\) is the sign of the first eigenvalue of \(-L_g\) on \(M\).

	Solutions of the Yamabe equation on the standard unit sphere \(S^n\) were classified by Obata \cite{obataConjecturesConformalTransformations1971}. Caffarelli, Gidas, and Spruck \cite{caffarelliAsymptoticSymmetryLocal1989} further studied the isolated singularity problem for the Yamabe equation. More precisely, let \(N\) be the north pole of \(S^n\). By stereographic projection, the Yamabe equation on \(S^n\setminus\{N\}\) is transformed into the Euclidean critical equation
	\[
	-\Delta u = n(n-2)u^{\frac{n+2}{n-2}},
	\qquad u>0
	\quad\text{in }\mathbb R^n .
	\]
	They proved that every positive solution is a standard bubble of the form
	\[
	u(x)=
	\lambda^{\frac{n-2}{2}}
	\left(1+\lambda^2|x-x_0|^2\right)^{\frac{2-n}{2}},
	\]
	where \(\lambda>0\) and \(x_0\in\mathbb R^n\).
	They also obtained that every positive solution of the Euclidean Yamabe equation in the punctured unit ball \(B_1\setminus\{0\}\), with \(0\) being a non-removable singularity, satisfies
	\[
	u(x)=u_0(x)(1+o(1))
	\quad\text{as }x\to0,
	\]
	where \(u_0\) is a Fowler solution of the form
	\[
	u_0(|x|)
	=
	|x|^{-\frac{n-2}{2}}\psi(\log |x|).
	\]
	Here \(\psi\) is a positive periodic solution of the corresponding Fowler ordinary differential equation of
	\[
	\psi''(t)-\frac{(n-2)^2}{4}\psi(t)
	+\frac{n(n-2)}{4}\psi(t)^{\frac{n+2}{n-2}}=0.
	\]
	Subsequently, Li \cite{liLocalAsymptoticSymmetry1996}  simplified and extended the asymptotic-symmetry argument. Korevaar, Mazzeo, Pacard and Schoen \cite{korevaarRefinedAsymptoticsConstant1999} improved the remainder to \(O(|x|^\alpha)\) for some \(\alpha>0\) and obtained a first-order expansion. Recently, Han, Li, and Li \cite{hanAsymptoticExpansionsSolutions2021} established asymptotic expansions to arbitrary order for the Yamabe equation and the \(\sigma_k\)-Yamabe equation. For further background on isolated singularities for the Yamabe equation, including higher-order and fractional analogues and related open problems, we refer to the recent survey by Du, Jin, Xiong and Yang \cite{duSurveyIsolatedSingularity2025}.

	When the underlying manifold has nonempty boundary, the boundary Yamabe problem asks for a conformal metric with constant scalar curvature in the interior and constant mean curvature on the boundary. It was formulated and studied by Escobar \cite{EscobarYamabeProblemManifolds1992,escobarConformalDeformationRiemannian1996} and has since been investigated extensively; see, for instance, \cite{escobarConformalDeformationRiemannian1992,marquesPrioriEstimatesYamabe2005,chenAsymptoticBehaviorLeast2015} and the references therein.

	The  flat case for \eqref{eq:Main_PDE} is
	\[
	\begin{cases}
		-\Delta u=0, & \text{in }\mathcal{B}_1^+,\\
		\partial_\nu u=u^{\frac{n}{n-2}},
		& \text{on }B_1\setminus\{0\}.
	\end{cases}
	\]
	Caffarelli, Jin, Sire and Xiong \cite{caffarelliLocalAnalysisSolutions2014} proved the optimal two-sided growth estimate and asymptotic radial symmetry of the boundary trace near a non-removable isolated singularity, together with cylindrical symmetry for the corresponding global singular solutions. While for non-conformally-flat interior metrics, curvature terms enter both the blow-up analysis and the associated Pohozaev identities. Marques \cite{MarquesIsolatedSingularitiesSolutions2008} established Fowler-type asymptotics in dimensions \(3\le n\le5\), and Xiong and Zhang \cite{xiongIsolatedSingularitiesSolutions2022} extended the result to dimension \(6\). Related results for asymptotically flat punctured metrics were obtained by Han, Xiong and Zhang \cite{hanAsymptoticBehaviorSolutions2023}. In these interior results, the local asymptotic behavior is ultimately identified through the classification of the Euclidean global singular solutions.

	We study the corresponding curved-metric boundary problem in dimensions \(n=3,4\). The first main result gives a sharp removability criterion and the optimal two-sided growth rate for non-removable singularities.

	\begin{thm}
		Suppose that \(n=3\) or \(n=4\), and let \(u\) be a nonnegative solution of \eqref{eq:Main_PDE}. Then \(u\) extends smoothly near \(0\) if and only if
		\[
		\liminf_{x\to0} d_g(x,0)^{\frac{n-2}{2}}u(x)=0.
		\]
		If the singularity is non-removable, then there exist positive constants \(c,C>0\) such that
		\[
		c\,d_g(x,0)^{-\frac{n-2}{2}}
		\le u(x)\le
		C\,d_g(x,0)^{-\frac{n-2}{2}}.
		\]
	\end{thm}

	The second main result establishes the asymptotic cylindrical symmetry of non-removable singularities.

	\begin{thm}
		Suppose that \(n=3\) or \(n=4\), and let \(u\) be a nonnegative solution of \eqref{eq:Main_PDE} with a non-removable singularity at \(0\). Then there exists \(0<\alpha<1\) such that
		\[
		u(x)=\bar{u}(|x'|,x_n)\bigl(1+O(|x|^\alpha)\bigr)
		\quad \text{as } x\to0,
		\]
		where \(\bar{u}(|x'|,x_n)\) denotes the average of \(u\) over the slice \(\{(y',x_n): |y'|=|x'|\}\).

	\end{thm}

	The main difficulties in the boundary setting are twofold.
	The first concerns the upper bound. In contrast with the interior Yamabe equation, where blow-up analysis is centered at interior concentration points, the critical nonlinearity in the boundary Yamabe problem appears in the boundary condition. Nevertheless, a sequence violating the desired estimate need not concentrate on the boundary. Our argument therefore has a two-stage structure. We first exclude boundary blow-up by a moving-spheres argument adapted to the nonlinear boundary condition. Combined with Harnack and gradient estimates up to the boundary, this boundary control is then transferred to the remaining interior regime through an elliptic reduction of the locally normalized boundary problem, thereby ruling out interior concentration.

	The second concerns the role of Fowler solutions. In the interior Yamabe problem, the limiting profiles are identified with Fowler solutions of the Euclidean critical equation, whereas no corresponding classification is available for the nonlinear boundary problem considered here. In this paper, both the removability criterion and the asymptotic cylindrical symmetry are established directly, without relying on such a classification.
	After the upper bound and the Harnack inequality have been established, we pass to cylindrical variables and combine a boundary Pohozaev identity with one-dimensional estimates for an averaged profile. These estimates show that a vanishing normalized liminf forces the whole profile to decay, and hence the singularity to be removable. By contraposition, every non-removable singularity satisfies the optimal lower bound. The asymptotic cylindrical symmetry is obtained separately by a perturbed moving-spheres argument centered at nearby boundary points, with auxiliary barriers controlling the metric and boundary errors. This yields the stated \(O(|x|^\alpha)\) asymptotic cylindrical symmetry. Thus both conclusions are obtained from local analysis, without first identifying the solution with a classified global Fowler-type profile.

	The restriction to dimensions \(3\) and \(4\) reflects a loss caused by the boundary geometry. In the interior problem, conformal normal coordinates give the effective second-order flatness \(g_{ij}=\delta_{ij}+O(|x|^2)\). In the present boundary setting, the conformal Fermi normalization yields only first-order flatness, and the Kelvin-transform errors in the moving-spheres argument are therefore one order larger. The auxiliary functions must absorb both interior geometric errors and boundary mean-curvature errors, and with the present estimates the far-region barrier and the continuation argument for the upper bound close only for \(n=3,4\). The same geometric loss reappears in the Pohozaev--ODE analysis. Within the present argument, the final absorption requires the decay exponent \(\beta=(n-2)/2\) to satisfy \(\beta\le 1\), which leads to the dimensional restriction \(n\le4\).

	The paper is organized as follows. Section~\ref{sec:upper_bound} establishes the upper-bound estimate and derives the Harnack inequality and gradient estimates. Section~\ref{sec:main_part} proves the removability criterion and the lower bound and establishes the asymptotic cylindrical symmetry. The Appendix contains the technical estimates used in the moving-spheres and asymptotic-symmetry arguments.
		\\
		\noindent\textbf{Acknowledgements.} The author would like to thank Prof. Jiguang Bao and Prof. Jingang Xiong for his helpful suggestions and  constant support.

	\section{The upper bound}\label{sec:upper_bound}
	This section establishes the following upper bound.
	\begin{thm}\label{thm:upper_bound}
		Suppose that \(n=3\) or \(n=4\), and let \(u\) be a positive solution to the equation \eqref{eq:Main_PDE}, then
		\begin{equation}\label{eq:upper_bd}
			\limsup_{x\to0} d_g(x,0)^{\frac{n-2}{2}}u(x) < \infty.
		\end{equation}
	\end{thm}

	The estimate is proved by a blow-up argument. The proof consists of two steps. We first exclude boundary blow-up behavior that would violate \eqref{eq:upper_bd}.
	Then, the interior behavior is shown to be controlled by the boundary, thereby excluding interior blow-up as well.

	\subsection{The boundary case}
	The argument proceeds by contradiction. Assume that \eqref{eq:upper_bd} fails for some sequence on the boundary. Then there exists a sequence \(x_k\subset\partial'\mathcal{B}_1^+\to0\) such that
	\begin{equation*}
		d_g(x_k,0)^{\frac{n-2}{2}}u(x_k)\to\infty \quad \text{as } k\to\infty.
	\end{equation*}

	\subsubsection{Step 1: Blow-up analysis}

	It is first shown that the points \(x_k\) can be chosen as local maximum points of \(u\) on the boundary. The proof of this fact is standard and we  describe it for readers’ convenience.

	Set \(f_k(y)=u(y)(d_k-d_g(y,x_k))^{\frac{n-2}{2}}\) on \(\partial'\mathcal B_1^+\cap\mathcal B_{d_k}(x_k)\), where
	\(d_k=\frac12d_g(x_k,0)\). Under this assumption, \(f_k(x_k)\to\infty\), while \(f_k=0\) on \(\partial\mathcal B_{d_k}(x_k)\). Let \(\hat x_k\) be a maximum point of \(f_k\) and set \(\sigma_k=\frac12(d_k-d_g(\hat x_k,x_k))\).
	The definition of \(\hat{x}_k\) gives
	\begin{equation}\label{des:domain}
		u(\hat{x}_k) (2\sigma_k)^{\frac{n-2}{2}} \geq u(x_k) d_k^{\frac{n-2}{2}} \to \infty.
	\end{equation}
	For any \(y\in\mathcal B_{\sigma_k}(\hat x_k)\cap\partial'\mathcal B_1^+\), one has \(d_k-d_g(y,x_k)\ge\sigma_k\), which implies
	\begin{equation*}
		u(y) \leq u(\hat{x}_k) \left(\frac{2\sigma_k}{d_k - d_g(y, x_k)}\right)^{\frac{n-2}{2}} \leq 2^{\frac{n-2}{2}} u(\hat{x}_k).
	\end{equation*}

	Define the blow-up sequence
	\begin{equation*}
		\hat{v}_k(y) = u(\hat{x}_k)^{-1} u\left(\exp_{\hat{x}_k}\bigl(u(\hat{x}_k)^{-\frac{2}{n-2}}y\bigr)\right), \quad y \in u(\hat{x}_k)^{\frac{2}{n-2}} \exp_{\hat{x}_k}^{-1}
		\bigl(\mathcal B_{\sigma_k}^+(\hat{x}_k)\bigr).
	\end{equation*}

	For every fixed \(R>0\), the rescaled half-ball \(B_{2R}^+\) is contained
	in the domain of \(\hat v_k\) for all sufficiently large \(k\), and it does not
	contain the rescaled singular point. Moreover,
	\[
	\hat{v}_k(0)=1,\qquad \hat{v}_k(y',0)\le C_0 \quad \text{for } |y'|\le 2R.
	\]

	The function \(\hat{v}_k\) satisfies
	\[
	-L_{g_k}\hat{v}_k=0\quad\text{in }B_{2R}^+,
	\]
	where \[g_k=(D\exp_{\hat{x}_k})^*_{u(\hat{x}_k)^{-\frac{2}{n-2}}y} g_{\exp_{\hat{x}_k}(u(\hat{x}_k)^{-\frac{2}{n-2}}y)}.\] On the flat boundary,
	\[
	\partial_{\nu_{g_k}}\hat{v}_k+\frac{n-2}{2}u(\hat{x}_k)^{-\frac{2}{n-2}} h_{g_k}\hat{v}_k=\hat{v}_k^{\frac n{n-2}}.
	\]
	Equivalently,
	\[
	\partial_{\nu_{g_k}}\hat{v}_k+q_k(y')\hat{v}_k=0,
	\]
	where
	\[
	q_k(y')=\frac{n-2}{2}u(\hat{x}_k)^{-\frac{2}{n-2}} h_{g_k}(y')-\hat{v}_k(y',0)^{\frac{2}{n-2}}.
	\]
	By the boundary trace control, \(\|q_k\|_{L^\infty(\partial'B_{2R}^+)}\le C_R\).
	Thus the function \(\hat v_k\) solves a uniformly elliptic equation with a uniformly oblique
	Robin boundary condition whose zeroth-order coefficient is uniformly bounded.

	The local boundary Harnack inequality and the local maximum principle for
	nonnegative solutions with oblique Robin boundary condition yield
	\[
	\sup_{B_R^+}\hat{v}_k\le C_R\hat{v}_k(0)=C_R.
	\]

	By the preceding local \(L^\infty\) estimate and standard elliptic estimates, a subsequence (still denoted \(\hat{v}_k\)) converges in \(C^2_{\rm loc}(\overline{\mathbb{R}^n_+})\) to a positive function \(V\) satisfying the limit equation
	\begin{equation}\label{eq:V}
		\begin{cases}
			-\Delta V = 0,                                        & \text{in } \mathbb{R}^n_+,          \\\
			\dfrac{\partial V}{\partial \nu} = V^{\frac{n}{n-2}}, & \text{on } \partial \mathbb{R}^n_+.\\
		\end{cases}
	\end{equation}
	By the classification theorem of Li and Zhu \cite{liUniquenessTheoremsMethod1995} for the boundary Yamabe problem, \(V\) must be a standard bubble:
	\begin{equation*}
		V(y) = c_n \left( \frac{\lambda}{|y' - y_0'|^2 + (y_n + \lambda)^2} \right)^{\frac{n-2}{2}}
	\end{equation*}
	for some \(\lambda > 0\) and \(y_0' \in \mathbb{R}^{n-1}\).

	Since \(V(\cdot,0)\) has a nondegenerate strict tangential maximum at \(y_0'\), and since the boundary equation gives the correct one-sided normal monotonicity into the half-space, \((y_0',0)\) is a strict relative maximum of \(V\) on \(\overline{\mathbb R^n_+}\). By the \(C^2_{\rm loc}\)-convergence up to the boundary, there exist points \(y_k'\to y_0'\) such that \((y_k',0)\) is a local maximum point of \(\hat v_k\) on the flat boundary. Thus, by redefining \(x_k\) as the pre-image of \((y_k',0)\), it may be assumed from the beginning that \(x_k\) is a local maximum point of \(u\) on \(\partial' \mathcal{B}_1^+\) and that
	\begin{equation*}
		u(x_k) \sigma_k^{\frac{n-2}{2}} \rightarrow \infty \quad \text{as } k \rightarrow \infty.
	\end{equation*}

	The metric is next simplified near \(x_k\) by conformal Fermi coordinates. Let \(x\) denote the corresponding local coordinate variable, centered so that \(x=0\) at \(x_k\). In these coordinates the original metric \(g\) can be written as \(g = dx_n^2 + g_{ab}(x)dx^a dx^b\).

	Let \(\bar g_k=w_k^{\frac{4}{n-2}}g\) be the conformally changed metric in these coordinates. The conformal factor \(w_k\) is chosen with \(w_k(0)=1\) and \(\partial_n w_k(0)=\frac{n-2}{2}h_g(0)\). Then
	\begin{equation*}
		-\partial_n w_k(0)+\frac{n-2}{2}h_g(0)=\frac{n-2}{2}h_{\bar g_k}(0),
	\end{equation*}
	which gives \(h_{\bar g_k}(0)=0\); therefore, near \(x=0\), \(h_{\bar g_k}(x)=O(|x|)\). Moreover, since \(\Delta w_k(0)=\sum_{i=1}^{n-1}\partial_{ii}w_k(0)+\partial_{nn}w_k(0)\), the value of \(\partial_{nn}w_k(0)\) may be chosen so that
	\begin{equation*}
		-\frac{4(n-1)}{n-2}\Delta w_k(0)+R_g(0)=R_{\bar g_k}(0),
	\end{equation*}
	guaranteeing \(R_{\bar g_k}(0)=0\) and thus \(R_{\bar g_k}(x)=O(|x|)\). In the resulting conformal Fermi coordinates,
	\begin{equation*}
		(\bar g_k)_{nn}(x)=1,\quad (\bar g_k)_{\alpha n}(x)=0\quad (1\le\alpha\le n-1),\quad
		\det \bar g_k=1,
	\end{equation*}
	and \(w_k(0)=1\), \(\nabla_{x'}w_k(0)=0\).

	Let \(u_k=w_k^{-1}u\), then \(u_k\) satisfies
	\begin{equation*}
		\begin{cases}
			-L_{\bar g_k}u_k = 0, & \text{in } B_\delta^+,\\
			B_{\bar g_k}u_k := \dfrac{\partial u_k}{\partial \nu_{\bar g_k}}+\frac{n-2}2h_{\bar g_k}u_k = u_k^{\frac{n}{n-2}}, & \text{on } \partial'B_\delta^+\setminus\{x(0)\},
		\end{cases}
	\end{equation*}
	where \(x(0)\) denotes the coordinate vector of the original singular boundary point in the conformal Fermi chart centered at \(x_k\).

	Now denote \(M_k:=u_k(0)\), set \(\varepsilon_k:=M_k^{-\frac{2}{n-2}}\), and define \(v_k(y)=M_k^{-1}u_k(\varepsilon_ky)\). Then \(v_k\)
	satisfies
	\begin{equation*}
		\begin{cases}
			-L_{g_k}v_k(y)=0,                    & \text{in } \mathcal{B}_{R_k}^+,                          \\
			B_{g_k}v_k(y) = v_k^{\frac{n}{n-2}}, & \text{on } \partial'\mathcal{B}_{R_k}^+\setminus\{S_k\},
		\end{cases}
	\end{equation*}
	where \((g_k)_{ij}(y)=(\bar g_k)_{ij}(\varepsilon_k y)\), \(R_k=\delta\varepsilon_k^{-1}\), and
	\(S_k=\varepsilon_k^{-1}x(0)\).

	The assumption on \(x_k\) gives \(d(x_k,0)\varepsilon_k^{-1} \to \infty\), which implies
	\begin{equation*}
		|S_k| \to \infty.
	\end{equation*}
	Moreover, \(v_k \to V\) in \(C^2_{\rm loc}(\mathbb{R}^n_+)\) as \(k \to \infty\), where \(V\) satisfies \eqref{eq:V}. Since \(x_k\) is a local maximum of \(u\) on the boundary and \eqref{des:domain} holds,
	\begin{equation*}
		V(0) = 1 \quad \text{and} \quad \nabla_{y'} V(0) = 0.
	\end{equation*}
	Consequently, the limiting bubble has the specific form
	\begin{equation*}
		V(y) = \left( \frac{(n-2)^2}{|y'|^2 + (y_n + n-2)^2} \right)^{\frac{n-2}{2}}.
	\end{equation*}

	The moving-sphere argument also requires the following coarse lower bound.
	\begin{prop}[Coarse lower bound]\label{prop:lower}
		For the scaled functions \(v_k\), there exists a constant \(\Lambda>0\), independent of \(k\), such that
		\begin{equation*}
			v_k(y) \ge \Lambda M_k^{-1}
			\quad \text{for all } |y| \le \frac{\delta}{2}M_k^{\frac{2}{n-2}}.
		\end{equation*}
	\end{prop}
	\begin{proof}
		See Appendix~\ref{app:A}.
	\end{proof}

	In these coordinates,
	\begin{equation*}
		(\bar g_k)_{ij}(x)=\delta_{ij}+O(|x|),\quad R_{\bar g_k}(x)=O(|x|),\quad h_{\bar g_k}(x)=O(|x|).
	\end{equation*}
	Consequently,
	\begin{equation*}
		\Delta_{g_k}=\Delta+\bar b_j\partial_j+\bar d_{ij}\partial_{ij},
	\end{equation*}
	and
	\begin{equation*}
		\frac{\partial}{\partial\nu_{g_k}}=-\partial_{y_n}
		\qquad\text{on }\partial'\mathcal B_{R_k}^+.
	\end{equation*}
	Here
	\begin{gather*}
		b_j(x) = \partial_i (\bar g_k)^{ij}(x) = O(1),\quad d_{ij}(x) = (\bar g_k)^{ij}(x) - \delta_{ij} = O(|x|).
	\end{gather*}
	Thus
	\begin{equation*}
		- L_{g_k} = -\Delta_{g_k} + c(n) R_{g_k} = -\Delta - \bar{b}_j \partial_j - \bar{d}_{ij} \partial_{ij} + \bar{c},\qquad c(n) = \frac{(n-2)}{4(n-1)}
	\end{equation*}
	and
	\begin{equation*}
		B_{g_k} = \frac{\partial}{\partial \nu_{g_k}} + H_k,
		\qquad
		H_k(y):=\frac{n-2}{2}\varepsilon_k h_{\bar g_k}(\varepsilon_k y).
	\end{equation*}
	Here \(h_{\bar g_k}\) denotes the mean curvature of the conformally changed
	metric \(\bar g_k\) in the unscaled coordinate \(x\), whereas \(H_k\) denotes the
	scaled zeroth-order coefficient in the boundary operator for \(v_k\).
	The scaled coefficients satisfy
	\begin{equation}\label{est:bcdh}
		\begin{cases}
			\bar{b}_j(y) = \varepsilon_k b_j(\varepsilon_k y)           & = O(\varepsilon_k),     \\
			\bar{d}_{ij}(y) = d_{ij}(\varepsilon_k y)                          & = O(\varepsilon_k)|y|,  \\
			\bar{c}(y) = c(n) R_{\bar g_k}(\varepsilon_k y) \varepsilon_k^2 & = O(\varepsilon_k^3)|y|,  \\
			H_k(y)  & = O(\varepsilon_k^2)|y|,
			\\ \nabla H_k(y) &=O(\varepsilon_k^2).
		\end{cases}
	\end{equation}
	Here and below, the subscript \(k\) is suppressed when no confusion can arise.
	\subsubsection{Step 2: Setting up the moving spheres framework}

	For \(\lambda > 0\) and any function \(v\), define its Kelvin transform with respect to \(\partial \mathcal{B}_\lambda\) by
	\begin{equation*}
		v^\lambda(y) := \left( \frac{\lambda}{|y|} \right)^{n-2} v(y^\lambda), \quad y^\lambda := \frac{\lambda^2 y}{|y|^2}.
	\end{equation*}
	The moving-spheres argument is carried out on the annular region
	\begin{equation*}
		\Sigma_\lambda := \mathcal{B}^+_{R_k} \setminus \overline{\mathcal{B}^+_{\lambda}} = \{y \mid \lambda < |y| < R_k,\ y_n \ge 0\},\quad R_k = \delta \varepsilon_k^{-1}.
	\end{equation*}
	Set
	\begin{equation*}
		w_\lambda(y) := v_k(y) - v_k^\lambda(y), \qquad y \in \Sigma_\lambda \setminus\{S_k\}.
	\end{equation*}
	A straightforward computation yields
	\begin{align*}
		-\Delta_{g_k} w_\lambda + \bar{c} w_\lambda = -E_\lambda(y)                                                                              & \quad \text{in } \Sigma_\lambda\setminus\{S_k\}, \\
		\left( \frac{\partial}{\partial\nu_{g_k}} + H_k(y) - \frac{n}{n-2} \xi_\lambda^{\frac{2}{n-2}} \right) w_\lambda(y) = E_\lambda^\partial(y) & \quad \text{on } \partial'\Sigma_\lambda,
	\end{align*}
	where \(\xi_\lambda\) is determined by
	\[
	v_k^{\frac n{n-2}}-(v_k^\lambda)^{\frac n{n-2}}
	=\frac n{n-2}\xi_\lambda^{\frac2{n-2}}(v_k-v_k^\lambda).
	\]
	and the error terms \(E_\lambda\) and \(E_\lambda^\partial\) are given by
	\begin{align*}
		E_{\lambda}(y)               & := \bigl(\bar{c}(y) v_k^\lambda(y) - \bigl(\dfrac{\lambda}{|y|}\bigr)^{n+2} \bar{c}(y^{\lambda}) v_{k}(y^{\lambda})\bigr) - (\bar{b}_j \partial_j v_k^\lambda + \bar{d}_{ij} \partial_{ij} v_k^\lambda)\\
		& \quad + \bigl(\dfrac{\lambda}{|y|}\bigr)^{n+2} \bigl( \bar{b}_j (y^{\lambda}) \partial_j v_k(y^{\lambda}) + \bar{d}_{ij}(y^{\lambda}) \partial_{ij} v_k(y^{\lambda}) \bigr),                                   \\
		E_\lambda^\partial(y) & :=- \left[ H_k(y) v_k^\lambda(y) - \left(\frac{\lambda}{|y|}\right)^n H_k(y^\lambda) v_k(y^\lambda) \right].
	\end{align*}

	The following estimates for \(E_\lambda\) and \(E_\lambda^\partial\) are needed for the auxiliary-function construction. To quantify the closeness of \(v_k\) to \(V\), define
	\begin{equation*}
		\sigma_k := \|v_k-V\|_{C^2(B_{R_0}^+)} \to 0 \quad \text{as } k\to\infty.
	\end{equation*}

	\begin{prop}\label{prop:El-estimate}
		Fix \(\Lambda_0>0\). For \(0<\lambda\le\Lambda_0\) and
		\(y \in \Sigma_\lambda\), the error terms satisfy the bounds
		\begin{equation}\label{est:E@l}
			|E_\lambda(y)| \le C \left( \varepsilon_k |y|^{1-n} + \varepsilon_k^3 |y|^{3-n} + \sigma_k \varepsilon_k |y|^{-n-1} \right)
		\end{equation}
		and
		\begin{equation}\label{est:Ebd-refined-routeB}
			|E_\lambda^\partial(y)|\le C\varepsilon_k^2\left(1-\frac{\lambda^2}{|y|^2}\right)|y|^{3-n}.
		\end{equation}
		where \(C>0\) is independent of \(y\) and \(k\), and \(\sigma_k := \|v_k - V\|_{C^2(B_{R_0}^+)} \to 0\) as
		\(k \to \infty\).
	\end{prop}
	\begin{proof}
		See Appendix~\ref{app:A}.
	\end{proof}

	\subsubsection{Step 3: Construction of an auxiliary function}

	The moving-spheres comparison will be applied to \(w_\lambda=v_k-v_k^\lambda\). The equations in Step~2 give
	\begin{equation}\label{eq:w-system-routeB}
		\begin{cases}
			-L_{g_k}w_\lambda=-E_\lambda, & \text{in } \Sigma_\lambda\setminus\{S_k\},\\
			\displaystyle
			\frac{\partial w_\lambda}{\partial\nu_{g_k}}-c_\lambda(y)w_\lambda=E_\lambda^\partial,
			& \text{on } \partial'\Sigma_\lambda,
		\end{cases}
	\end{equation}
	where \(c_\lambda(y)=\frac{n}{n-2}\xi_\lambda(y)^{\frac{2}{n-2}}-H_k(y)\).

	Let \(A_\rho>1\) be a fixed constant, whose value will be chosen sufficiently large in Step~5, and set \(\rho_k=A_\rho\varepsilon_k^{-1/2}\).

	All estimates below are understood for sufficiently large \(k\). The auxiliary function is only needed in \(\Sigma_\lambda\cap\{r<\rho_k\}\). The complementary far region is excluded in Step~5. Thus it
	suffices to construct a correction term \(h_\lambda\) such that, with \(\tilde w_\lambda=w_\lambda+h_\lambda\),
	one has
	\begin{equation*}
		\begin{cases}
			-L_{g_k}\tilde w_\lambda\ge0,
			& \text{in }(\Sigma_\lambda\cap\{r<\rho_k\})\setminus\{S_k\},\\[2mm]
			\displaystyle
			\frac{\partial\tilde w_\lambda}{\partial\nu_{g_k}}
			-c_\lambda(y)\tilde w_\lambda
			\ge0,
			& \text{on }\partial'\Sigma_\lambda\cap\{r<\rho_k\}.
		\end{cases}
	\end{equation*}

	This local requirement follows if \(h_\lambda\) satisfies
	\begin{equation*}
		-L_{g_k}h_\lambda\ge |E_\lambda|
		\quad\text{in }(\Sigma_\lambda\cap\{r<\rho_k\})\setminus\{S_k\},
		\qquad
		\frac{\partial h_\lambda}{\partial\nu_{g_k}}
		-c_\lambda h_\lambda
		\ge |E_\lambda^\partial|
		\quad\text{on }\partial'\Sigma_\lambda\cap\{r<\rho_k\}.
	\end{equation*}

	The auxiliary function is chosen in the form
	\begin{equation*}
		h_\lambda(y)=h_1(r)+y_nh_2(r),
	\end{equation*}
	where
	\begin{equation*}
		h_2(r)=-A_0\varepsilon_kr^{2-n}(1-\frac{\lambda^2}{r^2}),
	\end{equation*}
	and
	\begin{equation*}
		h_1(r)
		=-A_2\varepsilon_kF_n(r,\lambda)
		-A_3\sigma_k\varepsilon_kG_n(r,\lambda).
	\end{equation*}
	Here \(A_0,A_2,A_3>0\) are constants to be fixed in the order specified below and the radial kernels are
	\begin{equation*}
		F_n(r,\lambda)=\int_\lambda^r s^{1-n}(s-\lambda)\,ds,
		\qquad
		G_n(r,\lambda)=\int_\lambda^r s^{-n}\left(\frac{s}{\lambda}-1\right)\,ds.
	\end{equation*}
	Specificly, in dimensions \(n=3,4\),
	\begin{align*}
		F_3(r,\lambda) &= \ln\frac{r}{\lambda} + \frac{\lambda}{r} - 1,\qquad & G_3(r,\lambda) &= \frac{(\lambda - r)^2}{2\lambda^2 r^2}, \\
		F_4(r, \lambda) &= \frac{1}{\lambda} - \frac{1}{r} + \frac{\lambda}{2}(r^{-2} - \lambda^{-2}),\quad & G_4(r, \lambda) &= -\frac{1}{2\lambda}(r^{-2} - \lambda^{-2}) + \frac{1}{3}(r^{-3} - \lambda^{-3}).
	\end{align*}

	For a radial function \(f=f(r)\), \(\Delta f=f''(r)+(n-1)r^{-1}f'(r)=r^{1-n}(r^{n-1}f'(r))'\). Consequently,
	\begin{equation*}
		\Delta F_n(r,\lambda)=r^{1-n},
		\qquad
		\Delta G_n(r,\lambda)=r^{-n-1}.
	\end{equation*}
	Moreover, \(F_n(\lambda,\lambda)=G_n(\lambda,\lambda)=0\), \(h_1\le0\), and \(h_2(\lambda)=0\). Consequently,
	\begin{equation}\label{eq:h-inner-zero}
		h_\lambda=0
		\quad\text{on }r=\lambda.
	\end{equation}
	Since \(w_\lambda=0\) on \(r=\lambda\), also
	\[
	\tilde w_\lambda:=w_\lambda+h_\lambda=0
	\quad\text{on }r=\lambda.
	\]

	In the selected region \(\lambda<r<\rho_k\),
	\[
	-\Delta h_1
	=A_2\varepsilon_kr^{1-n}
	+A_3\sigma_k\varepsilon_kr^{-n-1}.
	\]
	Moreover,
	\[
	y_nh_2
	=-A_0\varepsilon_ky_nr^{2-n}
	+A_0\varepsilon_k\lambda^2y_nr^{-n}.
	\]
	Since \(y_nr^{-n}\) is harmonic in \(\mathbb R^n_+\),
	\[
	\left|-\Delta(y_nh_2)\right|
	\le CA_0\varepsilon_kr^{1-n}.
	\]
	The constants are fixed in the following order. First choose \(A_0>0\) large enough for the boundary estimate below. Then choose \(A_2\gg A_0+1\) and \(A_3\gg1\), and finally take \(k\) sufficiently large after \(A_\rho\) has been fixed. The lower-order terms in \(-L_{g_k}h_\lambda\) are controlled by \eqref{est:bcdh}. Since
	\(\varepsilon_kr\le\varepsilon_k\rho_k=A_\rho\varepsilon_k^{1/2}=o(1)\), these terms are absorbed by the two leading terms above. Hence
	\begin{equation}\label{eq:H-interior-final}
		-L_{g_k}h_\lambda\ge |E_\lambda|
		\quad\text{in }(\Sigma_\lambda\cap\{r<\rho_k\})\setminus\{S_k\}.
	\end{equation}

	On the flat boundary,
	\[
	\frac{\partial h_\lambda}{\partial\nu_{g_k}}-c_\lambda h_\lambda
	=-h_2(r)-c_\lambda h_1(r).
	\]
	In the selected region, \(c_\lambda\ge0\) for all large \(k\). Indeed, for \(r\le\rho_k=A_\rho\varepsilon_k^{-1/2}\) and \(\lambda\in[\lambda_0,\Lambda_0]\), the coarse lower bound for \(v_k\) and the local lower bound for \(v_k(y^\lambda)\) imply
	\[
	\xi_\lambda(y)^{\frac2{n-2}}\ge c\varepsilon_k.
	\]
	On the other hand,
	\[
	|H_k(y)|\le C\varepsilon_k^2r\le CA_\rho\varepsilon_k^{3/2}=o(\varepsilon_k).
	\]
	Thus \(c_\lambda\ge c\varepsilon_k/2\ge0\) for all large \(k\). Since \(h_1\le0\), \(-c_\lambda h_1\ge0\), and therefore
	\[
	\frac{\partial h_\lambda}{\partial\nu_{g_k}}-c_\lambda h_\lambda
	\ge
	A_0\varepsilon_kr^{2-n}\left(1-\frac{\lambda^2}{r^2}\right).
	\]
	Using \eqref{est:Ebd-refined-routeB}, in the selected region \(r<\rho_k\),
	\[
	|E_\lambda^\partial|
	\le C\varepsilon_k^2r^{3-n}\left(1-\frac{\lambda^2}{r^2}\right)
	=C(\varepsilon_kr)\varepsilon_kr^{2-n}\left(1-\frac{\lambda^2}{r^2}\right)
	\le C A_\rho\varepsilon_k^{1/2}\varepsilon_kr^{2-n}
	\left(1-\frac{\lambda^2}{r^2}\right).
	\]
	After the previously chosen \(A_0\) is fixed, the last coefficient is smaller than \(A_0\) for all sufficiently large \(k\). Therefore
	\begin{equation}\label{eq:H-boundary-final}
		\frac{\partial h_\lambda}{\partial\nu_{g_k}}-c_\lambda h_\lambda
		\ge |E_\lambda^\partial|
		\quad\text{on }\partial'\Sigma_\lambda\cap\{r<\rho_k\}.
	\end{equation}
	Combining \eqref{eq:w-system-routeB}, \eqref{eq:H-interior-final}, and \eqref{eq:H-boundary-final} yields
	\begin{equation*}
		-L_{g_k}\tilde w_\lambda\ge0
		\quad\text{in }(\Sigma_\lambda\cap\{r<\rho_k\})\setminus\{S_k\},
	\end{equation*}
	and
	\begin{equation*}
		\frac{\partial\tilde w_\lambda}{\partial\nu_{g_k}}
		-c_\lambda\tilde w_\lambda\ge0
		\quad\text{on }\partial'\Sigma_\lambda\cap\{r<\rho_k\}.
	\end{equation*}
	This completes the construction of the auxiliary function.

	\subsubsection{Step 4: Initialization of the moving spheres method}

	For the limiting bubble, a direct computation gives for $|y| > \lambda$.
	\begin{equation}\label{eq:restriction}
		V(y) - V^\lambda(y) \; \begin{cases}
			> 0, & \text{if } \lambda < n-2, \\
			= 0, & \text{if } \lambda = n-2, \\
			< 0, & \text{if } \lambda > n-2,
		\end{cases}
	\end{equation}

	The moving-spheres method is then applied to derive a contradiction.

	\begin{claim}
		There exist constants
		\[
		0<\lambda_0<\lambda_1<n-2
		\]
		such that, for all sufficiently large \(k\),
		\begin{equation}\label{eq:startup-final}
			\tilde w_\lambda(y)
			=v_k(y)-v_k^\lambda(y)+h_\lambda(y)
			\ge0
			\quad\text{in }\Sigma_\lambda\setminus\{S_k\}
		\end{equation}
		for every \(\lambda\in[\lambda_0,\lambda_1]\).
	\end{claim}

	\begin{proof}
		Choose once and for all \(0<\lambda_0<\lambda_1<n-2\). For each
		\(\lambda\in[\lambda_0,\lambda_1]\), decompose \(\Sigma_\lambda\) into
		\[
		\mathcal N_\eta=\{\lambda<|y|<\lambda+\eta\},
		\qquad
		\mathcal M_{\eta,R_0}=\{\lambda+\eta\le |y|\le R_0\},
		\qquad
		\mathcal F_{R_0,R_k}=\{R_0\le |y|\le R_k\},
		\]
		where  \(R_0\gg1\) is fixed and \(R_k=\delta\varepsilon_k^{-1}\) for \(k\) sufficiently large.

		\medskip\noindent\textbf{Middle region \(\lambda+\eta\le |y|\le R_0\).}

		By \eqref{eq:restriction}, on the compact set \(\mathcal M_{\eta,R_0}\), the minimum of \(V-V^\lambda\) is positive. Denote it by
		\[
		c_{\lambda,\eta,R_0}
		=\min_{\mathcal M_{\eta,R_0}}(V-V^\lambda)>0.
		\]
		Since \(v_k\to V\) in \(C^2_{\rm loc}(\overline{\mathbb R^n_+})\),
		\[
		\|v_k-V\|_{C^0(\mathcal M_{\eta,R_0})}
		+\|v_k^\lambda-V^\lambda\|_{C^0(\mathcal M_{\eta,R_0})}
		\to0.
		\]
		Moreover \(h_\lambda=O(\varepsilon_k)+O(\sigma_k\varepsilon_k)\) uniformly on \(\mathcal M_{\eta,R_0}\). Hence, for sufficiently large \(k\),
		\begin{equation}\label{eq:middle-positive}
			\tilde w_\lambda
			=v_k-v_k^\lambda+h_\lambda
			\ge \frac12c_{\lambda,\eta,R_0}>0
			\quad\text{on }\mathcal M_{\eta,R_0}.
		\end{equation}

		\medskip\noindent\textbf{Narrow region \(\lambda<|y|<\lambda+
			\eta\).}

		Since \(\rho_k\to\infty\), for sufficiently large \(k\), the narrow region \(\mathcal N_\eta\) is contained in \(\{r<\rho_k\}\).  Hence the construction of \(h_\lambda\) gives
		\[
		-L_{g_k}\tilde w_\lambda\ge0,
		\qquad
		\frac{\partial\tilde w_\lambda}{\partial\nu_{g_k}}
		-c_\lambda\tilde w_\lambda\ge0
		\quad\text{on }\partial'\mathcal N_\eta.
		\]

		Set \(\tilde{w}_\lambda^- = \max\{0, -\tilde{w}_\lambda\}\). The function \(\tilde w_\lambda^-\) vanishes on the inner boundary \(|y|=\lambda\), since \(v_k=v_k^\lambda\) and \(h_\lambda=0\) by \eqref{eq:h-inner-zero}. It also vanishes on \(|y|=\lambda+\eta\) by \eqref{eq:middle-positive}. Multiplying the interior inequality by \(\tilde{w}_\lambda^-\) and integrating over \(\mathcal{N}_\eta\) yields
		\begin{align*}
			0
			&\le \int_{\mathcal N_\eta}
			(-L_{g_k}\tilde w_\lambda)
			\tilde w_\lambda^-\,dV_{g_k}\\
			&=-\int_{\mathcal N_\eta}
			|\nabla_{g_k}\tilde w_\lambda^-|_{g_k}^2\,dV_{g_k}
			-\int_{\partial'\mathcal N_\eta}
			\frac{\partial\tilde w_\lambda}{\partial\nu_{g_k}}
			\tilde w_\lambda^-\,d\sigma_{g_k}
			+O(\varepsilon_k)
			\int_{\mathcal N_\eta}(\tilde w_\lambda^-)^2\,dV_{g_k}.
		\end{align*}

		The boundary inequality gives, on the support of \(\tilde w_\lambda^-\),
		\[
		-\frac{\partial\tilde w_\lambda}{\partial\nu_{g_k}}
		\tilde w_\lambda^-
		\le \|c_\lambda\|_{L^\infty(\partial'\mathcal N_\eta)}
		(\tilde w_\lambda^-)^2.
		\]
		Therefore,
		\begin{equation}\label{eq:narrow-energy}
			\int_{\mathcal N_\eta}
			|\nabla_{g_k}\tilde w_\lambda^-|_{g_k}^2\,dV_{g_k}
			\le
			C\int_{\partial'\mathcal N_\eta}
			(\tilde w_\lambda^-)^2\,d\sigma_{g_k}
			+C\varepsilon_k
			\int_{\mathcal N_\eta}(\tilde w_\lambda^-)^2\,dV_{g_k}.
		\end{equation}
		Since \(\tilde w_\lambda^-\) vanishes on the spherical sides of a domain of width \(\eta\), the trace--Poincar\'e inequalities imply
		\begin{align*}
			\int_{\partial'\mathcal N_\eta}
			(\tilde w_\lambda^-)^2\,d\sigma_{g_k}
			&\le C\eta
			\int_{\mathcal N_\eta}
			|\nabla_{g_k}\tilde w_\lambda^-|_{g_k}^2\,dV_{g_k},\\
			\int_{\mathcal N_\eta}
			(\tilde w_\lambda^-)^2\,dV_{g_k}
			&\le C\eta^2
			\int_{\mathcal N_\eta}
			|\nabla_{g_k}\tilde w_\lambda^-|_{g_k}^2\,dV_{g_k}.
		\end{align*}
		Substitution into \eqref{eq:narrow-energy} gives
		\[
		\int_{\mathcal N_\eta}
		|\nabla_{g_k}\tilde w_\lambda^-|_{g_k}^2\,dV_{g_k}
		\le C(\eta+\varepsilon_k\eta^2)
		\int_{\mathcal N_\eta}
		|\nabla_{g_k}\tilde w_\lambda^-|_{g_k}^2\,dV_{g_k}.
		\]
		Choosing \(\eta>0\) so that \(C\eta<1/2\), and then taking \(k\) sufficiently large to obtain
		\[
		\nabla\tilde w_\lambda^-
		\equiv0,
		\qquad
		\tilde w_\lambda^-
		\equiv0
		\quad\text{in }\mathcal N_\eta.
		\]
		Thus \(\tilde w_\lambda\ge0\) in the narrow region.

		\medskip\noindent\textbf{Far region \(R_0\le |y|\le R_k\).}

		It remains to treat \(R_0\le r\le R_k\). The function $v_k^{\lambda}(y)$ satisfies
		\begin{equation*}
			v_k^{\lambda_0}(y)
			\le (1-2\varepsilon_0)r^{2-n}
			\quad\text{for }r\ge R_0,
		\end{equation*}
		for a fixed \(\varepsilon_0>0\), after choosing \(R_0\) large and then \(k\) large. This follows from local convergence at the reflected point \(|y^{\lambda_0}|\le\lambda_0^2/R_0\) and the strict far-field gap of the limiting bubble for \(\lambda_0<n-1\).

		The barrier function \(\phi\) is constructed to satisfy \(v_k\ge\phi\) in \(R_0\le r\le R_k\). If \(n=3\), set
		\begin{align*}
			\phi(y)
			&=(1-\varepsilon_0)r^{-1}
			+A_1\varepsilon_k\log\frac r{R_0}
			+A_2\varepsilon_k^4 r^3\log\frac{r}{R_0}\notag\\
			&\quad+B_1\varepsilon_kr^{-1}y_n
			+B_2\varepsilon_k^3r\log\frac{r}{R_0}\,y_n.
		\end{align*}
		If \(n=4\), choose a constant \(D>0\) and set
		\begin{align*}
			\phi(y)
			&=(1-\varepsilon_0)r^{-2}
			+D\varepsilon_k
			+A_1\varepsilon_k(R_0^{-1}-r^{-1})
			+A_2\varepsilon_k^4(r^3-R_0^3)\notag\\
			&\quad+B_1\varepsilon_kr^{-2}y_n+B_2\varepsilon_k^3r\,y_n.
		\end{align*}
		The constants are chosen so that
		\[
		-L_{g_k}\phi<0
		\quad\text{in }R_0<r<R_k,
		\]
		and
		\begin{equation}\label{eq:phi-boundary-strong-final}
			B_{g_k}\phi\le -2|\bar h_{g_k}|\phi
			\quad\text{on }\partial'\{R_0<r<R_k\}.
		\end{equation}
		The verification is as follows. The leading Euclidean computations are
		\[
		-\Delta\left(\varepsilon_k\log\frac r{R_0}\right)=-\varepsilon_kr^{-2}
		\quad(n=3),
		\]
		\[
		-\Delta\left(\varepsilon_k(R_0^{-1}-r^{-1})\right)=-\varepsilon_kr^{-3}
		\quad(n=4),
		\]
		while
		\[
		-\Delta\left(\varepsilon_k^4 r^3\log\frac{r}{R_0}\right)
		\le -c\varepsilon_k^4 r\left(1+\log\frac{r}{R_0}\right)
		\quad(n=3),
		\]
		\[
		-\Delta(\varepsilon_k^4r^3)\le -c\varepsilon_k^4r
		\quad(n=4).
		\]
		Choose the constants  as
		\[
		B_1>A_0,\qquad
		A_1\gg B_1+1,\qquad
		A_2\gg A_1+B_2+D+1,
		\]
		with the term \(D\) omitted in dimension \(3\), and then take \(\delta>0\) sufficiently small. The negative terms in \(L_{g_k}\phi\) dominate the error contributions from \(\bar b_j,\bar d_{ij}\), and \(\bar c\). On the flat boundary,
		\[
		-\partial_{y_n}(B_1\varepsilon_kr^{2-n}y_n)
		=-B_1\varepsilon_kr^{2-n},
		\]
		and the additional terms
		\[
		-\partial_{y_n}\left(B_2\varepsilon_k^3r\log\frac{r}{R_0}y_n\right)
		=-B_2\varepsilon_k^3r\log\frac{r}{R_0}
		\quad(n=3),
		\]
		\[
		-\partial_{y_n}(B_2\varepsilon_k^3r y_n)
		=-B_2\varepsilon_k^3r
		\quad(n=4)
		\]
		dominate the mean-curvature contributions generated by the far-field gates. This gives \eqref{eq:phi-boundary-strong-final}.

		On \(r=R_0\), local convergence gives \(v_k\ge\phi\). On \(r=R_k\), the lower bound gives \(v_k\ge\phi\) for large \(k\): in dimension \(3\),
		\[
		\phi(R_k)=O(\varepsilon_k\log\varepsilon_k^{-1})
		+O(\delta^3\varepsilon_k\log\varepsilon_k^{-1})
		+O(\delta^2\varepsilon_k\log\varepsilon_k^{-1})
		=o(\varepsilon_k^{1/2}),
		\]
		while \(v_k\ge c\varepsilon_k^{1/2}\); in dimension \(4\), the non-removable singularity gives, after shrinking \(\delta\),
		\[
		v_k\ge L\varepsilon_k
		\quad\text{on }r=R_k,
		\]
		with \(L\) larger than the fixed coefficient of \(\varepsilon_k\) in \(\phi(R_k)\). Near the possible puncture \(S_k\), one has \(v_k\to+\infty\), while \(\phi\) is finite; hence \(v_k-\phi>0\) on a small sphere around \(S_k\). Applying the maximum principle to \(v_k-\phi\) in the punctured annulus and then letting the small radius around \(S_k\) tend to zero, using \eqref{eq:phi-boundary-strong-final} to exclude a negative flat-boundary minimum, yields
		\begin{equation}\label{eq:v-ge-phi-startup-final}
			v_k\ge\phi
			\quad\text{for }R_0\le r\le R_k.
		\end{equation}
		It remains to compare \(\phi+h_\lambda\) with \(v_k^\lambda\). If \(n=3\), then
		\[
		h_1(r)\ge -C_H\varepsilon_k\log\frac r{R_0}-C_H\varepsilon_k-C_H\sigma_k\varepsilon_k,
		\qquad
		y_nh_2(r)\ge -A_0\varepsilon_kr^{-1}y_n.
		\]
		Choosing \(A_1\gg C_H\), \(B_1>A_0\), and then taking \(k\) large gives
		\[
		\phi+h_\lambda-v_k^\lambda\ge0
		\quad\text{for }R_0\le r\le R_k,
		\qquad \lambda\in[\lambda_0,\lambda_1].
		\]
		If \(n=4\), then
		\[
		h_1(r)\ge -C_H\varepsilon_k-C_H\sigma_k\varepsilon_k,
		\qquad
		y_nh_2(r)\ge -A_0\varepsilon_kr^{-2}y_n.
		\]
		Choosing \(D>4C_H\), \(B_1>A_0\), and then taking \(k\) large gives the same inequality. Combining this with \eqref{eq:v-ge-phi-startup-final} gives \eqref{eq:startup-final}.
	\end{proof}
	\subsubsection{Step 5: Continuation of the moving spheres}

	Fix an arbitrary \(\Lambda_0>n-1\) and set
	\[
	\bar\lambda = \sup\left\{ \mu\in[\lambda_0,\Lambda_0]:
	\tilde w_\lambda\ge0\text{ in }\Sigma_\lambda
	\text{ for every }\lambda\in[\lambda_0,\mu] \right\}.
	\]
	By Step~4, the set is nonempty and \(\bar\lambda\ge\lambda_1\).

	We first claim that no first contact can occur in the far region. Recall that \(\rho_k=A_\rho\varepsilon_k^{-1/2}\).
	For \(\lambda\in[\lambda_0,\Lambda_0]\) and \(r\ge\rho_k\),
	\[
	v_k^\lambda(y)\le C_{\Lambda_0}r^{2-n}.
	\]
	Moreover,
	\[
	|h_\lambda(y)|
	\le
	C_H\varepsilon_k\log r+C_H\sigma_k\varepsilon_k
	\quad(n=3),
	\]
	and
	\[
	|h_\lambda(y)|
	\le
	C_H\varepsilon_k+C_H\sigma_k\varepsilon_k
	\quad(n=4).
	\]
	If \(n=3\), the coarse lower bound gives \(v_k\ge c\varepsilon_k^{1/2}\) for
	\(r\le R_k\), while
	\[
	v_k^\lambda\le C_{\Lambda_0}A_\rho^{-1}\varepsilon_k^{1/2},
	\qquad
	|h_\lambda|=o(\varepsilon_k^{1/2})
	\quad\text{for }r\ge\rho_k.
	\]
	If \(n=4\), after \(\delta\) is reduced, the non-removable singularity gives \(v_k\ge L\varepsilon_k\) for \(r\le R_k\), with \(L>4C_H\), and
	\[
	v_k^\lambda\le C_{\Lambda_0}A_\rho^{-2}\varepsilon_k.
	\]
	Thus, after choosing \(A_\rho\) large and then \(k\) large, one obtains
	\begin{equation}\label{eq:far-exclusion-final}
		\tilde w_\lambda>0
		\quad\text{for }\rho_k\le r\le R_k,
		\qquad
		\lambda\in[\lambda_0,\Lambda_0].
	\end{equation}

	We now prove that \(\bar\lambda=\Lambda_0\). Suppose, to the contrary, that \(\bar\lambda<\Lambda_0\). By continuity, \(\tilde
	w_{\bar\lambda}\ge0\) in \(\Sigma_{\bar\lambda}\). By \eqref{eq:far-exclusion-final}, any zero minimum away from
	the inner sphere must lie in the selected region \(\bar\lambda<r<\rho_k\). If the puncture \(S_k\) belongs to this
	region, a small half-ball around \(S_k\) is removed. Near \(S_k\), \(v_k\to+\infty\), and the functions \(v_k^{\bar\lambda}\), \(h_{\bar\lambda}\) remain bounded, which implies the comparison is strict on the added boundary. Hence
	the puncture does not create a contact point.

	It remains to exclude a contact point in the remaining selected region, we can directly apply the inequalities in Step~3.

	Set \(z=e^{-\alpha y_1}\tilde w_{\bar\lambda}\). Then
	\[
	-L_{g_k}\tilde w_{\bar\lambda}
	=e^{\alpha y_1}\left[-a^{ij}_k z_{ij}-(\bar b^i_k+2\alpha a^{i1}_k)z_i
	+(\bar c_k-\alpha \bar b^1_k-\alpha^2 a^{11}_k)z\right].
	\]
	To make \(\bar c_k-\alpha \bar b^1_k-\alpha^2 a^{11}_k\le -1\), choose \(\alpha\) large, independent of \(k\). The usual strong maximum principle applies at
	\(z\), and therefore to \(\tilde w_{\bar\lambda}\).

	The strong maximum principle and the boundary
	point lemma give \(\tilde w_{\bar\lambda}>0\) for \(\bar\lambda<r<\rho_k\), together with the required strictness
	on the inner boundary \(r=\bar\lambda\). Using the narrow-domain argument from Step~4 together with \eqref{eq:far-exclusion-final} in the far region, the comparison can be continued to all
	\(\lambda\in[\lambda_0,\bar\lambda+\varepsilon]\) for some \(\varepsilon>0\). This contradicts the definition of
	\(\bar\lambda\). Therefore \(\bar\lambda=\Lambda_0\).

	Let \(k\to\infty\), we have
	\[
	v_k(y)-v_k^{\Lambda_0}(y)+h_{\Lambda_0}(y)\ge0
	\]
	on every compact subset of \(\{|y|>\Lambda_0\}\), and using
	\(h_{\Lambda_0}\to0\), gives
	\[
	V(y)\ge V^{\Lambda_0}(y)
	\quad\text{for }|y|>\Lambda_0.
	\]
	This contradicts \eqref{eq:restriction}, since \(\Lambda_0>n-1\). The boundary upper bound follows.

\subsection{Harnack inequality and the gradient estimate}

The following scale-invariant Harnack inequality will be used below. The key feature is that the constant in this estimate depends only on the scaled boundary control, not on an a priori upper bound for the solution in the full half-annulus.

\begin{prop}[Scale-invariant Harnack inequality]\label{prop:scale-harnack}
	Let \(u\) be a positive solution to the equation \eqref{eq:Main_PDE}, where the metric \(g\) is written in conformal Fermi coordinates such that
	\begin{equation*}
		g_{ij}(x)=\delta_{ij}+O(|x|),\qquad
		\det g =1,\qquad
		R_g(x)=O(|x|),\qquad
		h_g(x)=O(|x|).
	\end{equation*}
	Assume the boundary decay upper bound
	\begin{equation}\label{eq:upper}
		u(x)\le C_0\,|x|^{\frac{2-n}{2}}\qquad \text{for } 0<|x|<\delta_0, \ x \in \partial'\mathcal{B}_1^+.
	\end{equation}
	Then there exist \(r_1>0\) and a constant \(C>0\), depending only on \(n\), the metric \(g\), and \(C_0\), such that for every \(0<r<r_1\),
	\begin{equation}\label{eq:Harnack}
		\sup_{\mathcal{A}_r} u \le C \inf_{\mathcal{A}_r} u.
	\end{equation}
	where \(\mathcal{A}_r=\{x\in\mathbb{R}^n_+: r/2 \le|x|\le 2r\}\).
\end{prop}

\begin{proof}
	Fix \(r\in(0,r_0/4)\) and consider the normalized half-annulus \(\mathcal{A}_1=\{y\in\mathbb{R}^n_+: 1/2\le|y|\le 2\}\). Define the rescaled function and metric by
	\[
	v_r(y)=r^{\frac{n-2}{2}}u(ry),\qquad g_r(y)=g(ry),\qquad y\in\mathcal{A}_1.
	\]
	Because the Yamabe equation is conformally invariant, \(v_r\) satisfies
	\begin{equation*}
		\begin{cases}
			-\Delta_{g_r} v_r + c(n) r^2 R_{g_r} v_r = 0                                             & \text{in } \mathcal{A}_1,            \\[4pt]
			\displaystyle \frac{\partial v_r}{\partial\nu_{g_r}} + \frac{n-2}{2} r h_{g_r} v_r = v_r^{\frac{n}{n-2}} & \text{on } \partial'\mathcal{A}_1.
		\end{cases}
	\end{equation*}
	By the boundary control assumption \eqref{eq:upper}, the scaled function is bounded on the flat boundary
	\[
	v_r(y)\le C_0|y|^{\frac{2-n}{2}}\le C_0\,2^{\frac{n-2}{2}}=:\Lambda \qquad \text{for } y\in\partial'\mathcal{A}_1.
	\]
	This uniform \(L^\infty\) bound allows the semilinear boundary term to be written as a linear Robin condition
	\[
	\frac{\partial v_r}{\partial\nu_{g_r}} + \frac{n-2}{2} r h_{g_r} v_r = q_r(y) v_r, \qquad \text{where } q_r(y) := v_r(y)^{\frac{2}{n-2}} \in [0, \Lambda^{\frac{2}{n-2}}].
	\]

	As \(r\to0\), the rescaled metrics \(g_r \to \delta_{ij}\) in \(C^2(\bar{\mathcal{A}_1})\), and the coefficients \(r^2R_{g_r}\) and \(rh_{g_r}\) vanish uniformly. Consequently, for all sufficiently small \(r>0\), the interior equation can be written in non-divergence form as
	\[
	\mathcal{L}_r v_r
	:=-a_r^{ij}(y)\partial_{ij}v_r
	+b_r^i(y)\partial_i v_r
	+c_r(y)v_r=0,
	\]
	where
	\[
	a_r^{ij}=g_r^{ij},
	\qquad
	|b_r^i|\le C_1,
	\qquad
	|c_r|\le C_1r^2.
	\]
	The ellipticity constants are independent of \(r\):
	\[
	\lambda_0|\xi|^2
	\le a_r^{ij}(y)\xi_i\xi_j
	\le \Lambda_0|\xi|^2
	\qquad
	\forall\xi\in\mathbb R^n.
	\]
	On the flat boundary, the boundary condition may be written as
	\[
	\mathcal M_r v_r
	:=\partial_n v_r+\beta_r^i(y)\partial_i v_r+\alpha_r(y)v_r=0,
	\]
	where \(\beta_r^i\) and \(\alpha_r\) are uniformly bounded for \(r\) on \(\partial'\mathcal A_1\).

	The domain \(\mathcal{A}_1\) is a bounded Lipschitz domain, and the boundary operator \(\mathcal{M}_r\) satisfies a uniform obliqueness condition because \(\partial_n\) is the outward normal direction and the geometric perturbation terms are fully controlled. This linear system fits into the framework of mixed boundary value problems studied in \cite[Chapter~1]{liebermanObliqueDerivativeProblems2012}. In particular, the coefficients satisfy hypotheses (1.20) and (1.36) of \cite{liebermanObliqueDerivativeProblems2012} with constants independent of \(r\).

	Because \(v_r\ge0\) and the equation has no inhomogeneous terms (i.e., \(f=G=0\) in the notation of \cite[Theorem 1.20]{liebermanObliqueDerivativeProblems2012}), the boundary weak Harnack inequality \cite[Theorem 1.20]{liebermanObliqueDerivativeProblems2012} together with the local maximum principle \cite[Theorem 1.27]{liebermanObliqueDerivativeProblems2012} can be applied directly to the domain \(\mathcal{A}_1\). A standard finite covering argument extends these local estimates to the entire annulus, yielding the full Harnack inequality
	\[
	\sup_{\mathcal{A}_1} v_r \le \tilde C\, \inf_{\mathcal{A}_1} v_r,
	\]
	where \(\tilde C\) depends only on \(n\), the ellipticity constants, the \(L^\infty\)-bounds of the coefficients (including the frozen nonlinear term \(\Lambda^{\frac{2}{n-2}}\)), and the obliqueness parameters, all of which are independent of \(r\).

	Scaling back to the original variables, with \(u(x)=r^{\frac{2-n}{2}}v_r(x/r)\), we obtain
	\[
	\sup_{\mathcal{A}_r} u \le \tilde C \inf_{\mathcal{A}_r} u.\qedhere
	\]
\end{proof}

\begin{rem}
	Once the boundary decay estimate \eqref{eq:upper_bd} is definitively established (as proven in Theorem~\ref{thm:upper_bound}), the scale-invariant Harnack inequality \eqref{eq:Harnack} holds everywhere in the punctured half-ball near the origin.
\end{rem}

\begin{cor}[Scale-invariant gradient estimates]\label{cor:gradient}
	Let \(u\) be a positive solution to the equation \eqref{eq:Main_PDE}. Then there exists a constant \(C_1 > 0\) such that for all \(0 < |x| < r_0/2\),
	\begin{equation*}
		|\nabla u(x)| \le C_1 |x|^{-1} u(x), \qquad |\nabla^2 u(x)| \le C_1 |x|^{-2} u(x).
	\end{equation*}
\end{cor}

\begin{proof}
	Fix \(x\in B_{r_0/2}^+\setminus\{0\}\), and set \(r=|x|\). Define
	\[
	v_r(y)=r^{\frac{n-2}{2}}u(ry),
	\qquad
	y\in \mathcal A_1:=\{y\in\mathbb R_+^n:1/2\le |y|\le2\}.
	\]
	As in Proposition~\ref{prop:scale-harnack}, \(v_r\) satisfies a uniformly elliptic equation in \(\mathcal A_1\) with a uniformly oblique boundary condition on \(\partial'\mathcal A_1\). The constants in these estimates are independent of \(r\).

	Choose nested half-annuli
	\[
	\mathcal A_1^{(0)}=B_{3/2}^+\setminus B_{2/3}^+,
	\qquad
	\mathcal A_1^{(1)}=B_{5/4}^+\setminus B_{3/4}^+.
	\]
	Boundary and interior Schauder estimates applied to the linearized equation give
	\begin{equation}\label{eq:scaled-C2-estimate}
		\|v_r\|_{C^2(\mathcal A_1^{(1)})}
		\le C\|v_r\|_{L^\infty(\mathcal A_1^{(0)})}.
	\end{equation}
	The Harnack inequality \eqref{eq:Harnack} gives, for every \(y\in\mathcal A_1^{(1)}\),
	\begin{equation}\label{eq:scaled-harnack-pointwise}
		\|v_r\|_{L^\infty(\mathcal A_1^{(0)})}
		\le C\inf_{\mathcal A_1^{(0)}}v_r
		\le C v_r(y).
	\end{equation}
	Combining \eqref{eq:scaled-C2-estimate} and \eqref{eq:scaled-harnack-pointwise},
	\begin{equation}\label{eq:scaled-gradient-hessian}
		|\nabla v_r(y)|\le C v_r(y),
		\qquad
		|\nabla^2v_r(y)|\le C v_r(y),
		\qquad
		y\in\mathcal A_1^{(1)}.
	\end{equation}

	Evaluating \eqref{eq:scaled-gradient-hessian} at
	\[
	y=\frac{x}{r},
	\qquad |y|=1,
	\]
	the scaling identities are
	\begin{align*}
		v_r(y)&=r^{\frac{n-2}{2}}u(x),\\
		|\nabla v_r(y)|&=r^{\frac n2}|\nabla u(x)|,\\
		|\nabla^2v_r(y)|&=r^{\frac{n+2}{2}}|\nabla^2u(x)|.
	\end{align*}
	Therefore \eqref{eq:scaled-gradient-hessian} implies
	\[
	r^{\frac n2}|\nabla u(x)|
	\le C r^{\frac{n-2}{2}}u(x),
	\qquad
	r^{\frac{n+2}{2}}|\nabla^2u(x)|
	\le C r^{\frac{n-2}{2}}u(x).
	\]
	Since \(r=|x|\),
	\[
	|\nabla u(x)|\le C|x|^{-1}u(x),
	\qquad
	|\nabla^2u(x)|\le C|x|^{-2}u(x).\qedhere
	\]
\end{proof}

\subsection{The inner case}
\begin{prop}\label{prop:interior_control}
	Let \(\tilde{u} \in C^2(\mathcal{B}_{\delta}^+ \setminus \{0\})\) be a positive solution to \(-\Delta_{\tilde{g}} \tilde{u} = 0\) in \(\mathcal{B}_{\delta}^+ \setminus \{0\}\) with the nonlinear boundary condition \(\frac{\partial \tilde{u}}{\partial \nu_{\tilde{g}}} = \tilde{u}^{\frac{n}{n-2}}\) on \(\partial' \mathcal{B}_{\delta}^+ \setminus \{0\}\) for some $\delta>0$.
	If there exists a constant \(C_0 > 0\) such that the flat boundary decay estimate holds:
	\begin{equation*} \tilde{u}(y) \le C_0 |y|^{\frac{2-n}{2}} \quad \text{for all } y \in \partial' \mathcal{B}_{\delta}^+ \setminus \{0\}, \end{equation*}
	then the interior decay estimate holds with the same rate:
	\begin{equation*} \tilde{u}(x) \le C_1 |x|^{\frac{2-n}{2}} \quad \text{for all } x \in \mathcal{B}_{\delta/2}^+ \setminus \{0\}, \end{equation*}
	where \(C_1 > 0\) is a constant independent of \(x\).
\end{prop}

\begin{proof}
	For any interior point \(x \in \mathcal{B}_{\delta/2}^+ \setminus \{0\}\), let \(r = |x|\). Define the rescaled function
	\begin{equation*} v_r(y) = r^{\frac{n-2}{2}} \tilde{u}(ry) \quad \text{for } y \in \mathcal{A} := \mathcal{B}_2^+ \setminus \mathcal{B}_{1/2}^+. \end{equation*}
	Since \(\tilde{u}>0\), the scaled function satisfies \(v_r\ge0\) in \(\mathcal{A}\). It satisfies \(-\Delta_{g_r} v_r = 0\) in \(\mathcal{A}\), where \(g_r(y) = \tilde{g}(ry) \to \delta_{ij}\) uniformly in \(C^2\) as \(r \to 0\).

	On the flat boundary \(\partial' \mathcal{A}\), the nonlinear boundary condition scales to:
	\begin{equation*} \frac{\partial v_r}{\partial \nu_{g_r}}(y) = v_r(y)^{\frac{n}{n-2}} = a_r(y) v_r(y), \quad \text{where } a_r(y) = v_r(y)^{\frac{2}{n-2}}. \end{equation*}
	By the flat boundary hypothesis, \(\tilde{u}(ry) \le C_0 |ry|^{\frac{2-n}{2}}\), which implies that for all \(y \in \partial' \mathcal{A}\) (\(1/2 \le |y| \le 2\)):
	\begin{equation*} v_r(y) \le r^{\frac{n-2}{2}} C_0 |ry|^{\frac{2-n}{2}} = C_0 |y|^{\frac{2-n}{2}} \le C_0 2^{\frac{n-2}{2}} := C_2. \end{equation*}
	Thus, the coefficient of the linear Robin boundary condition is uniformly bounded: \(\|a_r\|_{L^\infty(\partial' \mathcal{A})} \le C_2^{\frac{2}{n-2}}\).

	Consider the strictly interior half-annulus \(\mathcal{A}' = \mathcal{B}_{3/2}^+ \setminus \mathcal{B}_{2/3}^+\). The uniform Harnack inequality \eqref{eq:Harnack} gives
	\begin{equation*} \sup_{\mathcal{A}'} v_r \le C \inf_{\mathcal{A}'} v_r. \end{equation*}

	Since \(v_r\) is continuous up to the flat boundary \(\partial' \mathcal{A}'\), its infimum over the domain \(\mathcal{A}'\) is bounded by its infimum on the flat boundary \(\partial' \mathcal{A}'\). Furthermore, any value on the flat boundary is bounded by the uniform upper bound \(C_2\). Therefore
	\begin{equation*} \inf_{\mathcal{A}'} v_r \le \inf_{\partial' \mathcal{A}'} v_r \le \sup_{\partial' \mathcal{A}'} v_r \le C_2. \end{equation*}

	Combining these inequalities gives a uniform \(L^\infty\) bound for \(v_r\) in \(\mathcal{A}'\):
	\begin{equation*} \sup_{\mathcal{A}'} v_r \le C \cdot C_2 := C_1. \end{equation*}
	Evaluating this bound at \(\hat{x} = x/r \in \mathcal{A}'\) (since \(|\hat{x}| = 1\)), gives
	\begin{equation*} v_r(\hat{x}) = r^{\frac{n-2}{2}} \tilde{u}(x) \le C_1, \end{equation*}
	which yields \(\tilde{u}(x) \le C_1 |x|^{\frac{2-n}{2}}\).
\end{proof}
The local simplification needed above follows from the next two lemmas.

\begin{lem}\label{lem:prescribed_hg}
	For any smooth metric \(\tilde g\) on \(B_1^+\) and any constant \(M>0\), there exists a positive smooth function \(w\) such that \(\hat g=w^2\tilde g\) satisfies
	\[
	h_{\hat g}=M
	\quad\text{on }\partial'B_1^+.
	\]
\end{lem}
\begin{proof}
	For a conformal change \(\hat g=w^2\tilde g\), the boundary mean curvature transforms as
	\begin{equation}\label{eq:h-transform-w2}
		h_{\hat g}
		=w^{-1}\left(h_{\tilde g}+\frac{\partial\log w}{\partial\nu_{\tilde g}}\right)
		\quad\text{on }\partial'B_1^+.
	\end{equation}
	Thus \(h_{\hat g}=M\) is equivalent to
	\begin{equation}\label{eq:prescribed-h-bdy}
		\frac{\partial w}{\partial\nu_{\tilde g}}
		+h_{\tilde g}w=Mw^2
		\quad\text{on }\partial'B_1^+.
	\end{equation}

	In local coordinates, the outward unit normal has the form
	\[
	\nu_{\tilde g}
	=-\frac{1}{\sqrt{\tilde g^{nn}}}
	\sum_{i=1}^n\tilde g^{in}\frac{\partial}{\partial x_i}.
	\]
	Hence
	\[
	\frac{\partial x_n}{\partial\nu_{\tilde g}}
	=\nu_{\tilde g}(x_n)
	=-\sqrt{\tilde g^{nn}}.
	\]
	Define
	\begin{equation*}
		\rho(x)=\frac{x_n}{\sqrt{\tilde g^{nn}(x)}}.
	\end{equation*}
	Then
	\begin{equation}\label{eq:rho-properties}
		\rho=0\quad\text{on }\partial'B_1^+,
		\qquad
		\rho>0\quad\text{in }B_1^+,
		\qquad
		\frac{\partial\rho}{\partial\nu_{\tilde g}}=-1
		\quad\text{on }\partial'B_1^+.
	\end{equation}

	Let \(H\in C^\infty(\overline{B_1^+})\) be any smooth extension of \(h_{\tilde g}|_{\partial'B_1^+}\); for instance, one may take \(H(x',x_n)=h_{\tilde g}(x',0)\) in a collar neighborhood. Let \(\chi\in C^\infty([0,\infty))\) satisfy
	\[
	\chi=1\text{ on }[0,1],
	\qquad
	\chi=0\text{ on }[2,\infty),
	\qquad
	0\le\chi\le1.
	\]

	For \(\delta>0\), set
	\begin{equation*}
		w_\delta(x)
		=1+\rho(x)\chi\left(\frac{\rho(x)}{\delta}\right)(H(x)-M).
	\end{equation*}

	On \(\partial'B_1^+\), \(\rho=0\), so \(w_\delta=1\). In a collar neighborhood of the boundary, \(\chi(\rho/\delta)=1\), and \eqref{eq:rho-properties} gives
	\begin{align*}
		\frac{\partial w_\delta}{\partial\nu_{\tilde g}}
		&=\frac{\partial\rho}{\partial\nu_{\tilde g}}(H-M)
		+\rho\frac{\partial(H-M)}{\partial\nu_{\tilde g}}\\
		&=-(H-M)
		=M-h_{\tilde g}
		\quad\text{on }\partial'B_1^+.
	\end{align*}
	Therefore
	\[
	\frac{\partial w_\delta}{\partial\nu_{\tilde g}}
	+h_{\tilde g}w_\delta
	=M
	=Mw_\delta^2
	\quad\text{on }\partial'B_1^+,
	\]
	which is exactly \eqref{eq:prescribed-h-bdy}.

	Finally,
	\[
	|w_\delta-1|
	\le 2\delta\|H-M\|_{L^\infty(B_1^+)}.
	\]
	Choosing \(\delta>0\) so that
	\[
	2\delta\|H-M\|_{L^\infty(B_1^+)}<\frac12
	\]
	gives \(w_\delta\ge1/2\) on \(\overline{B_1^+}\). Hence \(w_\delta\) is an admissible positive conformal factor, and \eqref{eq:h-transform-w2} gives \(h_{\hat g}=M\).
\end{proof}

\begin{lem}\label{lem:elimination_R}
	Let \(n\ge3\), and suppose that \(g_0\) is a smooth metric with
	\[
	h_{g_0}=M>0
	\quad\text{on }\partial'B_1^+.
	\]
	After shrinking the radius, there exists a positive conformal factor \(w\) on \(B_r^+\) such that
	\[
	\hat g=w^{\frac4{n-2}}g_0
	\]
	satisfies
	\[
	R_{\hat g}=0\quad\text{in }B_r^+,
	\qquad
	h_{\hat g}=0\quad\text{on }\partial'B_r^+.
	\]
\end{lem}
\begin{proof}
	For \(\hat g=w^{4/(n-2)}g_0\), the conformal transformation laws reduce the desired conditions to the existence of a positive solution of the boundary value problem
	\begin{equation}\label{eq:local-conformal-zero-problem}
		\begin{cases}
			\displaystyle
			-\frac{4(n-1)}{n-2}\Delta_{g_0}w+R_{g_0}w=0,
			& \text{in }B_r^+,\\[2mm]
			\displaystyle
			\frac{\partial w}{\partial\nu_{g_0}}
			+\frac{n-2}{2}Mw=0,
			& \text{on }\partial'B_r^+,\\[2mm]
			w=1,
			& \text{on }\partial''B_r^+.
		\end{cases}
	\end{equation}

	Set
	\[
	c_n=\frac{4(n-1)}{n-2},
	\qquad
	V_r=\{\phi\in H^1(B_r^+):\phi=0\text{ on }\partial''B_r^+\}.
	\]
	The bilinear form associated with \eqref{eq:local-conformal-zero-problem} is
	\begin{equation*}
		\mathfrak B_r(\phi,\psi)
		=\int_{B_r^+}\bigl(c_n\langle\nabla\phi,\nabla\psi\rangle_{g_0}
		+R_{g_0}\phi\psi\bigr)\,dV_{g_0}
		+2(n-1)M\int_{\partial'B_r^+}\phi\psi\,dS_{g_0}.
	\end{equation*}
	Since \(M>0\), the boundary term is nonnegative. If \(|R_{g_0}|\le C_R\) on \(B_1^+\), then
	\begin{equation}\label{eq:B-lower-prep}
		\mathfrak B_r(\phi,
		\phi)
		\ge c_n\int_{B_r^+}|\nabla\phi|_{g_0}^2\,dV_{g_0}
		-C_R\int_{B_r^+}\phi^2\,dV_{g_0}.
	\end{equation}
	The Poincar\'e inequality on the half-ball, valid because \(\phi=0\) on \(\partial''B_r^+\), gives
	\begin{equation}\label{eq:Poincare-halfball}
		\int_{B_r^+}\phi^2\,dV_{g_0}
		\le C_Pr^2\int_{B_r^+}|\nabla\phi|_{g_0}^2\,dV_{g_0}.
	\end{equation}
	where \(C_P>0\) is a scale-invariant constant.

	Substituting \eqref{eq:Poincare-halfball} into \eqref{eq:B-lower-prep} yields
	\begin{equation*}
		\mathfrak B_r(\phi,\phi)
		\ge (c_n-C_RC_Pr^2)
		\int_{B_r^+}|\nabla\phi|_{g_0}^2\,dV_{g_0}.
	\end{equation*}
	Choosing \(r>0\) such that \(C_RC_Pr^2\le c_n/2\) gives
	\begin{equation}\label{eq:B-gradient-coercive}
		\mathfrak B_r(\phi,\phi)
		\ge \frac{c_n}{2}
		\int_{B_r^+}|\nabla\phi|_{g_0}^2\,dV_{g_0}.
	\end{equation}
	Combining \eqref{eq:B-gradient-coercive} with \eqref{eq:Poincare-halfball},
	\begin{equation}\label{eq:B-H1-coercive}
		\mathfrak B_r(\phi,\phi)
		\ge C\|\phi\|_{H^1(B_r^+)}^2
		\qquad
		\text{for all }\phi\in V_r,
	\end{equation}
	where \(C>0\) depends on \(r\), \(g_0\), and \(n\), but is strictly positive.

	Write \(w=1+v\), with \(v\in V_r\). By \eqref{eq:B-H1-coercive}, the Lax--Milgram theorem gives a unique \(v\in V_r\). Standard elliptic regularity for mixed boundary problems then implies
	\[
	w=1+v\in C^\infty(\overline{B_{r/2}^+}).
	\]

	The positivity of \(w\) follows from the maximum principle. The principal eigenvalue \(\mu_1(r)\) of the operator \(-L_{g_0}\) on \(\mathcal{B}_r^+\) with the prescribed mixed boundary conditions is characterized by the Rayleigh quotient
	\begin{equation*}
		\mu_1(r) = \inf_{\phi \in V \setminus \{0\}} \frac{\mathfrak B_r(\phi, \phi)}{\int_{\mathcal{B}_r^+} \phi^2 dV_{g_0}} \ge \inf_{\phi \in V \setminus \{0\}} \frac{\frac{c_n}{2} \int_{\mathcal{B}_r^+}|\nabla_{g_0} \phi|^2 dV_{g_0}}{C_P r^2 \int_{\mathcal{B}_r^+} |\nabla_{g_0} \phi|^2 dV_{g_0}} = \frac{c_n}{2 C_P r^2}.
	\end{equation*}
	Thus \(\mu_1(r)\ge c r^{-2}\to +\infty\) as \(r\to0\). In particular, \(\mu_1(r)\) is strictly positive for small \(r\) and \(-L_{g_0}\) satisfies the maximum principle on \(B_r^+\).

	If \(w\) attained a non-positive minimum, the minimum could not occur on \(\partial''B_r^+\), where \(w=1\). It also could not occur in the interior by the strong maximum principle. Thus it would occur at some point \(x_0\in\partial'B_r^+\). The Hopf boundary lemma would give
	\[
	\frac{\partial w}{\partial\nu_{g_0}}(x_0)<0.
	\]
	The boundary condition in \eqref{eq:local-conformal-zero-problem}, however, gives
	\[
	\frac{\partial w}{\partial\nu_{g_0}}(x_0)
	=-\frac{n-2}{2}Mw(x_0)\ge0,
	\]
	because \(M>0\) and \(w(x_0)\le0\). This contradiction proves \(w>0\).

	The conformal transformation laws now give \(R_{\hat g}=0\) in \(B_r^+\) and \(h_{\hat g}=0\) on \(\partial'B_r^+\).
\end{proof}

\begin{rem}
	Let \(\tilde{g}\) be the conformal metric on \(\mathcal{B}_{\delta}^+\) obtained by composing the two conformal changes in Lemma~\ref{lem:prescribed_hg} and Lemma~\ref{lem:elimination_R}, with \(\delta<r/2\). Then after the conformal change, the original equation becomes
	\begin{equation*}
		\begin{cases}
			-\Delta_{\tilde{g}} \tilde{u} = 0                                                & \text{in } \mathcal{B}_{\delta}^+\setminus\{0\},          \\
			\dfrac{\partial \tilde{u}}{\partial \nu_{\tilde{g}}} = \tilde{u}^{\frac{n}{n-2}} & \text{on } \partial'\mathcal{B}_{\delta}^+\setminus\{0\},
		\end{cases}
	\end{equation*}
	where \(\tilde{u}=\tilde{w}^{-1}u\). Under this transformation, the hypotheses of Proposition~\ref{prop:interior_control} is satisfied. Therefore, once the boundary blow-up has been excluded, the interior blow-up is also controlled.
\end{rem}
			\section{Proof of the main theorems}\label{sec:main_part}

			The upper bound \(\limsup_{x\to0} d_g(x,0)^{\frac{n-2}{2}}u(x) < \infty\) has been established by the method of moving spheres. This section proves the removability criterion and the asymptotic cylindrical symmetry near a non-removable boundary singularity.

			As the conformal normalization used above, there exists a smooth positive function \(\kappa\) such that the conformal metric \(\tilde g = \kappa^{\frac{4}{n-2}} g\) satisfies, in a neighborhood of \(0\),
			\begin{equation}\label{eq:metric_expansion}
				\begin{aligned}
					\tilde g_{ij}(x) & = \delta_{ij} + O(|x|),                   \\
					\det \tilde g    & = 1,                                      \\
					R_{\tilde g}(x)  & = O(|x|), \quad h_{\tilde g}(x) = O(|x|).
				\end{aligned}
			\end{equation}
			In particular, \(h_{\tilde g}(0)=0\). By the conformal invariance of the Yamabe equation, it is enough to work in the normalization in which \(g\) itself satisfies \eqref{eq:metric_expansion}.

			\begin{rem}
				In these coordinates, \(|\partial_k g^{ij}(x)| \le C\) and \(|R_g(x)| \le C|x|\). The boundary mean curvature satisfies \(|h_g(x)| \le C|x|\).
			\end{rem}

			\subsection{Pohozaev identity}
			The removability criterion is obtained from a Pohozaev invariant. The key point is that this invariant rules out oscillation of the normalized profile.

			For \(0<\rho<r_0\), let
			\begin{equation}\label{eq:P_def}
				\begin{aligned}
					P(\rho,u)& :=
					\int_{\partial''\mathcal B_{\rho}^{+}}
					\Bigl(
					\frac{n-2}{2} u \frac{\partial u}{\partial\nu}
					-\frac{\rho}{2} |\nabla u|^{2}
					+\rho\Bigl|\frac{\partial u}{\partial\nu}\Bigr|^{2}
					\Bigr) d\sigma
					\\
					& \qquad
					+\int_{\partial B_{\rho}\cap\{x_n=0\}}
					\rho \frac{n-2}{2(n-1)} u^{\frac{2(n-1)}{n-2}} ds',
				\end{aligned}
			\end{equation}
			where \(\nu=x/\rho\) is the unit outward radial vector, \(d\sigma\) denotes
			the induced Euclidean volume on the hemisphere \(\partial''\mathcal B_{\rho}^{+}=\partial\mathcal B_{\rho}^{+}\setminus\partial'\mathcal B_{\rho}^{+}\),
			and \(ds'\) the Euclidean volume on the equator.

			Set the multiplier
			\begin{equation*}
				X := x^{k}\partial_{k}u + \frac{n-2}{2} u.
			\end{equation*}
			The corresponding Pohozaev invariant is defined as
			\begin{equation*}
				\mathcal{P}(u):= \lim_{\rho\to0} P(\rho,u),
			\end{equation*}
			provided that the limit exists. The following identity and error estimates imply that the limit is well defined.

			\medskip
			\noindent\textbf{1.  Exact Pohozaev identity with geometric errors.}

			For \(0<\varepsilon<r<r_0\) let \(\mathcal A_{\varepsilon,r}= \mathcal B_{r}^{+}\setminus\overline{\mathcal B_{\varepsilon}^{+}}\).
			Multiply the equation \(-L_gu=0\) by \(X\) and integrate over \(\mathcal A_{\varepsilon,r}\).
			Writing \(L_gu = \Delta u + (L_gu-\Delta u)\) and setting \(\mathcal E_{\rm int}:=L_gu-\Delta u\), it follows that
			\begin{equation*}
				-\int_{\mathcal A_{\varepsilon,r}}X\Delta u dx
				= \int_{\mathcal A_{\varepsilon,r}}X \mathcal E_{\rm int} dx.
			\end{equation*}
			On the other hand, the vector-field identity
			\begin{equation*}
				\int_{\Omega}(-\Delta u)X dx
				= \int_{\partial\Omega}
				\Bigl(
				-\frac{\partial u}{\partial\nu_{\Omega}}X
				+\frac12|\nabla u|^{2}(x\cdot\nu_{\Omega})
				\Bigr) dA
			\end{equation*}
			holds for any smooth domain \(\Omega\) with outward unit normal \(\nu_{\Omega}\).
			Applying it to \(\Omega=\mathcal A_{\varepsilon,r}\) and noting that
			\(\partial\mathcal A_{\varepsilon,r}\) consists of the outer hemisphere
			\(\partial''\mathcal B_{r}^{+}\) (\( \nu_A=\nu=x/r\)),
			the inner hemisphere \(\partial''\mathcal B_{\varepsilon}^{+}\) (\( \nu_A=-\nu=-x/\varepsilon\)),
			and the flat annular part \(\partial'\mathcal A_{\varepsilon,r}\subset\{x_n=0\}\) (\( \nu_A=(0,\dots,-1)\)),
			it follows that
			\begin{align}
				-\int_{\mathcal A_{\varepsilon,r}}X\Delta u dx
				& = -P_{\rm flat}(r,u) + P_{\rm flat}(\varepsilon,u)
				\notag                                             \\
				& \quad
				+\int_{\partial'\mathcal A_{\varepsilon,r}}
				\Bigl( \partial_n u (x^{k}\partial_{k}u)
				+\frac{n-2}{2} u \partial_n u \Bigr) dx',
				\label{eq:vector}
			\end{align}
			where
			\begin{equation*}
				P_{\rm flat}(\rho,u)
				:=
				\int_{\partial''\mathcal B_{\rho}^{+}}
				\Bigl(
				\frac{n-2}{2} u \frac{\partial u}{\partial\nu}
				-\frac{\rho}{2} |\nabla u|^{2}
				+\rho\Bigl|\frac{\partial u}{\partial\nu}\Bigr|^{2}
				\Bigr) d\sigma.
			\end{equation*}

			On the flat boundary the Escobar equation can be rewritten as
			\begin{equation*}
				-\partial_n u = u^{\frac{n}{n-2}} + \mathcal E_{\rm bdy},
				\qquad
				\mathcal E_{\rm bdy}:= -\mu^{i}\partial_i u - \frac{n-2}{2}h_g u.
			\end{equation*}
			Substituting this relation into \eqref{eq:vector}, the flat-boundary integral becomes
			\begin{align}
				& \int_{\partial'\mathcal A_{\varepsilon,r}}  \Bigl( \partial_n u x^{k}\partial_{k}u
				+\frac{n-2}{2}u\partial_n u \Bigr)dx'  \notag                               \\
				& = -\int_{\partial'\mathcal A_{\varepsilon,r}}
				u^{\frac{n}{n-2}}\bigl(x^{k}\partial_{k}u\bigr)dx'
				-\frac{n-2}{2}\int_{\partial'\mathcal A_{\varepsilon,r}} u^{\frac{2(n-1)}{n-2}}dx'
				-\int_{\partial'\mathcal A_{\varepsilon,r}} \mathcal E_{\rm bdy} X dx'.
				\label{eq:flat}
			\end{align}

			The first two terms on the right-hand side of \eqref{eq:flat} are treated by the
			identity
			\begin{equation*}
				u^{\frac{n}{n-2}}\partial_{k}u
				= \frac{n-2}{2(n-1)} \partial_{k}\!\bigl(u^{\frac{2(n-1)}{n-2}}\bigr).
			\end{equation*}
			Integrating by parts on the flat annulus \(\{\varepsilon\le|x'|\le r\}\) and using
			the cancellation between the divergence of \(x'\) and the constant coefficient,
			one obtains precisely
			\begin{align}
				& -\int_{\partial'\mathcal A_{\varepsilon,r}}u^{\frac{n}{n-2}}(x^{k}\partial_{k}u) dx'
				-\frac{n-2}{2}\int_{\partial'\mathcal A_{\varepsilon,r}} u^{\frac{2(n-1)}{n-2}}dx'
				\notag                                                                                        \\
				& \qquad=
				-\frac{n-2}{2(n-1)}\biggl(
				\int_{|x'|=r} r u^{\frac{2(n-1)}{n-2}}ds'
				-\int_{|x'|=\varepsilon} \varepsilon u^{\frac{2(n-1)}{n-2}}ds'
				\biggr).
				\label{eq:boundary_parts}
			\end{align}

			Inserting \eqref{eq:boundary_parts} back into \eqref{eq:flat} and then into
			\eqref{eq:vector}, the boundary and spherical terms combine exactly into the difference of the Pohozaev integrals defined in
			\eqref{eq:P_def}.  Consequently,
			\begin{equation}\label{eq:Pohozaev_exact}
				P(\varepsilon,u) - P(r,u)
				= \int_{\mathcal A_{\varepsilon,r}} X\mathcal E_{\rm int}\,dx
				+ \int_{\partial'\mathcal A_{\varepsilon,r}} X \mathcal E_{\rm bdy} dx'.
			\end{equation}

			\medskip
			\noindent\textbf{2.  Convergence of the error terms.}

			In the chosen coordinates, the metric coefficients satisfy
			\[
			|g^{ij}-\delta^{ij}|\le C|x|,\qquad
			|\partial_i g^{ij}|\le C,\qquad
			|\mu^i|\le C|x|,
			\qquad
			|R_g|+|h_g|\le C|x|.
			\]
			Together with the upper bound and the interior Harnack estimates, this gives
			\[
			u(x)\le C|x|^{\frac{2-n}{2}},\qquad
			|\nabla u(x)|\le C|x|^{-\frac n2},\qquad
			|\nabla^2 u(x)|\le C|x|^{-\frac{n+2}{2}}
			\]
			and the Pohozaev multiplier satisfies
			\[
			|X|\le C\bigl(|x|\,|\nabla u|+u\bigr)
			\le C|x|^{\frac{2-n}{2}}.
			\]
			For the interior volume error, the terms in \(L_gu-\Delta u\) are pointwise bounded by
			\begin{equation*}
				|\mathcal E_{\rm int}| \le C\bigl(|\nabla u| + |x||\nabla^2 u| + |x|u\bigr) \le C|x|^{-\frac{n}{2}}.
			\end{equation*}

			Consequently, the pointwise product satisfies the scale-invariant estimate
			\begin{equation*}
				|X\mathcal E_{\rm int}| \le C |x|^{\frac{2-n}{2}} \cdot |x|^{-\frac{n}{2}} = C |x|^{1-n}.
			\end{equation*}
			Integration over the half-annulus \(\mathcal A_{\varepsilon,r}\) yields exactly the linear bound:
			\begin{align*}
				\Bigl|\int_{\mathcal A_{\varepsilon,r}} X\mathcal E_{\rm int}\,dx\Bigr|
				& \le C\int_{\mathcal A_{\varepsilon,r}} |x|^{1-n} dx                                    \\
				& \le C\int_{\varepsilon}^{r} \rho^{1-n} \cdot \rho^{n-1}d\rho = C(r-\varepsilon) \le C r.
			\end{align*}

			On the flat boundary, the error is bounded by:
			\begin{equation*}
				|\mathcal E_{\rm bdy}|\le C\bigl( |x||\nabla u| + |x| u \bigr) \le C|x|^{1-\frac{n}{2}}.
			\end{equation*}
			Consequently, the boundary integral satisfies
			\begin{align*}
				\Bigl|\int_{\partial'\mathcal A_{\varepsilon,r}} X\mathcal E_{\rm bdy}\,dx'\Bigr|
				& \le C\int_{\partial'\mathcal A_{\varepsilon,r}} |X||\mathcal E_{\rm bdy}|\,dx'                      \\
				& \le C\int_{\varepsilon}^{r} \rho^{\frac{2-n}{2}} \cdot \rho^{1-\frac{n}{2}} \cdot \rho^{n-2}d\rho \\
				& = C\int_{\varepsilon}^{r} \rho^{2-n} \cdot \rho^{n-2}d\rho = C(r-\varepsilon) \le C r.
			\end{align*}

			Both error terms are therefore absolutely integrable up to the origin.

			Define
			\begin{equation*}
				E_{\rm vol}(r):= \lim_{\varepsilon\to0}\int_{\mathcal A_{\varepsilon,r}} X \mathcal E_{\rm int} dx,
				\qquad
				I_h(r):= \lim_{\varepsilon\to0}\int_{\partial'\mathcal A_{\varepsilon,r}} X \mathcal E_{\rm bdy} dx',
			\end{equation*}
			and thus they satisfy
			\begin{equation}\label{eq:error_bounds}
				|E_{\rm vol}(r)|\le C r,\qquad |I_h(r)|\le C r.
			\end{equation}

			\medskip
			\noindent\textbf{3.  Existence of the Pohozaev invariant.}
			Since the right-hand side of \eqref{eq:Pohozaev_exact} possesses a finite limit
			as \(\varepsilon\to0\), the left-hand side forces
			\(\displaystyle \lim_{\varepsilon\to0}P(\varepsilon,u)\) to exist.  Denote this limit by
			\begin{equation*}
				\mathcal{P}(u):= \lim_{\rho\to0} P(\rho,u).
			\end{equation*}
			Passing to the limit \(\varepsilon\to0\) in \eqref{eq:Pohozaev_exact} gives the
			exact identity
			\begin{equation*}
				P(r,u) = \mathcal{P}(u) - E_{\rm vol}(r) - I_h(r),\qquad 0<r<r_0,
			\end{equation*}
			where \(\mathcal{P}(u)\) is a well-defined finite number (the Pohozaev invariant) and the remainder terms satisfy the quantitative
			bounds \eqref{eq:error_bounds}.

			\subsection{Removable singularity}

			\begin{lem}\label{lem:pohozaev_limit}
				Let \(u\) be a positive solution to the boundary Yamabe problem \eqref{eq:Main_PDE} in the normalized coordinates with \(n \ge 3\). Assume that the Pohozaev invariant \(\mathcal{P}(u) := \lim_{r \to 0} P(r, u)\) exists. If
				\begin{equation*}
					\liminf_{x\to 0}|x|^{\frac{n-2}{2}}u(x)=0,
				\end{equation*}
				then \(\mathcal{P}(u) = 0\).
			\end{lem}

			\begin{proof}
				The established upper bound gives \(\limsup_{x\to 0}|x|^{\frac{n-2}{2}}u(x)<\infty\). Two cases are possible.

				\medskip
				\noindent\textbf{Case 1: The limit vanishes.}
				Suppose \(\limsup_{x\to0}|x|^{\frac{n-2}{2}}u(x) = 0\). Because \(u\) is positive, this assumption immediately implies that the strict limit exists and vanishes:
				\begin{equation*}\lim_{x\to0}|x|^{\frac{n-2}{2}}u(x) = 0.\end{equation*}
				Equivalently, \(u=o(|x|^{\frac{2-n}{2}})\). The gradient estimate following from \eqref{eq:Harnack} gives
				\begin{equation*}
					|\nabla u(x)|\le C|x|^{-1}u = o(|x|^{-\frac{n}{2}}).
				\end{equation*}

				It remains to evaluate the asymptotic behavior of the Pohozaev integral \(P(r,u)\) term by term as \(r \to 0\). Recall that the \((n-1)\)-dimensional area measure of the hemisphere \(\partial''\mathcal{B}_r^+\) scales as \(O(r^{n-1})\), and the \((n-2)\)-dimensional area measure of the flat equator \(\partial B_r \cap \{x_n=0\}\) scales as \(O(r^{n-2})\).

				For the mixed boundary derivative term on the hemisphere, the pointwise estimate is
				\begin{equation*}
					u \frac{\partial u}{\partial \nu} = o(r^{\frac{2-n}{2}}) \cdot o(r^{-\frac{n}{2}}) = o(r^{1-n}).
				\end{equation*}
				After integration over the hemisphere, we have
				\begin{equation*}\int_{\partial''\mathcal{B}_r^+} u \frac{\partial u}{\partial \nu}  d\sigma = o(r^{1-n}) \cdot O(r^{n-1}) = o(1).\end{equation*}

				Moreover, for the quadratic gradient terms on the hemisphere, the radial weight \(r\) provides an extra vanishing factor
				\begin{equation*}
					r |\nabla u|^2 = r \cdot o(r^{-n}) = o(r^{1-n}).
				\end{equation*}
				The integral over the hemisphere is bounded by
				\begin{equation*}\int_{\partial''\mathcal{B}_r^+} r \Bigl( \frac{1}{2}|\nabla u|^2 - \Bigl|\frac{\partial u}{\partial \nu}\Bigr|^2 \Bigr)  d\sigma = o(r^{1-n}) \cdot O(r^{n-1}) = o(1).\end{equation*}

				We next consider the nonlinear boundary term on the flat equator. The critical Sobolev exponent is exactly \(p = \frac{2(n-1)}{n-2}\). Substituting the asymptotic bound for \(u\), one obtains
				\begin{equation*}
					u^{\frac{2(n-1)}{n-2}} = o\Bigl( (r^{\frac{2-n}{2}})^{\frac{2(n-1)}{n-2}} \Bigr) = o(r^{1-n}).
				\end{equation*}
				Multiplying by the radial weight \(r\) yields \(r u^{\frac{2(n-1)}{n-2}} = o(r^{2-n})\). Integrating this over the equator gives
				\begin{equation*}\int_{\partial B_r \cap \{x_n=0\}} r u^{\frac{2(n-1)}{n-2}}  ds' = o(r^{2-n}) \cdot O(r^{n-2}) = o(1).\end{equation*}

				Summing all these components, the preceding estimates implies that the entire Pohozaev integral vanishes asymptotically
				\begin{equation*}
					P(r,u) = o(1) + o(1) + o(1) = o(1) \quad \text{as } r \to 0.
				\end{equation*}
				and thus
				\(\mathcal{P}(u) = 0.\)

				\medskip
				\noindent\textbf{Case 2: The solution oscillates.}
				Suppose \(\limsup_{x\to0}|x|^{\frac{n-2}{2}}u(x) > 0\).
				Let \(\bar{u}(r)\) denote the spherical average of \(u\) over \(\partial''\mathcal{B}_r^+\), namely,
				\begin{equation*}
					\bar u(r) = \frac{1}{|\partial''\mathcal{B}_r^+|}\int_{\partial''\mathcal{B}_r^+} u(x) d\sigma.
				\end{equation*}
				The uniform Harnack inequality gives \(u(x) \asymp \bar{u}(r)\) for \(|x|=r\) which means that there exists a constant \(C>1\), independent of \(r\), such that
				\[
				C^{-1}\bar u(r)\le u(x)\le C\bar u(r),
				\qquad |x|=r.
				\]

				Define the normalized function
				\begin{equation*}
					w(r) := r^{\frac{n-2}{2}}\bar{u}(r).
				\end{equation*}
				By the hypothesis \(\liminf_{r \to 0} w(r) = 0\), there exists a sequence of radii \(r_i \to 0\) such that \(r_i\) are local minima of \(w(r)\). This critical point condition dictates
				\begin{equation}\label{eq:local_min_condition}
					w_i := w(r_i) \to 0, \qquad \text{and} \qquad \left. \frac{d}{dr} \left( r^{\frac{n-2}{2}}\bar{u}(r) \right) \right|_{r=r_i} = 0.
				\end{equation}

				Fix a unit vector \(e_1 \in \partial B_1 \cap \{x_n=0\}\) and define the rescaled sequence
				\begin{equation*}
					v_i(y) := \frac{u(r_i y)}{u(r_i e_1)}, \qquad y \in \overline{\mathcal{B}_{R_i}^+} \setminus \{0\},
				\end{equation*}
				where \(R_i = r_0 / r_i \to \infty\). By construction, \(v_i(e_1) = 1\). The rescaled metric is \(g_i(y) = g(r_i y) = \delta_{ij} + O(r_i|y|)\).

				The boundary condition \(B_g u = u^{\frac{n}{n-2}}\) transforms under rescaling to:
				\begin{equation*}
					-\partial_{y_n} v_i = r_i u(r_i e_1)^{\frac{2}{n-2}} v_i^{\frac{n}{n-2}} + O(r_i|y|)v_i.
				\end{equation*}
				The coefficient of the nonlinear term is determined by \(w_i\):
				\begin{equation*}
					r_i u(r_i e_1)^{\frac{2}{n-2}} = r_i \left( u(r_i e_1) r_i^{\frac{n-2}{2}} r_i^{-\frac{n-2}{2}} \right)^{\frac{2}{n-2}} \asymp r_i \left( w_i r_i^{-\frac{n-2}{2}} \right)^{\frac{2}{n-2}} = w_i^{\frac{2}{n-2}}.
				\end{equation*}
				Since \(w_i \to 0\), the nonlinear coefficient vanishes, namely, \(\lim_{i \to \infty} w_i^{\frac{2}{n-2}} = 0\).

				The normalization at the fixed boundary point does not lose compactness information. Indeed, by the uniform Harnack inequality on each fixed annulus
				\(K\Subset \overline{\mathbb R^n_+}\setminus\{0\}\),
				\begin{equation*}
					C_K^{-1}\le v_i\le C_K \qquad \text{on }K,
				\end{equation*}
				because \(u(r_i e_1)\asymp \bar u(r_i)\) and \(r_i^{\frac{n-2}{2}}u(r_i e_1)\asymp w_i\). Hence the rescaled boundary coefficient
				\(r_i u(r_i e_1)^{2/(n-2)}\) is \(O(w_i^{2/(n-2)})\) and tends to zero. Standard elliptic interior and boundary estimates then imply that \(\{v_i\}\) is precompact in
				\(C^2_{\rm loc}(\overline{\mathbb{R}^n_+}\setminus\{0\})\). Passing to a subsequence, \(v_i \to v\), where the limit function \(v \ge 0\) satisfies the linear Neumann system
				\begin{equation*}
					\begin{cases}
						-\Delta v = 0    & \text{in } \mathbb{R}^n_+ \setminus \{0\},         \\
						\partial_n v = 0 & \text{on } \partial\mathbb{R}^n_+ \setminus \{0\}.
					\end{cases}
				\end{equation*}

				Since \(v\) satisfies the homogeneous Neumann condition \(\partial_n v = 0\), define its even reflection across the boundary:
				\begin{equation*}
					\tilde{v}(y', y_n) :=
					\begin{cases}
						v(y', y_n)  & \text{if } y_n \ge 0, \\
						v(y', -y_n) & \text{if } y_n < 0.
					\end{cases}
				\end{equation*}
				By the Schwarz Reflection Principle, \(\tilde{v}\) is a positive, smooth harmonic function on \(\mathbb{R}^n \setminus \{0\}\). Applying Bôcher's Theorem to \(\tilde{v}\), there exist constants \(A, B \ge 0\) such that:
				\begin{equation*}
					\tilde{v}(y) = \frac{A}{|y|^{n-2}} + B.
				\end{equation*}
				Restricting back to the half-space, \(v(y) = A|y|^{2-n} + B\) for \(y \in \overline{\mathbb{R}^n_+} \setminus \{0\}\).

				The normalization condition \(v(e_1) = 1\) implies:
				\begin{equation}\label{eq:norm_AB}
					A + B = 1.
				\end{equation}
				Furthermore, the local minimum condition \eqref{eq:local_min_condition} transfers exactly to the rescaled averages. If \(\bar v_i(\rho)\) denotes the average of \(v_i\) on \(\partial''\mathcal B_\rho^+\), then
				\begin{equation*}
					\bar v_i(\rho)=\frac{\bar u(r_i\rho)}{u(r_i e_1)}.
				\end{equation*}
				Thus multiplication by the constant \(u(r_i e_1)^{-1}\) does not affect the critical point condition. After the change of variables \(r=r_i\rho\),
				\begin{equation*}
					\left.\frac{d}{d\rho}\left(\rho^{\frac{n-2}{2}}\bar v_i(\rho)\right)\right|_{\rho=1}=0.
				\end{equation*}
				Passing to the limit \(i \to \infty\) gives
				\begin{equation*}
					\left. \frac{d}{dr} \left( r^{\frac{n-2}{2}}\bar{v}(r) \right) \right|_{r=1} = 0.
				\end{equation*}
				Since \(v\) is radial, \(\bar{v}(r) = A r^{2-n} + B\). Expanding the derivative:
				\begin{equation*}
					\frac{d}{dr} \left( A r^{\frac{2-n}{2}} + B r^{\frac{n-2}{2}} \right) \bigg|_{r=1} = \frac{2-n}{2}A + \frac{n-2}{2}B = 0 \implies A = B.
				\end{equation*}
				Coupled with \eqref{eq:norm_AB}, it follows that \(A = B = \frac{1}{2}\). Hence,
				\begin{equation*}
					v(y) = \frac{1}{2}|y|^{2-n} + \frac{1}{2}.
				\end{equation*}

				The standard flat Pohozaev integral is evaluated for the limit function \(v\) at \(r=1\):
				\begin{equation*}
					P_{\rm flat}(1, v) := \int_{\partial''\mathcal{B}_1^+} \left( \frac{n-2}{2} v \frac{\partial v}{\partial \nu} - \frac{1}{2} |\nabla v|^2 + \left| \frac{\partial v}{\partial \nu} \right|^2 \right) d\sigma.
				\end{equation*}
				Evaluating \(v\) and its radial derivative at \(r=1\):
				\begin{equation*}
					v(1) = 1, \qquad \frac{\partial v}{\partial \nu}(1) = \frac{1}{2}(2-n)(1)^{1-n} = \frac{2-n}{2}.
				\end{equation*}
				Substitute these directly into the integral:
				\begin{align*}
					P_{\rm flat}(1, v) & = \int_{\partial''\mathcal{B}_1^+} \left[ \frac{n-2}{2} (1) \left(\frac{2-n}{2}\right) - \frac{1}{2} \left(\frac{2-n}{2}\right)^2 + \left(\frac{2-n}{2}\right)^2 \right] d\sigma \notag \\
					& = \left[ -\frac{(n-2)^2}{4} + \frac{(n-2)^2}{8} \right] |\partial''\mathcal{B}_1^+| \notag                                                                                              \\
					& = - \frac{(n-2)^2}{8} \omega_n' =: -c_0.
				\end{align*}
				Crucially, \(P_{\rm flat}(1, v)\) is a strictly negative constant \(-c_0 < 0\).

				The Pohozaev integral \(P(r_i,u)\) can be rewritten in terms of the rescaled function \(v_i(y)\). Using \(u(x) = u(r_i e_1) v_i(x/r_i)\) and \(d\sigma_x = r_i^{n-1}d\sigma_y\), we have
				\begin{align*}
					P(r_i, u) & = \int_{\partial''\mathcal{B}_{r_i}^+} \left( \frac{n-2}{2} u u_\nu - \frac{r_i}{2}|\nabla u|^2 + r_i u_\nu^2 \right) d\sigma_x \notag \\
					& \quad + \int_{\partial B_{r_i}\cap\{x_n=0\}} r_i \frac{n-2}{2(n-1)} u^{\frac{2n-2}{n-2}} dx' \notag                                    \\
					& = u(r_i e_1)^2 r_i^{n-2} \left[ P_{\rm flat}(1, v_i) + o(1) \right]   \notag                                                      \\
					& \qquad+ \int_{\partial B_1\cap\{y_n=0\}} r_i u(r_i e_1)^{\frac{2n-2}{n-2}} r_i^{n-2} v_i^{\frac{2n-2}{n-2}} dy' \notag                 \\
					& = u(r_i e_1)^2 r_i^{n-2} \left[ P_{\rm flat}(1, v_i) + O \left( r_i u(r_i e_1)^{\frac{2}{n-2}} \right) + o(1) \right].
				\end{align*}
				Let \(w_i^*:= u(r_i e_1)r_i^{\frac{n-2}{2}}\). By Harnack inequality, \(w_i^*\asymp w_i \to 0\). The boundary error coefficient \(r_i u(r_i e_1)^{\frac{2}{n-2}} \asymp w_i^{\frac{2}{n-2}}\) vanishes as \(i \to \infty\).

				Because \(v_i \to v\) in \(C^2_{\rm loc}\), the integral \(P_{\rm flat}(1, v_i) \to P_{\rm flat}(1, v) = -c_0\). Thus, the full Pohozaev integral at the critical sequence follows the precise asymptotic scaling
				\begin{equation*}
					P(r_i, u) = (w_i^*)^2 \left( -c_0 + o(1) \right).
				\end{equation*}
				By the existence of the limit \(\mathcal{P}(u) = \lim_{r \to 0} P(r, u)\), evaluating it along the sequence \(r_i\):
				\begin{equation*}
					\mathcal{P}(u) = \lim_{i \to \infty} P(r_i, u) = \lim_{i \to \infty} \left[ (w_i^*)^2 (-c_0) \right] = 0.
				\end{equation*}
				This concludes the proof.
			\end{proof}

			\begin{lem}\label{lem:removable}
				Assume that \(\mathcal{P}(u) \ge 0\). If \(\liminf_{x \to 0} u(x)|x|^{\frac{n-2}{2}} = 0\), then \(\limsup_{x \to 0} u(x)|x|^{\frac{n-2}{2}} = 0\).
			\end{lem}

			\begin{proof}
				Argue by contradiction. Suppose that
				\begin{equation*}
					\limsup_{x\to0} u(x)|x|^{\frac{n-2}{2}} > 0.
				\end{equation*}
				Together with the hypothesis \(\liminf_{x\to0} u(x)|x|^{\frac{n-2}{2}} = 0\), the solution oscillates infinitely often as \(x\to0\).

				Set \(t = -\ln |x|\) and define the spherical average over the hemisphere
				\begin{equation*}
					\bar u(r) = \frac{1}{|\partial''\mathcal{B}_r^+|}\int_{\partial''\mathcal{B}_r^+} u(x) d\sigma,
					\qquad r = e^{-t}.
				\end{equation*}
				Introduce the normalized function
				\begin{equation*}
					w(t) = r^{\frac{n-2}{2}} \bar u(r) = e^{-t\frac{n-2}{2}} \bar u(e^{-t}).
				\end{equation*}
				By the assumptions, there exists a sequence of local minima \(t_i\to\infty\) such that
				\begin{equation*}
					w_t(t_i)=0,\qquad w_i:=w(t_i)\to 0.
				\end{equation*}
				Fix a sufficiently small constant \(\varepsilon_0>0\) (to be chosen later) and let \([\bar t_i,t_i^*]\subset[T_0,\infty)\) be the maximal interval containing \(t_i\) on which \(w(t)\le\varepsilon_0\) and
				\begin{equation*}
					w(\bar t_i)=w(t_i^*)=\varepsilon_0.
				\end{equation*}
				Since \(\limsup_{t\to\infty} w(t) > \varepsilon_0\) and \(w_i\to0\), it follows that \(\bar t_i\to\infty\).

				Recall the error notation introduced in the Pohozaev identity,
				\(-\Delta u=\mathcal E_{\rm int}\) and
				\(-\partial_n u=u^{\frac n{n-2}}+\mathcal E_{\rm bdy}\) on \(\{x_n=0\}\).
				In the normalized coordinates, these errors satisfy
				\begin{equation*}
					|\mathcal E_{\rm int}(x)|\le C|x|^{-1}u(x),\qquad
					|\mathcal E_{\rm bdy}(x)|\le C u(x).
				\end{equation*}
				Integrating \(-\Delta u = \mathcal E_{\rm int}\) over \(\mathcal B_r^+\),
				the divergence theorem and the boundary condition on the flat part give
				\begin{equation}\label{eq:avg_eq}
					\frac{d}{dr}\bigl(r^{n-1}\bar u'(r)\bigr)
					= -\frac{1}{\omega'_n}\int_{\partial B_r\cap\{x_n=0\}}u^{\frac{n}{n-2}} ds' + O(r^{n-2}) \bar u(r),
				\end{equation}
				where \(\omega'_n=|\partial''\mathcal{B}_1^+|\) and the \(O(r^{n-2})\) term absorbs the metric errors in \(\mathcal E_{\rm int}\) and \(\mathcal E_{\rm bdy}\).

				The uniform Harnack inequality gives \(u(x) \asymp \bar{u}(r)\) for \(|x|=r\) where \(\bar{u}(r)\) is the spherical average of \(u\) over \(\partial''\mathcal{B}_r^+\). Consequently, the nonlinear integral on the equator is bounded from both sides by the spherical average
				\begin{equation*}
					c_1 r^{n-2} \bar u(r)^{\frac{n}{n-2}} \le \int_{\partial B_r\cap\{x_n=0\}} u^{\frac{n}{n-2}} ds' \le C_1 r^{n-2} \bar u(r)^{\frac{n}{n-2}}.
				\end{equation*}
				Insert these bounds into \eqref{eq:avg_eq} and convert to the cylindrical variable \(t=-\ln r\). Setting \(\alpha = \frac{n-2}{2}\), \(p = \frac{n}{n-2}\), and \(w(t) = r^\alpha \bar u(r)\), a direct computation shows that the \(r\)-powers cancel out
				\begin{equation*}
					r\frac{d}{dr}\bigl(r^{n-1}\bar u'(r)\bigr) = r^{\frac{n-2}{2}}\Bigl(w_{tt} - \alpha^2 w\Bigr).
				\end{equation*}
				All metric errors are of order \(O(e^{-t}) w\). Consequently, \(w\) satisfies, for \(t\) large enough, the double-sided differential inequalities
				\begin{equation}\label{eq:ode_bounds}
					- C_1 w^p - C_2 e^{-t} w \le w_{tt} - \alpha^2 w \le - c_1 w^p + C_2 e^{-t} w,
				\end{equation}
				where \(C_1, C_2, c_1\) are positive constants depending only on \(n\), the metric, and the uniform Harnack constant.

				On the safe interval \([\bar t_i,t_i^*]\), \(w(t)\le\varepsilon_0\).
				After decreasing \(\varepsilon_0\) and increasing the initial time if necessary,
				the lower bound in \eqref{eq:ode_bounds} gives
				\[
				w_{tt}\ge \frac{\alpha^2}{2}w
				\qquad\text{on }[\bar t_i,t_i^*].
				\]
				Hence \(w_t\) is strictly increasing on this interval. Since
				\(w_t(t_i)=0\),
				\[
				w_t\le0\quad\text{on }[\bar t_i,t_i],
				\qquad
				w_t\ge0\quad\text{on }[t_i,t_i^*].
				\]
				The length bounds below follow from the differential inequalities and a
				first-order comparison. Recall that \(w_i = w(t_i)\).

				\medskip
				\noindent\emph{Right interval \([t_i,t_i^*]\) -- upper bound.}
				Here \(w_t\ge0\). Using the lower bound in \eqref{eq:ode_bounds}, we have
				\begin{equation*}
					w_{tt} - \alpha^2 w \ge - C_1 w^p - C_2 e^{-t} w.
				\end{equation*}
				Multiply by \(w_t\) and integrate from \(t_i\) to \(t\) (\(t_i\le t\le t_i^*\)):
				\begin{align*}
					\frac12 w_t(t)^2 - \frac12 \alpha^2\bigl(w(t)^2-w_i^2\bigr)
					& \ge -C_1\int_{t_i}^t w^p w_s ds - C_2\int_{t_i}^t e^{-s} w w_s ds.
				\end{align*}
				The first right-hand integral is \(-\frac{C_1}{p+1}\bigl(w(t)^{p+1}-w_i^{p+1}\bigr)\).
				For the second, integration by parts gives
				\begin{equation*}
					-\int_{t_i}^t e^{-s} \frac12 (w^2)_s ds
					= -\frac12 e^{-t} w(t)^2 + \frac12 e^{-t_i} w_i^2 - \frac{1}{2}\int_{t_i}^t e^{-s} w^2 ds.
				\end{equation*}
				Since \(w(s)\le w(t)\) on this interval, the last integral is bounded by \(C e^{-t_i} w(t)^2\).
				Combining these estimates and using \(w_i\ll 1\) gives
				\begin{equation*}
					w_t(t)^2 \ge \alpha^2\bigl(w^2-w_i^2\bigr)\bigl(1 - C w^{p-1} - C e^{-t_i}\bigr).
				\end{equation*}
				Because \(w(t)\le\varepsilon_0\) and \(p>1\), the term \(C w^{p-1}\) can be made arbitrarily small by choosing \(\varepsilon_0\) small enough.
				Thus, for large \(i\),
				\begin{equation*}
					w_t(t) \ge \alpha \sqrt{w^2-w_i^2} \bigl(1 - C e^{-t_i} - C w^{p-1}\bigr).
				\end{equation*}
				Integrating \(dt = dw/w_t\) from \(w_i\) to \(\varepsilon_0\) yields
				\begin{align*}
					t_i^*- t_i
					& \le \frac{1}{\alpha} \int_{w_i}^{\varepsilon_0} \frac{1 + C e^{-t_i} + C w^{p-1}}{\sqrt{w^2 - w_i^2}}   dw \\
					& = \frac{1 + C e^{-t_i}}{\alpha} \left[ \ln \frac{\varepsilon_0 + \sqrt{\varepsilon_0^2 - w_i^2}}{w_i} \right]
					+ \frac{C}{\alpha} \int_{w_i}^{\varepsilon_0} \frac{w^{p-1}}{\sqrt{w^2 - w_i^2}}   dw
				\end{align*}

				The last integral remains uniformly bounded as \(w_i\to0\) since \(p>1\) (in fact, it has a finite limit) and we denote this bound by \(C_{\rm nonlin}\).
				Moreover, \(\ln\frac{\varepsilon_0+\sqrt{\varepsilon_0^2-w_i^2}}{w_i} = \ln\frac{2\varepsilon_0}{w_i} + O(w_i^2)\).
				Hence, after renaming constants,
				\begin{equation}\label{eq:ub_right}
					t_i^*- t_i \le \frac1\alpha \ln\frac1{w_i} + C_R e^{-t_i} \ln\frac1{w_i} + C.
				\end{equation}

				\noindent\emph{Right interval \([t_i,t_i^*]\) -- lower bound.}
				For the opposite inequality, a first-order comparison is used which does not
				freeze the coefficient \(e^{-t}\) at an endpoint. Dropping the nonlinear
				term from the upper bound in \eqref{eq:ode_bounds} gives
				\[
				w_{tt}\le (\alpha^2+C e^{-t})w.
				\]
				Define
				\[
				Y_R(t):=w(t)+\frac1\alpha w_t(t),
				\qquad t\in[t_i,t_i^*].
				\]
				Since \(w_t\ge0\) on this interval, \(0\le w\le Y_R\), and hence
				\begin{align*}
					Y_R'(t)
					&= w_t(t)+\frac1\alpha w_{tt}(t) \\
					&\le w_t(t)+\alpha w(t)+C e^{-t}w(t) \\
					&\le (\alpha+C e^{-t})Y_R(t).
				\end{align*}
				The Gronwall's inequality gives
				\[
				Y_R(t)
				\le Y_R(t_i)\exp\left(
				\alpha(t-t_i)+C\int_{t_i}^t e^{-s}ds
				\right)
				\le C w_i e^{\alpha(t-t_i)}.
				\]
				Since \(w(t_i^*)=\varepsilon_0\) and \(w\le Y_R\), it follows that
				\[
				\varepsilon_0\le C w_i e^{\alpha(t_i^*-t_i)}.
				\]
				Therefore
				\begin{equation}\label{eq:lb_right}
					t_i^*- t_i \ge \frac1\alpha \ln\frac1{w_i} - C.
				\end{equation}

				\noindent\emph{Left interval \([\bar t_i,t_i]\) -- lower bound.}
				Reverse time by setting
				\[
				W_i(\tau):=w(t_i-\tau),
				\qquad 0\le \tau\le t_i-\bar t_i.
				\]
				Then \(W_i'(\tau)=-w_t(t_i-\tau)\ge0\), and
				\[
				W_i''(\tau)=w_{tt}(t_i-\tau)
				\le (\alpha^2+C e^{-(t_i-\tau)})W_i(\tau).
				\]
				Define
				\[
				Y_L(\tau):=W_i(\tau)+\frac1\alpha W_i'(\tau).
				\]
				Since \(0\le W_i\le Y_L\),
				\begin{align*}
					Y_L'(\tau)
					&= W_i'(\tau)+\frac1\alpha W_i''(\tau) \\
					&\le W_i'(\tau)+\alpha W_i(\tau)
					+C e^{-(t_i-\tau)}W_i(\tau) \\
					&\le (\alpha+C e^{-(t_i-\tau)})Y_L(\tau).
				\end{align*}
				Hence
				\begin{align*}
					Y_L(\tau)
					&\le Y_L(0)\exp\left(
					\alpha\tau+C\int_0^\tau e^{-(t_i-s)}ds
					\right) \\
					&\le C w_i e^{\alpha\tau}.
				\end{align*}
				This follows from
				\[
				\int_0^\tau e^{-(t_i-s)}ds
				=e^{-(t_i-\tau)}-e^{-t_i}
				\le e^{-\bar t_i}\le C.
				\]
				Taking \(\tau=t_i-\bar t_i\) and using
				\(W_i(t_i-\bar t_i)=w(\bar t_i)=\varepsilon_0\), one obtains
				\begin{equation}\label{eq:lb_left}
					t_i - \bar t_i \ge \frac1\alpha \ln\frac1{w_i} - C.
				\end{equation}

				Subtract \eqref{eq:lb_left} from \eqref{eq:ub_right}. The divergent
				main terms \(\frac1\alpha\ln\frac1{w_i}\) cancel exactly, leaving
				\[
				t_i^*- 2t_i + \bar t_i
				\le C e^{-t_i}\ln\frac1{w_i}+C.
				\]
				By \eqref{eq:lb_left},
				\[
				\ln\frac1{w_i}\le C(t_i-\bar t_i+1).
				\]
				Writing \(s_i=t_i-\bar t_i\ge0\), one has
				\[
				e^{-t_i}\ln\frac1{w_i}
				\le C e^{-\bar t_i}e^{-s_i}(s_i+1)\le C.
				\]
				Therefore
				\begin{equation}\label{eq:linear}
					t_i^*\le 2t_i - \bar t_i + C.
				\end{equation}
				Combining the two lower bounds \eqref{eq:lb_right} and \eqref{eq:lb_left}
				\begin{equation}\label{eq:total_lb}
					t_i^*- \bar t_i \ge \frac{2}{\alpha} \ln\frac1{w_i} - C.
				\end{equation}

				It remains to evaluate the Pohozaev identity at \(r_i = e^{-t_i}\). Recall that
				\begin{equation*}
					P(r_i,u) = \mathcal{P}(u) - E_{\rm vol}(r_i) - I_h(r_i).
				\end{equation*}
				We uses the robust rescaling argument already established in Lemma~\ref{lem:pohozaev_limit}. At the local minimum \(r_i\), the rescaled functions \(v_i(y) = \frac{u(r_i y)}{u(r_i e_1)}\) converge in \(C^2_{\rm loc}(\overline{\mathbb{R}^n_+}\setminus\{0\})\) to the limit profile \(v(y) = \frac{1}{2}|y|^{2-n} + \frac{1}{2}\).

				By the exact asymptotic expansion evaluated along this critical sequence, the Pohozaev integral is dominated by the limit profile:
				\begin{equation*}
					P(r_i,u)
					=(w_i^*)^2\left(P_{\rm flat}(1,v_i)+o(1)\right)
					=-c_0(w_i^*)^2+o\bigl((w_i^*)^2\bigr),
				\end{equation*}
				where \(c_0>0\) is a positive dimensional constant. The Harnack inequality gives \(w_i^*\asymp w_i\). Hence, after decreasing \(c_0\) if necessary, for sufficiently large \(i\),
				\begin{equation*}
					P(r_i,u) \le -\frac{c_0}{2} w_i^2.
				\end{equation*}
				By hypothesis \(\mathcal{P}(u) \ge 0\), it follows the fundamental balance
				\begin{equation*}
					\frac{c_0}{2} w_i^2 \le -P(r_i,u) = -\mathcal{P}(u) + E_{\rm vol}(r_i) + I_h(r_i) \le |E_{\rm vol}(r_i)| + |I_h(r_i)|.
				\end{equation*}

				By standard pointwise estimates on the error operators, the volume and boundary errors are separated as the volume and boundary error integrals
				\begin{equation*}
					|E_{\rm vol}(r_i)| \le C \int_{\mathcal{B}_{r_i}^+} |X| |\mathcal E_{\rm int}|  dx, \qquad
					|I_h(r_i)| \le C \int_{\partial'\mathcal{B}_{r_i}^+} |X| |\mathcal E_{\rm bdy}|  dx'.
				\end{equation*}
				Both domains of integration are split at the intermediate radius \(r_i^* = e^{-t_i^*}\), defining the core region (\(|x| \le r_i^*\)) and the safe region (\(r_i^*< |x| \le r_i\))
				\begin{equation*}
					|E_{\rm vol}(r_i)| + |I_h(r_i)| \le (E_{\rm core} + I_{\rm core}) + (E_{\rm safe} + I_{\rm safe}).
				\end{equation*}

				\noindent\emph{The core.}
				Using the global upper bounds \(u(x)\le C|x|^{-\alpha}\), \(|\nabla u(x)|\le C|x|^{-\alpha-1}\), and \(|\nabla^2 u(x)| \le C|x|^{-\alpha-2}\), the uniform bounds are \(|X| \le C|x|^{-\alpha}\), \(|\mathcal E_{\rm int}| \le C|x|^{-\alpha-1}\), and \(|\mathcal E_{\rm bdy}| \le C|x|^{-\alpha}\).
				For the volume core integral
				\begin{equation*}
					E_{\rm core} \le C \int_0^{r_i^*} \rho^{-\alpha} \cdot \rho^{-\alpha-1} \cdot \rho^{n-1}d\rho = C \int_0^{r_i^*} \rho^{n-2\alpha-2} d\rho = C \int_0^{r_i^*} 1   d\rho = C e^{-t_i^*},
				\end{equation*}
				where the critical relation \(2\alpha=n-2\) has been used.
				For the flat boundary core integral
				\begin{equation*}
					I_{\rm core} \le C \int_0^{r_i^*} \rho^{-\alpha} \cdot \rho^{-\alpha} \cdot \rho^{n-2}d\rho = C \int_0^{r_i^*} \rho^{n-2\alpha-2} d\rho = C e^{-t_i^*}.
				\end{equation*}

				\noindent\emph{The safe zone.}
				In the safe zone the estimates are refined using the ODE estimates. For \(t \in [t_i, t_i^*]\) (corresponding to the annulus \(r_i^*\le |x| \le r_i\)), the Harnack inequality yields \(u(x) \asymp w(t) r^{-\alpha}\) and \(|\nabla u(x)| \le C w(t) r^{-\alpha-1}\).
				Consequently, \(|X| \le C w(t) r^{-\alpha}\), \(|\mathcal E_{\rm int}| \le C w(t) r^{-\alpha-1}\), and \(|\mathcal E_{\rm bdy}| \le C w(t) r^{-\alpha}\).
				Evaluating the volume safe integral with the substitution \(dr = -e^{-t}dt\)
				\begin{equation*}
					E_{\rm safe} \le C \int_{r_i^*}^{r_i} w(t)^2 r^{-2\alpha-1} r^{n-1} dr = C \int_{r_i^*}^{r_i} w(t)^2 dr = C \int_{t_i}^{t_i^*} w(t)^2 e^{-t} dt.
				\end{equation*}
				Similarly, evaluating the boundary safe integral over the flat annulus
				\begin{equation*}
					I_{\rm safe} \le C \int_{r_i^*}^{r_i} w(t)^2 r^{-2\alpha} r^{n-2} dr = C \int_{r_i^*}^{r_i} w(t)^2 dr = C \int_{t_i}^{t_i^*} w(t)^2 e^{-t} dt.
				\end{equation*}

				The first-order comparison used in the proof of \eqref{eq:lb_right} gives, for \(t\in[t_i,t_i^*]\),
				\[
				w(t)\le C w_i e^{\alpha(t-t_i)}.
				\]
				Substituting this into the combined safe zone integral yields
				\begin{align*}
					E_{\rm safe} + I_{\rm safe}
					& \le C \int_{t_i}^{t_i^*} w(t)^2 e^{-t} dt \le C w_i^2 \int_{t_i}^{t_i^*} e^{2\alpha(t-t_i)} e^{-t} dt \\
					& = C w_i^2 e^{-2\alpha t_i} \int_{t_i}^{t_i^*} e^{(2\alpha-1)t} dt.
				\end{align*}

				For \(n=3\), \(\alpha=1/2\), the exponent inside the integral is zero, and hence
				\begin{equation*}
					E_{\rm safe}+I_{\rm safe}
					\le C w_i^2 e^{-t_i}(t_i^*-t_i).
				\end{equation*}
				It remains only to absorb the logarithmic length factor. From the right
				upper bound \eqref{eq:ub_right},
				\begin{equation*}
					t_i^*-t_i
					\le \frac1\alpha\log\frac1{w_i}
					+ C e^{-t_i}\log\frac1{w_i}+C.
				\end{equation*}
				On the other hand, the left lower bound \eqref{eq:lb_left} gives directly
				\begin{equation*}
					\log\frac1{w_i}\le C(t_i-\bar t_i+1).
				\end{equation*}
				Consequently,
				\begin{equation*}
					t_i^*-t_i \le C(t_i-\bar t_i+1).
				\end{equation*}
				Writing \(s_i=t_i-\bar t_i\ge0\) gives
				\begin{equation*}
					e^{-t_i}(t_i^*-t_i)
					\le C e^{-\bar t_i}e^{-s_i}(s_i+1)
					= o(1),
				\end{equation*}
				since \((s+1)e^{-s}\) is bounded on \([0,\infty)\) and
				\(\bar t_i\to\infty\). Therefore, we have
				\begin{equation*}
					E_{\rm safe}+I_{\rm safe}
					\le o(1)w_i^2.
				\end{equation*}
				For \(n=4\), \(\alpha=1\), the integral evaluates to
				\begin{equation*}
					E_{\rm safe} + I_{\rm safe} \le C w_i^2 e^{-2t_i} e^{t_i^*} = C w_i^2 e^{t_i^*-2t_i} \le C w_i^2 e^{-\bar t_i} = o(1) w_i^2,
				\end{equation*}
				where the proof uses the linear cut-off \eqref{eq:linear} and the fact that \(\bar t_i\to\infty\).
				Thus the total safe zone error is strictly absorbed into the left-hand side \(\frac{c_0}{2} w_i^2\) for large \(i\).

				Combining the estimates gives
				\begin{equation*}
					w_i^2 \le C e^{-t_i^*}.
				\end{equation*}
				Taking logarithms,
				\begin{equation*}
					2\ln\frac1{w_i} \ge t_i^*-C.
				\end{equation*}
				Insert this lower bound into the total length inequality \eqref{eq:total_lb}
				\begin{equation*}
					t_i^*-\bar t_i \ge \frac{2}{\alpha}\ln\frac1{w_i} - C
					\ge \frac{1}{\alpha}\bigl(t_i^*-C\bigr) - C
				\end{equation*}
				and rearrange
				\begin{equation*}
					\bar t_i \le t_i^*\Bigl(1 - \frac1\alpha\Bigr) + C.
				\end{equation*}

				Evaluate for the two relevant dimensions.
				\begin{itemize}
					\item For \(n=3\): \(\alpha = \frac12\), hence \(1 - \frac1\alpha = -1\). The inequality forces
					\(\bar t_i \le -t_i^*+C \le C\).
					\item For \(n=4\): \(\alpha = 1\), hence \(1 - \frac1\alpha = 0\). The inequality gives
					\(\bar t_i \le C\).
				\end{itemize}
				In both cases \(\bar t_i \le C\) uniformly, which contradicts the fact that \(\bar t_i\to\infty\).
				This contradiction shows that the original assumption was false.
				Hence
				\begin{equation*}
					\limsup_{x\to0} u(x)|x|^{\frac{n-2}{2}} = 0,
				\end{equation*}
				completing the proof.
			\end{proof}

			\begin{lem}\label{lem:removable_criterion}
				Let \(u\) be a positive solution of the boundary Yamabe problem
				\eqref{eq:Main_PDE} in \(\mathcal{B}_{r_0}^+\setminus\{0\}\) with
				\(n\in\{3,4\}\).  If
				\begin{equation*}
					\lim_{x\to0}|x|^{\frac{n-2}{2}}u(x)=0,
				\end{equation*}
				then the singularity at the origin is removable; i.e.\ \(u\) extends to a
				smooth solution in the whole half-ball \(\mathcal{B}_{r_0}^+\).
			\end{lem}

			\begin{proof}
				Set
				\begin{equation*}
					\bar u(r):=\frac{1}{\omega_n' r^{n-1}}\int_{\partial''\mathcal{B}_r^+} u d\sigma,
					\qquad \omega_n':=|\partial''\mathcal{B}_1^+|,
				\end{equation*}
				The Harnack inequality \eqref{eq:Harnack} implies
				\begin{equation}\label{eq:pointwise}
					\Lambda^{-1}\bar u(r)\le u(x)\le \Lambda\bar u(r),\qquad
					|x|=r,\;x\in\overline{\mathcal{B}_r^+}
				\end{equation}
				for all sufficiently small \(r>0\).

				With the same error notation, the equation and boundary condition are written as
				\begin{equation*}
					-\Delta u=\mathcal E_{\rm int},
					\qquad
					-\partial_n u=u^{\frac n{n-2}}+\mathcal E_{\rm bdy}
					\quad\text{on }\{x_n=0\},
				\end{equation*}
				where, in the normalized coordinates,
				\begin{equation}\label{eq:errors}
					|\mathcal E_{\rm int}(x)|\le C|x|^{-1}u(x),
					\qquad
					|\mathcal E_{\rm bdy}(x)|\le C u(x).
				\end{equation}

				Integrating \(\Delta u = -\mathcal E_{\rm int}\) over \(\mathcal{B}_r^+\) and applying the divergence theorem give
				\begin{equation*}
					\int_{\partial''\mathcal{B}_r^+} \partial_r u d\sigma
					- \int_{\partial B_r\cap\{x_n=0\}} \partial_n u dx'
					= -\int_{\mathcal{B}_r^+} \mathcal E_{\rm int} dx.
				\end{equation*}
				Inserting the boundary condition on the flat part yields
				\begin{equation}\label{eq:integrated_avg_identity}
					\int_{\partial''\mathcal{B}_r^+} \partial_r u\,d\sigma
					+ \int_{\partial'\mathcal{B}_r^+} u^{\frac{n}{n-2}}\,dx'
					=
					-\int_{\mathcal{B}_r^+} \mathcal E_{\rm int}\,dx
					-\int_{\partial'\mathcal{B}_r^+} \mathcal E_{\rm bdy}\,dx'.
				\end{equation}
				The first term is precisely \(\omega_n' r^{n-1}\bar u'(r)\).

				Differentiating \eqref{eq:integrated_avg_identity} with respect to \(r\) gives
				\begin{equation}\label{eq:barODE}
					\frac{d}{dr}\bigl(r^{n-1}\bar u'(r)\bigr)
					+\frac{1}{\omega_n'}\int_{\partial B_r\cap\{x_n=0\}} u^{\frac{n}{n-2}}ds'
					= H(r),
				\end{equation}
				where
				\begin{equation*}
					H(r):=-\frac{1}{\omega_n'}\frac{d}{dr}\left(\int_{\mathcal{B}_r^+}
					\mathcal E_{\rm int} dx +\int_{\partial'\mathcal{B}_r^+}
					\mathcal E_{\rm bdy} ds'\right).
				\end{equation*}

				Using \eqref{eq:errors}, \eqref{eq:pointwise}, and the fact that the surface area of the equator is \(\omega_{n-2}r^{n-2}\), one obtains
				\begin{equation*}
					|H(r)|\le C\left(r^{n-1}\,r^{-1}\bar u(r)+r^{n-2}\bar u(r)\right)\le C r^{n-2}\bar u(r).
				\end{equation*}

				The same Harnack inequality gives
				\begin{equation}\label{eq:boundary_nonlinear_average}
					c_0 r^{n-2}\bar u(r)^{\frac{n}{n-2}}
					\le \int_{\partial B_r\cap\{x_n=0\}} u^{\frac{n}{n-2}}ds'
					\le C_0 r^{n-2}\bar u(r)^{\frac{n}{n-2}}.
				\end{equation}

				Combining \eqref{eq:barODE}--\eqref{eq:boundary_nonlinear_average} gives
				\begin{equation*}
					\frac{d}{dr}\bigl(r^{n-1}\bar u'(r)\bigr)
					\le -c_1 r^{n-2}\bar u(r)^{\frac{n}{n-2}} + C r^{n-2}\bar u(r),
				\end{equation*}
				where \(c_1,C>0\) depend only on \(n,\Lambda\) and the metric.

				\medskip
				\noindent\textbf{The \(C^0\) estimate.}

				Introduce the cylindrical variables
				\begin{equation*}
					t=-\ln r,\qquad
					\alpha=\frac{n-2}{2},
					\qquad
					w(t)=e^{-\alpha t}\bar u(e^{-t})=r^\alpha\bar u(r).
				\end{equation*}
				A direct computation gives
				\begin{equation*}
					\bar u(r)=r^{-\alpha}w(t),
					\qquad
					\bar u'(r)=-r^{-\alpha-1}\bigl(w_t(t)+\alpha w(t)\bigr),
				\end{equation*}
				and hence
				\begin{equation*}
					\frac{d}{dr}\bigl(r^{n-1}\bar u'(r)\bigr)
					=r^{\frac{n-4}{2}}\bigl(w_{tt}(t)-\alpha^2w(t)\bigr).
				\end{equation*}

				Moreover,
				\begin{align*}
					r^{n-2}\bar u(r)^{\frac n{n-2}}
					&=r^{n-2-\alpha\frac n{n-2}}w(t)^{\frac n{n-2}}
					=r^{\frac{n-4}{2}}w(t)^{\frac n{n-2}},\\
					r^{n-2}\bar u(r)
					&=r^{\frac{n-4}{2}}e^{-t}w(t).
				\end{align*}
				Thus \eqref{eq:barODE} is equivalent to
				\begin{equation}\label{eq:asymp_ODE}
					w_{tt}(t)-\alpha^2w(t)=h(t)w(t),
					\qquad t\ge T_0,
				\end{equation}
				where
				\begin{equation}\label{eq:h_potential_bound}
					h(t)=-\beta(t)w(t)^{\frac2{n-2}}+O(e^{-t}),
					\qquad
					0<c_0\le \beta(t)\le C_0.
				\end{equation}
				The assumption \(\lim_{x\to0}|x|^\alpha u(x)=0\), together with \eqref{eq:pointwise}, gives \(\lim_{t\to\infty}w(t)=0\), which immediately implies the vanishing of the potential, namely, \(\lim_{t\to\infty}h(t)=0\).

				The estimate for \(w\) is obtained in two stages: first a suboptimal exponential decay is derived from the maximum principle, and then the sharp critical decay rate \(\alpha\) is recovered from the exact Green representation.

				Let \(0<\varepsilon<\alpha^2\) and choose \(T_\varepsilon\ge T_0\) such that
				\(|h(t)|\le\varepsilon\) for all \(t\ge T_\varepsilon\). Then
				\begin{equation*}
					w_{tt}(t)\ge (\alpha^2-\varepsilon)w(t),
					\qquad t\ge T_\varepsilon.
				\end{equation*}
				With \(\mu=\sqrt{\alpha^2-\varepsilon}\), fix \(S>T_\varepsilon\) and set
				\[
				\Phi_S(t)
				:=w(T_\varepsilon)e^{-\mu(t-T_\varepsilon)}
				+w(S)e^{-\mu(S-t)},
				\qquad T_\varepsilon\le t\le S.
				\]
				Then \(\Phi_S''=\mu^2\Phi_S\), and the endpoint comparisons are
				\[
				\Phi_S(T_\varepsilon)\ge w(T_\varepsilon),
				\qquad
				\Phi_S(S)\ge w(S).
				\]
				For \(z_S=w-\Phi_S\), one has
				\[
				z_S''-\mu^2z_S\ge0
				\quad\text{on }[T_\varepsilon,S],
				\qquad
				z_S(T_\varepsilon)\le0,\quad z_S(S)\le0.
				\]
				The maximum principle gives \(z_S\le0\) on
				\([T_\varepsilon,S]\). Therefore, for \(T_\varepsilon\le t\le S\),
				\[
				w(t)
				\le w(T_\varepsilon)e^{-\mu(t-T_\varepsilon)}
				+w(S)e^{-\mu(S-t)}.
				\]
				Fixing \(t\) and letting \(S\to\infty\), the second term tends to zero
				because \(w(S)\to0\). Hence
				\begin{equation}\label{eq:first_exp_decay}
					0<w(t)\le C_\varepsilon e^{-\mu t},
					\qquad t\ge T_\varepsilon.
				\end{equation}

				Variation of parameters applied to \eqref{eq:asymp_ODE} yields
				\begin{equation}\label{eq:integral_rep}
					w(t)=B e^{-\alpha t}
					-\frac1{2\alpha}\int_t^\infty e^{\alpha(t-s)}h(s)w(s)\,ds
					-\frac1{2\alpha}\int_{T_0}^t e^{-\alpha(t-s)}h(s)w(s)\,ds.
				\end{equation}
				The growing homogeneous mode is absent because \(w(t)\to0\). From
				\eqref{eq:h_potential_bound} and \eqref{eq:first_exp_decay},
				\begin{equation*}
					|h(s)w(s)|
					\le Cw(s)^{\frac n{n-2}}+Ce^{-s}w(s)
					\le C e^{-\sigma s},
				\end{equation*}
				where \(\sigma = \min(\mu p, 1+\mu)\). Choosing \(\varepsilon\) sufficiently small makes \(\mu\) arbitrarily close to \(\alpha\), hence \(\sigma > \alpha\). Therefore
				\begin{align*}
					\int_t^\infty e^{\alpha(t-s)} |h(s) w(s)|   ds \le C e^{\alpha t} \int_t^\infty e^{-(\alpha+\sigma)s}   ds = \frac{C}{\alpha+\sigma} e^{-\sigma t},\\
					\int_{T_0}^t e^{-\alpha(t-s)} |h(s) w(s)|   ds \le C e^{-\alpha t} \int_{T_0}^t e^{(\alpha-\sigma)s}   ds \le C' e^{-\alpha t}.
				\end{align*}
				Substitution into \eqref{eq:integral_rep} gives
				\begin{equation*}
					w(t) \le B e^{-\alpha t} + \tilde{C} e^{-\sigma t} + C' e^{-\alpha t} \le C e^{-\alpha t}.
				\end{equation*}
				Hence
				\begin{equation*}
					\bar u(r)=r^{-\alpha}w(-\log r)\le C,
					\qquad 0<r<r_0/2,
				\end{equation*}
				Coupling this with the pointwise equivalence \(u(x) \asymp \bar{u}(|x|)\), yields the uniform \(C^0\) bound near the origin
				\begin{equation*}
					u(x) \le C, \qquad \text{for } 0 < |x| < \frac{r_0}{2}.
				\end{equation*}

				The scale-invariant gradient estimate from Corollary~\ref{cor:gradient}, applied
				with exponent \(\beta=0\), gives
				\begin{equation}\label{eq:bounded_grad_estimate}
					|\nabla u(x)|\le C|x|^{-1},
					\qquad 0<|x|<r_0/2.
				\end{equation}
				Thus \(\nabla_g u\in L^2(\mathcal B_{r_0/2}^+)\) by \(n\ge3\).

				It remains to pass through the isolated point in the weak formulation. For every
				\(\varphi\in C_c^\infty(\overline{\mathcal B_{r_0/2}^+})\), the weak formulation is
				\begin{equation*}
					\int_{\mathcal B_{r_0/2}^+}
					\bigl(\langle\nabla_g u,\nabla_g\varphi\rangle_g+c(n)R_g u\varphi\bigr) dV_g
					+\int_{\partial'\mathcal B_{r_0/2}^+}\frac{n-2}{2}h_g u\varphi d\sigma_g
					=\int_{\partial'\mathcal B_{r_0/2}^+}u^{\frac n{n-2}}\varphi d\sigma_g.
				\end{equation*}
				It is enough to prove this identity by cutting out the singularity. Set
				\begin{equation*}
					\Omega_\varepsilon:=\mathcal B_{r_0/2}^+\setminus \overline{\mathcal B_\varepsilon^+}.
				\end{equation*}
				Using \(-\Delta_g u+c(n)R_gu=0\) in \(\Omega_\varepsilon\), Green's first identity gives
				\begin{equation*}
					\int_{\Omega_\varepsilon}
					\bigl(\langle\nabla_g u,\nabla_g\varphi\rangle_g+c(n)R_g u\varphi\bigr) dV_g=\int_{\partial\Omega_\varepsilon}\varphi \partial_{\nu_\Omega}u d\sigma_g
				\end{equation*}
				On the flat boundary,
				\begin{equation*}
					\partial_{\nu_g}u=u^{\frac n{n-2}}-\frac{n-2}{2}h_gu.
				\end{equation*}
				Therefore
				\begin{equation*}
					\begin{aligned}
						&\int_{\Omega_\varepsilon}\bigl(\langle\nabla_g u,\nabla_g\varphi\rangle_g
						+c(n)R_g u\varphi\bigr) dV_g+\int_{\partial'\mathcal B_{r_0/2}^+\setminus \mathcal B_\varepsilon}\frac{n-2}{2}h_g u\varphi d\sigma_g  \\
						&\qquad =\int_{\partial'\mathcal B_{r_0/2}^+\setminus \mathcal B_\varepsilon}u^{\frac n{n-2}}\varphi d\sigma_g+\mathcal I_\varepsilon,
					\end{aligned}
				\end{equation*}
				where
				\begin{equation*}
					\mathcal I_\varepsilon=-\int_{\partial''\mathcal B_\varepsilon^+}\varphi \frac{\partial u}{\partial \nu_\varepsilon} d\sigma_g
				\end{equation*}
				and here \(\nu_\varepsilon\) is the outward unit normal of \(\mathcal B_\varepsilon^+\).
				By \eqref{eq:bounded_grad_estimate},
				\begin{equation*}
					|\mathcal I_\varepsilon|
					\le C\|\varphi\|_{L^\infty}\varepsilon^{-1}\varepsilon^{n-1}
					=C\|\varphi\|_{L^\infty}\varepsilon^{n-2}\to0.
				\end{equation*}
				Letting \(\varepsilon\to0\) gives the weak boundary Yamabe equation on
				\(\mathcal B_{r_0/2}^+\).

				Since \(u\in L^\infty\), the boundary datum \(u^{n/(n-2)}\) is bounded. Standard local boundary regularity for Robin boundary conditions gives
				\(u\in W^{2,p}_{\rm loc}(\overline{\mathcal B_{r_0/2}^+})\) for every \(p<\infty\). Taking \(p>n\) gives \(u\in C^{1,\alpha}\). Since \(n=3,4\), the boundary nonlinearity \(u^{n/(n-2)}\) is smooth in \(u\), and Schauder bootstrapping yields \(u\in C^\infty(\overline{\mathcal B_{r_0/4}^+})\). Thus the singularity at the origin is smoothly removable.
			\end{proof}
		\subsubsection{Behavior of non-removable singularities}
		Theorem~\ref{thm:upper_bound}, Lemma~\ref{lem:removable}, and Lemma~\ref{lem:removable_criterion} imply that any non-removable singularity satisfies the two-sided estimate
		\begin{equation}\label{eq:two_sided_bound}
			c_0|x|^{\frac{2-n}{2}}\le u(x)\le C_0|x|^{\frac{2-n}{2}}\qquad\text{for }0<|x|\le r_0,
		\end{equation}
		with constants \(0<c_0\le C_0<\infty\). Together with the scale-invariant local gradient estimate, it also gives
		\begin{equation}\label{eq:sec3_log_derivative_bound}
			|\nabla\log u(x)|\le C|x|^{-1}
			\qquad\text{for }0<|x|\le r_0 .
		\end{equation}
		This estimate will be improved below to obtain the asymptotic cylindrical symmetry
		of non-removable singularities.

		\begin{lem}\label{lem:asymptotic_symmetry}
			Let \(u\) be a positive solution of the boundary Yamabe equation with an isolated non-removable singularity at the origin. Then there exist constants \(C>0\) and \(\sigma\in(0,1)\) such that for all \(x \in \mathcal{B}_{r_0/2}^+ \setminus \{0\}\),
			\[
			|\nabla_\tau\log u(x)|\le C|x|^{\sigma-1}
			\]
			for every horizontal unit vector \(\tau\) orthogonal to \(x'\). Consequently,
			\[
			u(x)=\bar u(|x'|,x_n)(1+O(|x|^\sigma)),
			\]
			where \(\bar u(|x'|,x_n)\) denotes the average of \(u\) over the horizontal sphere
			\(\{(y',x_n): |y'|=|x'|\}\).
		\end{lem}

		\begin{proof}
			Set \(\alpha_0:=(n-2)/2\). Fix
			\(x=(x',x_n)\in\mathcal{B}_{r_0/2}^+\setminus\{0\}\). Let
			\(\tau\in\mathbb{R}^{n-1}\times\{0\}\) be a horizontal unit vector such thaty \(\tau\cdot x'=0\). If \(x'=0\), any horizontal
			unit vector may be chosen. The moving-sphere centre is taken to be
			\(x_0=(x'-R\tau,0)\in\partial'\mathbb{R}^n_+\), where the auxiliary scale is
			\(R=|x|^\beta\) for a parameter \(\beta\in(0,1)\) to be optimized.

			Set
			$a=|x_0|$ and $ r=|x-x_0|$. Since \(\tau\perp x'\), we have
			\[
			a^2=|x'|^2+R^2,
			\qquad
			r^2=|R\tau+x_ne_n|^2=R^2+x_n^2.
			\]
			Consequently,
			\begin{equation*}
				|a-r|=\frac{|a^2-r^2|}{a+r}\le C\frac{|x|^2}{R}.
			\end{equation*}

			We shall use the following perturbed moving-spheres comparison.

			\begin{claim}\label{claim:ms_input}
				Let \(0<\theta_*<\frac14\) and
				\[
				\Sigma_\lambda
				=\{y\in\mathbb R^n_+:\lambda<|y-x_0|<Ka\}\setminus\{0\},
				\]
				where \(K\gg1\) is fixed. There exist a universal \(\gamma\in(0,1)\)
				and a constant \(C_0>0\) such that, for every
				\(\theta_*a\le\lambda\le a(1-C_0a^\gamma)\), one can find a negative
				regular barrier \(H_\lambda\le0\) and a nonnegative transition barrier
				\(\Psi_\lambda\ge0\) for which the reflected function
				\(u_{x_0,\lambda}\) satisfies
				\begin{equation}\label{eq:sec3_ms_input_used}
					u_{x_0,\lambda}(y) \le u(y) + |H_\lambda(y)| + \Psi_\lambda(y)
					\qquad \text{for all } y \in \Sigma_\lambda.
				\end{equation}
				Here
				\[
				u_{x_0,\lambda}(y)=\left(\frac{\lambda}{|y-x_0|}\right)^{n-2}
				u\left(x_0+\frac{\lambda^2(y-x_0)}{|y-x_0|^2}\right),
				\]
				and the barriers satisfy \(H_\lambda=\Psi_\lambda=0\) on \(|y-x_0|=\lambda\). Moreover,
				with \(r_y:=|y-x_0|\),
				\begin{equation}\label{eq:sec3_ms_barrier_eval_input}
					|H_\lambda(y)|+\Psi_\lambda(y)
					\le C(r_y-\lambda)|y|^{-\alpha_0}
					\qquad\text{for }y\in\Sigma_\lambda.
				\end{equation}
			\end{claim}

			Assume Claim~\ref{claim:ms_input}. We choose the inversion radius in the form
			\[
			\lambda=r(1-\varepsilon).
			\]
			To apply the claim, it is enough to ensure
			\[
			r(1-\varepsilon)\le a(1-C_0a^\gamma).
			\]
			Since \(r-a\le C|x|^2/R\), this condition follows from
			\[
			r\varepsilon\ge r-a+C_0a^{1+\gamma}.
			\] Thus, choosing
			\begin{equation}\label{eq:sec3_epsilon_choice}
				\varepsilon=C_2\left(\frac{|x|^2}{R^2}+R^\gamma\right),
			\end{equation}
			with \(C_2\) sufficiently large. For \(|x|\) sufficiently small, it also gives \(\varepsilon<1/4\) and
			\(\lambda\ge a/2\). Hence \eqref{eq:sec3_ms_input_used} can be applied at \(y=x\).

			Expanding the Kelvin transform \(u_{x_0,\lambda}(x)=(1-\varepsilon)^{n-2}u(x^\lambda)\) around \(x\), the reflected point is
			\[
			x^\lambda=x_0+\frac{\lambda^2(x-x_0)}{|x-x_0|^2}
			=x-2\varepsilon(x-x_0)+O(\varepsilon^2R).
			\]
			With the choice of \(R\) made below, one has \(\varepsilon R=o(|x|)\). Hence the Taylor segment remains in an annulus comparable to \(|x|\). With
			\[
			\left(\frac{\lambda}{r}\right)^{n-2}=1-(n-2)\varepsilon+O(\varepsilon^2),
			\]
			one has
			\[
			u_{x_0,\lambda}(x)
			=u(x)-(n-2)\varepsilon u(x)
			-2\varepsilon\nabla u(x)\cdot(x-x_0)
			+O(\varepsilon^2R^2|x|^{-2}u(x)).
			\]
			Substitution into \eqref{eq:sec3_ms_input_used} gives
			\begin{equation}\label{eq:sec3_taylor_difference_quotient}
				\frac{n-2}{2}u(x)+\nabla u(x)\cdot(x-x_0)
				\ge
				-C\varepsilon R^2|x|^{-2}u(x)
				-\frac{|H_\lambda(x)|+\Psi_\lambda(x)}{2\varepsilon}.
			\end{equation}

			At the point \(y=x\), one has \(|x-x_0|-\lambda=\varepsilon r\sim\varepsilon R\).
			The barrier estimate in Claim~\ref{claim:ms_input}, together with the lower
			bound in \eqref{eq:two_sided_bound}, therefore gives
			\begin{equation}\label{eq:sec3_barrier_eval_bound}
				|H_\lambda(x)|+\Psi_\lambda(x)
				\le C\varepsilon R|x|^{-\alpha_0}
				\le C\varepsilon R u(x).
			\end{equation}

			Writing \(x-x_0=R\tau+x_ne_n\), dividing
			\eqref{eq:sec3_taylor_difference_quotient} by \(Ru(x)\), using
			\eqref{eq:sec3_barrier_eval_bound} and \eqref{eq:sec3_log_derivative_bound}, yields
			\[
			\nabla_\tau\log u(x)
			\ge
			-C\left(R^{-1}+\varepsilon R|x|^{-2}\right).
			\]
			After division by \(Ru(x)\), the zeroth-order term and the normal derivative term are bounded by \(CR^{-1}\), while the barrier contribution is \(O(1)\), hence also \(O(R^{-1})\) since \(R<1\).
			Repeating the same argument with \(-\tau\) gives
			\[
			|\nabla_\tau\log u(x)|
			\le C\left(R^{-1}+\varepsilon R|x|^{-2}\right).
			\]
			Using \eqref{eq:sec3_epsilon_choice},
			\[
			|\nabla_\tau\log u(x)|
			\le C\left(R^{-1}+R^{1+\gamma}|x|^{-2}\right),
			\]
			since the contribution of \(|x|^2/R^2\) is again \(R^{-1}\).

			Balancing the error scales, namely, \(R^{-1}=R^{1+\gamma}|x|^{-2}\), gives \(R=|x|^{\frac{2}{2+\gamma}}\). With this radius, \(R^{-1}=|x|^{\sigma-1}\), where \(\sigma=\frac{\gamma}{2+\gamma}\in(0,1)\). Therefore
			\[
			|\nabla_\tau \log u(x)| \le C_4 |x|^{\sigma-1}.
			\]

			For \(x=(x',x_n)\) with \(x'\ne0\), consider the horizontal sphere
			\begin{equation*}
				\Sigma_x:=\{(y',x_n): |y'|=|x'|\}.
			\end{equation*}
			This is an \((n-2)\)-dimensional sphere inside the horizontal hyperplane \(\{y_n=x_n\}\), not an \((n-2)\)-dimensional slice of the ambient half-space. Any two points \(x_1,x_2\in\Sigma_x\) can be joined by a great-circle arc lying entirely in \(\Sigma_x\).
			Its tangent \(\tau\) is everywhere horizontal and orthogonal to the horizontal radial direction. Integrating the gradient bound along this arc gives
			\[
			|\log u(x_1)-\log u(x_2)|
			\le \int_{\text{arc}} |\nabla_\tau \log u|\,ds
			\le \pi C_4 |x|^\sigma.
			\]

			Exponentiating gives \(e^{-\pi C_4 |x|^{\sigma}} \le u(x_1)/u(x_2) \le e^{\pi C_4 |x|^{\sigma}}\). Hence \(u(x)=\bar u(|x'|,x_n)(1+O(|x|^\sigma))\), where \(\bar u(|x'|,x_n)\) denotes the average over \(\Sigma_x\). If \(x'=0\), the horizontal sphere degenerates to a single point, and the assertion is immediate.
		\end{proof}

		\subsection{The perturbed moving-sphere method: proof of Claim~\ref{claim:ms_input}}

		It remains to prove the perturbed moving-spheres comparison used above. The argument follows the same decomposition as in Step~4 of Section~\ref{sec:upper_bound}. The new ingredients are the Kelvin-transform error estimates and the transition barrier near the singular point.

		Throughout this subsection, let
		\[
		\alpha_0:=\frac{n-2}{2},
		\qquad p:=\frac n{n-2},
		\]
		and
		\[
		x_0\in\partial'\mathbb R^n_+,
		\qquad a:=|x_0|,
		\qquad r:=|y-x_0|.
		\]
		For \(\lambda>0\), define
		\[
		y^\lambda:=x_0+\frac{\lambda^2(y-x_0)}{|y-x_0|^2},
		\qquad
		u_{x_0,\lambda}(y):=\left(\frac{\lambda}{r}\right)^{n-2}u(y^\lambda).
		\]
		The comparison region is
		\[
		\Sigma_\lambda:=\{y\in\mathbb R^n_+:\lambda<|y-x_0|<Ka\}\setminus\{0\},
		\]
		where \(K\gg1\) is fixed once and for all.

		In the normalized coordinates, the equation may be written as
		\begin{equation*}
			-\Delta u=\mathcal E_{\rm int}
			\quad\text{in }\mathbb R^n_+,
			\qquad
			-\partial_n u=u^p+\mathcal E_{\rm bdy}
			\quad\text{on }\partial'\mathbb R^n_+.
		\end{equation*}
		The metric expansion \eqref{eq:metric_expansion}, the two-sided estimate, and the
		scale-invariant local estimates,
		\[
		|\nabla^m u(z)|\le C_m|z|^{-\alpha_0-m},\qquad m=0,1,2,3,
		\]
		give
		\begin{equation}\label{eq:sec3_error_derivative_bounds}
			|\mathcal E_{\rm int}(z)|\le C|z|^{-\frac n2},
			\qquad
			|\nabla \mathcal E_{\rm int}(z)|\le C|z|^{-\frac n2-1},
		\end{equation}
		and, on the flat boundary,
		\begin{equation}\label{eq:sec3_boundary_error_derivative_bounds}
			|\mathcal E_{\rm bdy}(z)|\le C|z|^{1-\frac n2},
			\qquad
			|\nabla' \mathcal E_{\rm bdy}(z)|\le C|z|^{-\frac n2}.
		\end{equation}
		Indeed,
		\[
		\mathcal E_{\rm int}=(g^{ij}-\delta^{ij})\partial_{ij}u
		+ (\partial_i g^{ij})\partial_ju+\text{lower-order terms},
		\]
		where \(g^{ij}-\delta^{ij}=O(|x|)\), \(\partial g=O(1)\), and \(R_g=O(|x|)\).
		On the boundary,
		\[
		\mathcal E_{\rm bdy}=-\mu^i\partial_i u-\frac{n-2}{2}h_gu,
		\]
		with \(\mu^i=O(|x|)\), \(\nabla'\mu^i=O(1)\), \(h_g=O(|x|)\), and \(\nabla'h_g=O(1)\).
		This proves \eqref{eq:sec3_error_derivative_bounds}--\eqref{eq:sec3_boundary_error_derivative_bounds}.

		\begin{lem}\label{lem:sec3_kelvin_error_estimates}
			Assume that, for a fixed \(\theta_*\in(0,1/4)\),
			\[
			\theta_*a\le\lambda<a,
			\qquad \lambda<r<Ka,
			\]
			and set \(d:=a-\lambda\). Define
			\[
			\mathcal R_{\rm int}(y)
			:=\mathcal E_{\rm int}(y)-\left(\frac{\lambda}{r}\right)^{n+2}\mathcal E_{\rm int}(y^\lambda),
			\]
			and, on \(\partial'\mathbb R^n_+\),
			\[
			\mathcal R_{\rm bdy}(y)
			:=\mathcal E_{\rm bdy}(y)-\left(\frac{\lambda}{r}\right)^n\mathcal E_{\rm bdy}(y^\lambda).
			\]
			Then, in the transition region \(\kappa d\le |y|\le c_0a\),
			\begin{equation*}
				|\mathcal R_{\rm int}(y)|
				\le C(r-\lambda)|y|^{-\frac n2-1},
				\qquad
				|\mathcal R_{\rm bdy}(y)|
				\le C(r-\lambda)|y|^{-\frac n2}.
			\end{equation*}
			In the regular region \(|y|\ge c_0a\),
			\begin{equation}\label{eq:sec3_regular_errors}
				|\mathcal R_{\rm int}(y)|\le Ca^{-\frac n2},
				\qquad
				|\mathcal R_{\rm bdy}(y)|\le Ca^{-\frac n2}(r-\lambda).
			\end{equation}
		\end{lem}

		\begin{proof}
			See Appendix~\ref{app:sec3_moving_sphere_tools}.
		\end{proof}

		\begin{lem}\label{lem:sec3_boundary_linear_coeff}
			On \(\partial'\mathbb R^n_+\), define the mean-value coefficient
			\[
			c_\lambda(y):=
			\begin{cases}
				\dfrac{u(y)^p-u_{x_0,\lambda}(y)^p}{u(y)-u_{x_0,\lambda}(y)},
				&u(y)\ne u_{x_0,\lambda}(y),\\[2.5mm]
				p u(y)^{p-1},
				&u(y)=u_{x_0,\lambda}(y).
			\end{cases}
			\]
			Then
			\begin{equation}\label{eq:sec3_boundary_coefficient_bound}
				0\le c_\lambda(y)\le C|y|^{-1}
				\qquad\text{on }\Sigma_\lambda\cap\partial'\mathbb R^n_+.
			\end{equation}
			The same bound applies to the coefficient obtained by the mean-value theorem on
			any negative-part support in the narrow-domain argument below.
		\end{lem}

		\begin{proof}
			Nonnegativity follows from the monotonicity of \(t\mapsto t^p\). The mean-value theorem gives
			\[
			c_\lambda(y)\le C\bigl(u(y)^{p-1}+u_{x_0,\lambda}(y)^{p-1}\bigr).
			\]
			Since \(p-1=2/(n-2)\), the two-sided estimate gives
			\[
			u(y)^{p-1}\le C|y|^{-1}.
			\]
			For the Kelvin term,
			\[
			u_{x_0,\lambda}(y)^{p-1}
			=\left(\frac{\lambda}{r}\right)^2u(y^\lambda)^{p-1}
			\le C\left(\frac{\lambda}{r}\right)^2|y^\lambda|^{-1}.
			\]
			Using the Kelvin lower geometry in Appendix~\ref{app:sec3_moving_sphere_tools},
			\[
			|y^\lambda|\ge C^{-1}\frac{\lambda}{r}|y|,
			\]
			and therefore
			\[
			u_{x_0,\lambda}(y)^{p-1}\le C|y|^{-1}.
			\]
			This proves \eqref{eq:sec3_boundary_coefficient_bound}.
		\end{proof}

		\begin{lem}\label{lem:dyadic_transition_barrier}
			Assume \(\theta_*a\le\lambda<a\), set \(d:=a-
			\lambda\), and fix
			\(0<\kappa,c_0\ll1\). In the transition region
			\[
			\mathcal T_\lambda:=\Sigma_\lambda\cap\{\kappa d\le |y|\le2c_0a\},
			\]
			the transition barrier \(\Psi_\lambda\) appearing in Claim~\ref{claim:ms_input}
			is constructed as a nonnegative function on \(\mathcal T_\lambda\) such that
			\(\Psi_\lambda=0\) on \(r=\lambda\). It is extended by zero outside
			\(\mathcal T_\lambda\) and the differential inequalities below are used only in the
			region where the barrier is constructed. In the core
			\(\mathcal T_\lambda^{\rm core}:=\Sigma_\lambda\cap\{\kappa d\le |y|\le c_0a\}\),
			\begin{equation*}
				-\Delta\Psi_\lambda\ge |\mathcal R_{\rm int}|,
			\end{equation*}
			while on the flat boundary of the core,
			\begin{equation*}
				-\partial_n\Psi_\lambda-c_\lambda\Psi_\lambda
				\ge |\mathcal R_{\rm bdy}|.
			\end{equation*}
			Moreover,
			\begin{equation*}
				0\le\Psi_\lambda(y)
				\le C(r-\lambda)|y|^{1-\frac n2}
				\qquad\text{in }\mathcal T_\lambda,
			\end{equation*}
			and, in the cutoff shell \(c_0a\le |y|\le2c_0a\), the rough localization
			estimates satisfy
			\[
			|\Delta\Psi_\lambda|\le Ca^{-\frac n2},
			\]
			and, on the flat boundary of this shell,
			\[
			\bigl|-\partial_n\Psi_\lambda-c_\lambda\Psi_\lambda\bigr|
			\le C(r-\lambda)a^{-\frac n2}.
			\]
		\end{lem}

		\begin{proof}
			See Appendix~\ref{app:sec3_moving_sphere_tools}.
		\end{proof}

		\begin{lem}\label{lem:sec3_combined_barrier}
			The regular barrier \(H_\lambda\) appearing in Claim~\ref{claim:ms_input} is
			chosen as
			\[
			H_\lambda:=h_1+h_2+h_3,
			\]
			where
			\[
			h_1:=-A_1a^{\alpha_0}\int_\lambda^r s^{1-n}(s-\lambda)\,ds,
			\]
			\[
			h_2:=-A_2a^{-\frac n2}y_n(r-\lambda),
			\qquad
			h_3:=-A_3a^{-\frac n2}(r^2-\lambda^2).
			\]
			After choosing \(A_1,A_2,A_3\) sufficiently large, the corrected barrier satisfies
			\begin{equation*}
				-\Delta(H_\lambda+\Psi_\lambda)
				\ge |\mathcal R_{\rm int}|
			\end{equation*}
			in \(\Sigma_\lambda\cap\{|y|\ge\kappa d\}\), and
			\begin{equation*}
				-\partial_n(H_\lambda+\Psi_\lambda)
				-c_\lambda(H_\lambda+\Psi_\lambda)
				\ge |\mathcal R_{\rm bdy}|
			\end{equation*}
			on its flat boundary. Moreover \(H_\lambda\le0\) and \(H_\lambda=0\) on \(r=\lambda\).
		\end{lem}

		\begin{proof}
			A direct computation gives
			\[
			-\Delta h_1=A_1a^{\alpha_0}r^{1-n},
			\qquad
			-\Delta h_2=(n+1)A_2a^{-\frac n2}\frac{y_n}{r}\ge0,
			\qquad
			-\Delta h_3=2nA_3a^{-\frac n2}.
			\]
			On \(\partial'\mathbb R^n_+\),
			\[
			-\partial_nh_1=0,
			\qquad
			-\partial_nh_2=A_2a^{-\frac n2}(r-\lambda),
			\qquad
			-\partial_nh_3=0.
			\]
			Since \(H_\lambda\le0\) and \(c_\lambda\ge0\), the Robin lower-order term
			\(-c_\lambda H_\lambda\) is favorable.

			In the regular region \(|y|\ge c_0a\), Lemma~\ref{lem:sec3_kelvin_error_estimates} gives \eqref{eq:sec3_regular_errors}. The
			interior terms above dominate \(Ca^{-n/2}\), and the boundary term
			\(A_2a^{-n/2}(r-\lambda)\) dominates the boundary error. In the transition core,
			Lemma~\ref{lem:dyadic_transition_barrier} gives the domination directly. In the
			cutoff shell \(c_0a\le|y|\le2c_0a\), the commutators generated by localizing
			\(\Psi_\lambda\) are bounded by \(Ca^{-n/2}\) in the interior and by
			\(C(r-\lambda)a^{-n/2}\) on the flat boundary; these are absorbed by increasing
			\(A_3\) and \(A_2\), respectively. This proves the lemma.
		\end{proof}

		\begin{proof}[Proof of Claim~\ref{claim:ms_input}]
			Choose \(K>1\) sufficiently large. On the outer boundary \(r=Ka\), the two-sided estimate gives
			\[
			u(y)\ge cK^{-\alpha_0}a^{-\alpha_0},
			\]
			while the Kelvin transform satisfies
			\[
			u_{x_0,\lambda}(y)\le CK^{2-n}a^{-\alpha_0}.
			\] Since \(2-n<-\alpha_0\), choosing \(K\) large enough yields
			\[
			u_{x_0,\lambda}(y)<u(y)
			\quad\text{on } r=Ka.
			\]
			The barrier terms are of smaller order on this boundary, so the comparison remains strict.

			We first start the moving spheres at a small proportional radius. Choose \(M>1\) sufficiently large, and then choose \(\theta_*>0\) sufficiently small compared with \(M^{-1}\). With this choice, the reflected term is controlled by \(u\) away from a narrow shell near \(r=\theta_*a\).
			More precisely, for
			\(\lambda_0:=\theta_*a\), the cap \(|y|\le\kappa a\) is controlled by the singular
			lower bound for \(u\), and the region \(|y|\ge\kappa a\), separated from the moving
			sphere, is controlled by the small factor \((\lambda_0/r)^{n-2}\).
			The remaining
			collar has width \(\eta a\) with \(\eta\) arbitrarily small.

			We apply the narrow-domain argument from Section~\ref{sec:upper_bound}, together with the boundary coefficient bound \eqref{eq:sec3_boundary_coefficient_bound}, to
			\[
			W_{\lambda_0}:=u-u_{x_0,\lambda_0}+H_{\lambda_0}+\Psi_{\lambda_0},
			\]
			this gives \(W_{\lambda_0}\ge0\) in
			\(\Sigma_{\lambda_0}\cap\{|y|\ge\kappa(a-\lambda_0)\}\). In the remaining cap,
			the direct comparison \(u\ge4u_{x_0,\lambda_0}\) and together with the fact that the barrier terms are lower order, gives \(W_{\lambda_0}\ge0\) on the cap boundary, and also imply the
			weaker pointwise estimate \eqref{eq:sec3_ms_input_used} in the cap.

			Define the admissible set
			\[
			\Lambda:=\{\mu\in[\lambda_0,a):
			W_\lambda\ge0\text{ in }
			\Sigma_\lambda\cap\{|y|\ge\kappa(a-\lambda)\}
			\text{ for every }\lambda\in[\lambda_0,\mu]\}.
			\]
			This set is nonempty. Let \(\bar\lambda:=\sup\Lambda\). It remains to show that
			\(a-\bar\lambda\le La^{1+\gamma}\). Suppose, to the contrary, that
			\[
			d_{\bar\lambda}:=a-\bar\lambda>La^{1+\gamma}.
			\]
			Choose \(\lambda\in\Lambda\) with \(\lambda<\bar\lambda\) sufficiently close to
			\(\bar\lambda\). Set \(d_\lambda:=a-\lambda\). Since
			\(d_{\bar\lambda}>La^{1+\gamma}\), for such \(\lambda\) one has
			\(d_\lambda\sim d_{\bar\lambda}\). Fix \(0<\theta\ll1\), choose \(0<\ell\le \theta d_\lambda\), and take \(\mu\) such that
			\[
			\bar\lambda<\mu<\lambda+\frac{\ell}{4}.
			\]
			This is possible as \(\lambda\) is chosen sufficiently close to
			\(\bar\lambda\). It is enough to prove that \(W_\nu\ge0\) for every
			\(\nu\in[\lambda,\mu]\), contradicting the definition of \(\bar\lambda\).

			In the newly exposed cap \(|y|\le2\kappa d_\lambda\), the inversion point
			\(I_\nu(y)=x_0+\nu^2(y-x_0)/|y-x_0|^2\) satisfies \(|I_\nu(y)|\sim d_\nu\) once
			\(\kappa\) is small. Hence
			\[
			u_{x_0,\nu}(y)\le Cd_\nu^{-\alpha_0},
			\qquad
			u(y)\ge c(\kappa d_\lambda)^{-\alpha_0}.
			\]
			After choosing \(\kappa\) sufficiently small, these estimates give the direct comparison \(u\ge4u_{x_0,\nu}\). The
			explicit barrier sizes give
			\[
			|H_\nu(y)|+\Psi_\nu(y)=o(d_\nu^{-\alpha_0})
			\]
			uniformly for \(a<a_0\), so \(W_\nu>0\) in the cap.

			In the middle region
			\[
			\mathcal M_{\lambda,\ell}:=
			\{y\in\Sigma_\lambda:
			r\ge\lambda+\ell/2,
			\ |y|\ge2\kappa d_\lambda,
			\ r<Ka\},
			\]
			\(W_\lambda\) is strictly positive by the strong maximum principle and the Hopf
			boundary lemma. The boundary zero-order term is harmless as \(c_\lambda\ge0\).
			The boundary zero of \(W_\lambda\) would contradict the Hopf lemma applied to
			\(-\partial_nW_\lambda-c_\lambda W_\lambda\ge0\). Since this region is separated from the moving sphere and from
			the cap, \(W_\mu\) depends uniformly continuously on \(\mu\). After choosing
			\(\mu-\lambda\) sufficiently small, the same positivity holds for all
			\(\nu\in[\lambda,\mu]\).

			It remains to treat the narrow shell
			\[
			\mathcal N_\nu:=
			\{y\in\Sigma_\nu:
			\nu<r<\lambda+\ell/2,
			\ |y|\ge\kappa d_\nu
			\}.
			\]
			Nonnegativity on the non-flat boundary is verified as follows. On \(r=\nu\), all
			comparison terms vanish and \(W_\nu=0\). On \(|y|=\kappa d_\nu\), the cap
			comparison applies. On \(r=\lambda+\ell/2\), either \(|y|\le2\kappa d_\lambda\),
			in which case the cap comparison applies, or \(|y|\ge2\kappa d_\lambda\), in
			which case the middle-region comparison applies. In \(\mathcal N_\nu\), Lemma~\ref{lem:sec3_combined_barrier} gives
			\[
			-\Delta W_\nu\ge0,
			\qquad
			-\partial_nW_\nu-c_\nu W_\nu\ge0.
			\]
			Since \(|y|\ge\kappa d_\nu\), Lemma~\ref{lem:sec3_boundary_linear_coeff} gives
			\(c_\nu\le C(\kappa d_\nu)^{-1}\). The shell has width at most \(\ell\), so
			the narrow domain technique can be applied again if
			\(C\ell/(\kappa d_\nu)<1/2\). Since \(d_\nu\sim d_\lambda\) for
			\(\nu\in[\lambda,\mu]\), this is ensured by choosing \(\theta\) sufficiently
			small. Thus \(W_\nu\ge0\) in the narrow shell.

			Combining the cap, middle-region and narrow-shell estimates gives
			\(W_\nu\ge0\) in the comparison domain for every \(\nu\in[\lambda,\mu]\). Since
			\(\mu>\bar\lambda\), this contradicts \(\bar\lambda=\sup\Lambda\). Therefore
			\[
			\bar\lambda\ge a-La^{1+\gamma}.
			\]
			Increasing the constant in the claim if necessary, we may assume \(C_0\ge L\).
			Now fix any \(\lambda\in[\lambda_0,a-C_0a^{1+\gamma}]\). Since
			\(\lambda\le\bar\lambda\), the definition of \(\Lambda\) gives
			\(W_\lambda\ge0\) in
			\(\Sigma_\lambda\cap\{|y|\ge\kappa(a-\lambda)\}\). Hence
			\[
			u_{x_0,\lambda}\le u+H_\lambda+\Psi_\lambda
			\le u+|H_\lambda|+\Psi_\lambda
			\]
			there. In the remaining cap \(|y|<\kappa(a-\lambda)\), the same direct cap comparison as above gives \(u\ge4u_{x_0,\lambda}\) after reducing
			\(\kappa\) and then taking \(a_0\) small. Thus
			\(u_{x_0,\lambda}\le u\), and the weaker estimate
			\eqref{eq:sec3_ms_input_used} follows in the cap as well. The sign and vanishing
			properties of \(H_\lambda\) and \(\Psi_\lambda\) follow from their definitions in
			Lemmas~\ref{lem:dyadic_transition_barrier} and~\ref{lem:sec3_combined_barrier}.
			It remains to verify the size bound stated in the claim. Let
			\(r_y=|y-x_0|\). Since \(\theta_*a\le\lambda<a\) and \(r_y<Ka\), the explicit
			formulas for \(H_\lambda=h_1+h_2+h_3\) give
			\[
			|h_1(y)|+|h_2(y)|+|h_3(y)|
			\le C(r_y-\lambda)|y|^{-\alpha_0}.
			\]
			Here \(|y|\le C_Ka\) in \(\Sigma_\lambda\), and the factor \(|y|^{-\alpha_0}\) only becomes larger as \(|y|\) decreases. The transition-barrier size estimate gives
			\[
			\Psi_\lambda(y)
			\le C(r_y-\lambda)|y|^{1-\frac n2}
			=C(r_y-\lambda)|y|^{-\alpha_0}
			\]
			in \(\mathcal T_\lambda\), and \(\Psi_\lambda=0\) outside the transition region.
			Thus \eqref{eq:sec3_ms_barrier_eval_input} follows.
		\end{proof}
		\section{Proof of Proposition~\ref{prop:lower} and \ref{prop:El-estimate}}\label{app:A}

		\begin{proof}[Proof of Proposition~\ref{prop:lower}]
			Since \(u\) is continuous and strictly positive in \(\mathcal{B}_1^+ \setminus \{0\}\), for sufficiently small \(\delta > 0\), there exists \(\Lambda_1 > 0\) such that \(u \ge \Lambda_1\) on \(\partial'' \mathcal{B}_\delta^+\).

			Let \(\phi\) be the solution to the mixed boundary value problem
			\begin{align*}
				\begin{cases}
					-L_g \phi = 0,                                                                  & \text{in } \mathcal{B}_\delta^+,            \\
					B_g \phi = \dfrac{\partial \phi}{\partial \nu_g} + \dfrac{n-2}{2} h_g \phi = 0, & \text{on } \partial' \mathcal{B}_\delta^+,  \\
					\phi = \Lambda_1,                                                               & \text{on } \partial'' \mathcal{B}_\delta^+.
				\end{cases}
			\end{align*}
			For \(\delta\) sufficiently small, the first eigenvalue of \(-L_g\) is positive, so \(-L_g\) is coercive and the unique smooth solution \(\phi\) exists by the Lax--Milgram theorem.

			Define \(W = u - \phi\). Then \(W\) satisfies
			\begin{equation*}
				\begin{cases}
					-L_g W = 0,                    & \text{in } \mathcal{B}_\delta^+\setminus\{0\},                           \\
					B_g W = u^{\frac{n}{n-2}} > 0, & \text{on } \partial' \mathcal{B}_\delta^+ \setminus \{0\}, \\
					W = u - \Lambda_1 \ge 0,       & \text{on } \partial'' \mathcal{B}_\delta^+.
				\end{cases}
			\end{equation*}
			The puncture does not create an additional negative boundary minimum.  Indeed, in the non-removable case the positivity of the solution and the isolated singularity imply \(u(x)\to +\infty\) as \(x\to0\), while \(\phi\) is bounded. Hence \(W>0\) on \(\partial''\mathcal B_\varepsilon^+\) for  sufficiently small \(\varepsilon>0\).

			Applying the maximum principle and the Hopf boundary lemma on \(\mathcal B_\delta^+\setminus\mathcal B_\varepsilon^+\), and then letting \(\varepsilon\to0\), gives \(W\ge0\).  If the singularity is removable, the same conclusion follows directly from the ordinary maximum principle on the filled half-ball.  Hence \(u \ge \phi\) in \(\mathcal{B}_\delta^+ \setminus \{0\}\).

			In the smaller half-ball \(\mathcal{B}_{\delta/2}^+\), the Harnack inequality implies \(\phi \ge \Lambda > 0\) for some constant \(\Lambda \in (0,\Lambda_1)\). Consequently,
			\begin{equation*}
				u(x) \ge \Lambda > 0 \quad\text{for all } x \in \mathcal{B}_{\delta/2}^+.
			\end{equation*}

			Translating back to \(v_k(y) = M_k^{-1} u_k(\exp_{x_k}(M_k^{-\frac{2}{n-2}}y))\) and noting that the conformal factor \(\kappa_k  \) close to $1$ near the origin, the desired estimate is obtained.
		\end{proof}

		\begin{proof}[Proof of Proposition~\ref{prop:El-estimate}]
			Write \(v_k=V+\psi_k\), and keep the notation
			\[
			H_k(y)=\frac{n-2}{2}\varepsilon_k h_{\bar g_k}(\varepsilon_k y)
			\]
			for the scaled boundary coefficient. For the limit profile \(V\), the relevant Kelvin bounds are
			\begin{equation}\label{eq:V-kelvin-routeB}
				|V^\lambda(y)|\le C_{\Lambda_0}r^{2-n},\qquad
				|\nabla V^\lambda(y)|\le C_{\Lambda_0}r^{1-n},\qquad
				|\nabla^2V^\lambda(y)|\le C_{\Lambda_0}r^{-n}.
			\end{equation}

			For the remainder \(\psi_k := v_k - V\), note that \(v_k(0) = V(0) = 1\). Since \(0\) is a local maximum of \(v_k\) on the flat boundary, the tangential derivatives vanish: \(\nabla_{y'} v_k(0) = 0 = \nabla_{y'} V(0)\). Furthermore, the boundary condition \(-\partial_{y_n} v_k(0) + H_k(0)v_k(0) = v_k(0)^{\frac{n}{n-2}}\) together with \(H_k(0)=0\) dictates that \(\partial_{y_n} v_k(0) = -1 = \partial_{y_n} V(0)\). Consequently, the remainder satisfies \(\psi_k(0)=0\) and
			\(\nabla\psi_k(0)=0\). Then in the compact region \(|y^\lambda|\le2\Lambda_0\), the convergence \(v_k \to V\) in \(C^2_{\rm loc}(\overline{\mathbb{R}^n_+})\) provides the uniform Taylor bounds
			\begin{equation}\label{eq:psi-est-routeB}
				|\psi_k(y^\lambda)|\le C\sigma_k |y^\lambda|^2,\qquad
				|\nabla\psi_k(y^\lambda)|\le C\sigma_k |y^\lambda|,\qquad
				|\nabla^2\psi_k(y^\lambda)|\le C\sigma_k.
			\end{equation}
			A further computation from \eqref{eq:psi-est-routeB} then shows the Kelvin transform of \(\psi_k\) satisfies
			\begin{equation}\label{eq:psi-kelvin-routeB}
				|\psi_k^\lambda(y)|\le C_{\Lambda_0}\sigma_k r^{-n},\qquad
				|\nabla\psi_k^\lambda(y)|\le C_{\Lambda_0}\sigma_k r^{-n-1},\qquad
				|\nabla^2\psi_k^\lambda(y)|\le C_{\Lambda_0}\sigma_k r^{-n-2}.
			\end{equation}

			It remains to estimate the terms in \(E_\lambda\). At the original point \(y\), the coefficient bounds \(|\bar b_j(y)|\le
			C\varepsilon_k\), \(|\bar d_{ij}(y)|\le C\varepsilon_kr\), and \(|\bar c(y)|\le C\varepsilon_k^3r\), together with
			\eqref{eq:V-kelvin-routeB}--\eqref{eq:psi-kelvin-routeB}, give
			\begin{align*}
				|\bar b_j(y)\partial_jv_k^\lambda(y)|
				&\le
				C\varepsilon_k r^{1-n}
				+C\sigma_k\varepsilon_k r^{-n-1},\\
				|\bar d_{ij}(y)\partial_{ij}v_k^\lambda(y)|
				&\le
				C\varepsilon_k r^{1-n}
				+C\sigma_k\varepsilon_k r^{-n-1},\\
				|\bar c(y)v_k^\lambda(y)|
				&\le
				C\varepsilon_k^3 r^{3-n}
				+C\sigma_k\varepsilon_k^3 r^{1-n}.
			\end{align*}
			Since \(r\le R_k=\delta\varepsilon_k^{-1}\), after shrinking \(\delta\) if
			necessary,
			\[
			\sigma_k\varepsilon_k^3 r^{1-n}
			\le C\sigma_k\varepsilon_k r^{-n-1}.
			\]
			Thus the original-point contribution is bounded by
			\[
			C\left(
			\varepsilon_k r^{1-n}
			+\varepsilon_k^3 r^{3-n}
			+\sigma_k\varepsilon_k r^{-n-1}
			\right).
			\]

			At the reflected point \(y^\lambda\), use \(|y^\lambda|=\lambda^2/r\), \(|\bar b_j(y^\lambda)|\le C\varepsilon_k\),
			\(|\bar d_{ij}(y^\lambda)|\le C\varepsilon_k\lambda^2/r\), and \(|\bar c(y^\lambda)|\le
			C\varepsilon_k^3\lambda^2/r\). Since \(v_k,\nabla v_k,\nabla^2v_k\) are uniformly bounded on \(B^+_{\Lambda_0}\),
			and \(\nabla\psi_k(y^\lambda)=O(\sigma_k|y^\lambda|)\), \(\nabla^2\psi_k(y^\lambda)=O(\sigma_k)\), one obtains
			\begin{align*}
				\left(\frac{\lambda}{r}\right)^{n+2}
				|\bar c(y^\lambda)v_k(y^\lambda)|
				&\le
				C_{\Lambda_0}\varepsilon_k^3
				\lambda^{n+4}r^{-n-3},\\
				\left(\frac{\lambda}{r}\right)^{n+2}
				|\bar b_j(y^\lambda)\partial_jv_k(y^\lambda)|
				&\le
				C_{\Lambda_0}\varepsilon_k
				\lambda^{n+2}r^{-n-2}
				+C_{\Lambda_0}\sigma_k\varepsilon_k
				\lambda^{n+4}r^{-n-3},\\
				\left(\frac{\lambda}{r}\right)^{n+2}
				|\bar d_{ij}(y^\lambda)\partial_{ij}v_k(y^\lambda)|
				&\le
				C_{\Lambda_0}\varepsilon_k
				\lambda^{n+4}r^{-n-3}
				+C_{\Lambda_0}\sigma_k\varepsilon_k
				\lambda^{n+4}r^{-n-3}.
			\end{align*}
			Since \(r\ge\lambda\) and \(\lambda\le\Lambda_0\),
			\[
			\lambda^{n+2}r^{-n-2}\le C_{\Lambda_0}r^{1-n},
			\qquad
			\lambda^{n+4}r^{-n-3}\le C_{\Lambda_0}r^{3-n},
			\]
			and also
			\[
			\lambda^{n+4}r^{-n-3}\le C_{\Lambda_0}r^{-n-1}
			\]
			when it multiplies the \(\sigma_k\)-terms. Hence the reflected-point
			contribution is also bounded by
			\[
			C_{\Lambda_0}\left(
			\varepsilon_k r^{1-n}
			+\varepsilon_k^3 r^{3-n}
			+\sigma_k\varepsilon_k r^{-n-1}
			\right).
			\]
			This proves \eqref{est:E@l}.

			For the boundary error, on \(\partial'\Sigma_\lambda\), using
			\(v_k^\lambda(y)=(\lambda/r)^{n-2}v_k(y^\lambda)\), it can be rewritten as
			\[
			E_\lambda^\partial(y)
			=
			-\left(\frac{\lambda}{r}\right)^{n-2}v_k(y^\lambda)
			\left[
			H_k(y)
			-\left(\frac{\lambda}{r}\right)^2H_k(y^\lambda)
			\right].
			\]
			Since
			\(|H_k(y)|\le C\varepsilon_k^2r\) and \(|\nabla H_k(y)|\le C\varepsilon_k^2\), one has
			\begin{align*}
				\left|
				H_k(y)
				-\left(\frac{\lambda}{r}\right)^2H_k(y^\lambda)
				\right|
				&\le
				|H_k(y)-H_k(y^\lambda)|
				+
				\left(1-\frac{\lambda^2}{r^2}\right)|H_k(y^\lambda)|\le
				C\varepsilon_k^2
				\left(1-\frac{\lambda^2}{r^2}\right)r.
			\end{align*}
			Here the flat-boundary identities
			\[
			|y-y^\lambda|
			=
			r-\frac{\lambda^2}{r}
			=
			r\left(1-\frac{\lambda^2}{r^2}\right),
			\qquad
			|y^\lambda|\le r
			\]
			have been used.
			Finally,
			\[
			v_k(y^\lambda)\le C_{\Lambda_0},
			\qquad
			\left(\frac{\lambda}{r}\right)^{n-2}
			\le C_{\Lambda_0}r^{2-n}.
			\]
			Therefore
			\[
			|E_\lambda^\partial(y)|
			\le
			C_{\Lambda_0}\varepsilon_k^2
			\left(1-\frac{\lambda^2}{r^2}\right)r^{3-n},
			\]
			which proves \eqref{est:Ebd-refined-routeB}.
		\end{proof}
		\section{Proof of Lemma \ref{lem:sec3_kelvin_error_estimates} and \ref{lem:dyadic_transition_barrier}}
		\label{app:sec3_moving_sphere_tools}

		This appendix contains the estimates used in the perturbed moving-sphere input of
		Section~\ref{sec:main_part}. The arguments are standard narrow-domain and
		localized-barrier arguments, but the details are recorded to make the perturbative
		comparison independent of the exact flat conformal invariance.

		\begin{proof}[Proof of Lemma~\ref{lem:sec3_kelvin_error_estimates}]
			The elementary identity
			\[
			y^\lambda
			=\frac{\lambda^2}{r^2}y+\biggl(1-\frac{\lambda^2}{r^2}\biggr)x_0
			\]
			implies
			\[
			|y^\lambda|^2
			=\left(\frac{\lambda}{r}\right)^2|y|^2
			+\left(1-\frac{\lambda^2}{r^2}\right)(a^2-\lambda^2).
			\]
			Since \(a^2-\lambda^2>0\), this gives the sharper lower bound
			\begin{equation*}
				|y^\lambda|\ge \frac{\lambda}{r}|y|\ge C_K^{-1}|y|,
			\end{equation*}
			as \(r\le Ka\) and \(\lambda\ge\theta_*a\).
			In particular, the rough bounds
			\[
			|\mathcal R_{\rm int}(y)|\le C|y|^{-\frac n2},
			\qquad
			|\mathcal R_{\rm bdy}(y)|\le C|y|^{1-\frac n2}
			\]
			hold in \(\Sigma_\lambda\).

			The transition estimates are proved first. Let \(\delta:=r-\lambda\). If
			\(\delta\ge c|y|\), the rough bounds immediately give
			\[
			|\mathcal R_{\rm int}(y)|\le C\delta |y|^{-\frac n2-1},
			\qquad
			|\mathcal R_{\rm bdy}(y)|\le C\delta |y|^{-\frac n2}.
			\]
			Consider now \(\delta<c|y|\), and put
			\[
			\omega:=\frac{y-x_0}{r},
			\qquad
			z(s):=x_0+s\omega,
			\qquad
			s_\lambda:=\frac{\lambda^2}{r}.
			\]
			Then \(y=z(r)\), \(y^\lambda=z(s_\lambda)\), and
			\[
			|r-s_\lambda|=r-\frac{\lambda^2}{r}
			=\frac{r+\lambda}{r}(r-\lambda)
			\le C(r-\lambda).
			\]
			For every \(s\in[s_\lambda,r]\), \(|z(s)-y|\le C\delta<c'|y|\), hence
			\(|z(s)|\sim |y|\). Also \(s\sim r\sim a\) as \(\lambda\ge\theta_*a\) and \(r<Ka\).

			Define
			\[
			F(s):=s^{\frac{n+2}{2}}\mathcal E_{\rm int}(z(s)).
			\]
			Since
			\[
			\left(\frac{\lambda}{r}\right)^{n+2}
			=\left(\frac{s_\lambda}{r}\right)^{\frac{n+2}{2}},
			\]
			one has
			\[
			\mathcal R_{\rm int}(y)
			=r^{-\frac{n+2}{2}}(F(r)-F(s_\lambda)).
			\]
			Using \eqref{eq:sec3_error_derivative_bounds},
			\[
			|F'(s)|
			\le Cs^{\frac n2}|z(s)|^{-\frac n2}
			+Cs^{\frac{n+2}{2}}|z(s)|^{-\frac n2-1}.
			\]
			Multiplying by \(r^{-(n+2)/2}\), and using \(s\sim r\sim a\) and
			\(|z(s)|\sim|y|\), gives
			\[
			r^{-\frac{n+2}{2}}|F'(s)|
			\le C\bigl(a^{-1}|y|^{-\frac n2}+|y|^{-\frac n2-1}\bigr)
			\le C|y|^{-\frac n2-1},
			\]
			because \(|y|\le c_0a\). The mean value theorem therefore yields
			\[
			|\mathcal R_{\rm int}(y)|
			\le C(r-\lambda)|y|^{-\frac n2-1}.
			\]

			The boundary estimate is the same. Set
			\[
			G(s):=s^{\frac n2}\mathcal E_{\rm bdy}(z(s)).
			\]
			Then
			\[
			\mathcal R_{\rm bdy}(y)=r^{-\frac n2}(G(r)-G(s_\lambda)),
			\]
			and \eqref{eq:sec3_boundary_error_derivative_bounds} gives
			\[
			r^{-\frac n2}|G'(s)|
			\le C\bigl(a^{-1}|y|^{1-\frac n2}+|y|^{-\frac n2}\bigr)
			\le C|y|^{-\frac n2}.
			\]
			Hence
			\[
			|\mathcal R_{\rm bdy}(y)|
			\le C(r-\lambda)|y|^{-\frac n2}.
			\]

			For the regular estimates, first assume \(r\le2a\). Then
			\(|y|\ge c_0a\) and \(|y|\le r+a\le3a\), hence \(|y|\sim a\). The rough
			bound gives
			\[
			|\mathcal R_{\rm int}(y)|\le C|y|^{-\frac n2}\le Ca^{-\frac n2}.
			\]
			For the boundary error, if \(\delta:=r-\lambda\ge c|y|\), then
			\[
			|\mathcal R_{\rm bdy}(y)|\le C|y|^{1-\frac n2}
			\le C\delta |y|^{-\frac n2}
			\le C(r-\lambda)a^{-\frac n2}.
			\]
			If \(\delta<c|y|\), then for \(s\in[s_\lambda,r]\) the segment
			\(z(s)=x_0+s\omega\) satisfies \(|z(s)|\sim |y|\sim a\). The same
			one-dimensional mean-value argument used above gives
			\[
			r^{-\frac n2}|G'(s)|\le Ca^{-\frac n2},
			\]
			and therefore
			\[
			|\mathcal R_{\rm bdy}(y)|\le C(r-\lambda)a^{-\frac n2}.
			\]
			This proves the regular estimates when \(r\le2a\).

			If \(r\ge2a\), then \(r-\lambda\ge a/2\) and \(|y|\ge r-a\ge r/2\ge a\).
			The rough estimates imply
			\[
			|\mathcal R_{\rm int}(y)|\le Ca^{-\frac n2},
			\]
			and
			\[
			|\mathcal R_{\rm bdy}(y)|\le C|y|^{1-\frac n2}
			\le Ca^{1-\frac n2}
			\le C(r-\lambda)a^{-\frac n2}.
			\]
			This proves the lemma.
		\end{proof}

		\begin{proof}[Proof of Lemma~\ref{lem:dyadic_transition_barrier}]
			Put
			\[
			\delta:=r-\lambda,
			\qquad
			\sigma_\lambda(r)
			:=\frac{\lambda}{n-2}\left[1-\left(\frac{\lambda}{r}\right)^{n-2}\right].
			\]
			Then
			\[
			\Delta\sigma_\lambda=0,
			\qquad
			\sigma_\lambda(\lambda)=0,
			\]
			and, since \(\theta_*a\le\lambda\le r\le Ka\),
			\begin{equation*}
				C^{-1}\delta\le\sigma_\lambda(r)\le C\delta,
				\qquad
				|\nabla\sigma_\lambda|\le C.
			\end{equation*}
			On the flat boundary,
			\begin{equation*}
				\partial_n\sigma_\lambda
				=\sigma_\lambda'(r)\frac{y_n}{r}=0.
			\end{equation*}

			Choose smooth functions \(\Theta_{\rm col},\Theta_{\rm bulk}\in C^\infty([0,\infty))\)
			such that
			\[
			0\le\Theta_{\rm col},\Theta_{\rm bulk}\le1,
			\qquad
			\Theta_{\rm col}+\Theta_{\rm bulk}=1,
			\]
			\[
			\Theta_{\rm col}=1\text{ on }[0,\eta],
			\qquad
			\Theta_{\rm col}=0\text{ on }[2\eta,\infty),
			\]
			and therefore
			\[
			\Theta_{\rm bulk}=0\text{ on }[0,\eta],
			\qquad
			\Theta_{\rm bulk}=1\text{ on }[2\eta,\infty).
			\]
			The derivatives of these cutoffs are supported in
			\[
			\eta |y|\le \delta\le2\eta |y|.
			\]
			In this overlap region,
			\begin{equation}\label{eq:sec3_overlap_sigma_large}
				\sigma_\lambda\sim\delta\sim |y|.
			\end{equation}
			Hence every commutator without an explicit \(\sigma_\lambda\)-factor has size
			\(O(|y|^{-n/2})\), which is bounded by
			\(C\sigma_\lambda |y|^{-n/2-1}\).

			The proof next records the local barrier construction used on each
			dyadic patch. Let \(\zeta\) be one of the two cutoff types used below, at a
			transition scale \(\rho\), with support contained in a patch where \(|y|\sim\rho\);
			in the collar case one also has \(|p_\lambda(\omega)|\sim\rho\). For each
			such admissible cutoff \(\zeta\), a nonnegative local barrier
			\(B_{\rho,\zeta}\) is constructed with
			\[
			B_{\rho,\zeta}=0\quad\text{on }r=\lambda,
			\qquad
			0\le B_{\rho,\zeta}\le C\sigma_\lambda\rho^{1-\frac n2},
			\]
			\[
			-\Delta B_{\rho,\zeta}
			\ge c\zeta\sigma_\lambda\rho^{-\frac n2-1}
			\]
			in the interior part of the patch, and on the flat boundary,
			\[
			-\partial_nB_{\rho,\zeta}-c_\lambda B_{\rho,\zeta}
			\ge c\zeta\sigma_\lambda\rho^{-\frac n2}.
			\]
			Here the constants depend only on \(n\), \(\theta_*\), and the fixed cutoff
			family. The construction is scale invariant. Rescale the patch by
			\(y=y_*+\rho z\) and set
			\[
			S(z):=\rho^{-1}\sigma_\lambda(y_*+\rho z).
			\]
			Since \(r^{2-n}\) is harmonic away from \(x_0\), \(S\) is harmonic on the
			rescaled patch. Moreover \(S\ge0\), \(S=0\) on the moving-sphere face,
			\(|\nabla_zS|\le C\), and, on the flat boundary,
			\(\partial_{z_n}S=0\). In a fixed boundary strip \(0\le z_n\le\eta_*\), take
			\[
			V_0(z):=S(z)q(z_n),
			\qquad
			q(t)=1-\frac{t}{2\eta_*}-\frac{t^2}{4\eta_*^2}.
			\]
			Then on \(z_n=0\),
			\[
			-\partial_{z_n}V_0-\rho c_\lambda V_0
			=S\{-q'(0)-\rho c_\lambda q(0)\}.
			\]
			Since \(c_\lambda\le C|y|^{-1}\) and \(|y|\sim\rho\), one has
			\(\rho c_\lambda\le C\). Choosing \(\eta_*\) small gives a positive boundary
			margin. Also,
			\[
			-\Delta_zV_0=-S q''-2(\partial_{z_n}S)q'
			\ge cS
			\]
			in the boundary strip, after reducing \(\eta_*\), for \(q'<0\), \(q''<0\),
			and \(\partial_{z_n}S\ge0\) in the half-space. Away from the flat boundary, use a
			fixed scale-one Dirichlet barrier \(V_1\), vanishing on the moving-sphere face,
			with
			\[
			0\le V_1\le CS,
			\qquad
			-\Delta_zV_1\ge C_*S
			\]
			on the support of the enlarged local cutoff. This follows from a fixed mixed
			boundary problem on a compact family of rescaled patches and the scale-one Green
			estimate. Choose \(V_1\) with a fixed vertical cutoff so that it vanishes near
			\(z_n=0\); hence it does not disturb the positive Robin margin supplied by
			\(M_0V_0\). Taking \(V=M_0V_0+V_1\), with \(M_0\) and then \(C_*\) fixed large,
			produces simultaneous interior and boundary margins on the enlarged patch.

			The localization is harmless after choosing an enlarged admissible cutoff
			\(\widehat\zeta\) from the same fixed family, with \(\widehat\zeta\equiv1\) on
			\(\operatorname{supp}\zeta\), and constructing the scale-one barrier on
			\(\operatorname{supp}\widehat\zeta\). The commutators produced by the final
			localization are bounded by a fixed multiple of the scale-one margins, since the
			cutoff family is compact in \(C^2\). For the bulk cutoffs this uses
			\(\delta\gtrsim |y|\) on the support where the radial transition derivatives are
			active, hence \(\sigma_\lambda\sim |y|\). For the collar cutoffs, the dyadic cutoff
			is constant along the \(x_0\)-centered radial rays, so the dangerous derivative
			falling on \(\sigma_\lambda\) is absent; the remaining derivatives belong to the
			fixed compact cutoff family. Denote the resulting localized scale-one barrier
			still by \(V\). Scaling back by setting
			\[
			B_{\rho,\zeta}(y)
			:=\rho^{2-\frac n2}V\!\left(\frac{y-y_*}{\rho}\right)
			\]
			gives the local barrier with the stated bounds.

			The bulk part is now constructed. Let \(\{\chi_j\}\) be a standard dyadic partition
			in \(|y|\) on \(\{\kappa d\le |y|\le2c_0a\}\), with
			\[
			\chi_j\subset\{\rho_j/2\le |y|\le4\rho_j\},
			\qquad
			|\nabla\chi_j|\le C\rho_j^{-1},
			\qquad
			|\Delta\chi_j|\le C\rho_j^{-2},
			\]
			and bounded overlap. Define
			\[
			\zeta_j^{\rm bulk}(y)
			:=\chi_j(|y|)\Theta_{\rm bulk}\!\left(\frac{\delta}{|y|}\right).
			\]
			On the support of \(\zeta_j^{\rm bulk}\), either \(\delta\ge2\eta|y|\), or the
			point lies in the overlap where \eqref{eq:sec3_overlap_sigma_large} holds. In
			both cases \(\sigma_\lambda\ge c\rho_j\), and the radial dyadic cutoff is one of
			the admissible cutoffs in the local construction above. Let
			\(B_j^{\rm bulk}:=B_{\rho_j,\zeta_j^{\rm bulk}}\). Then
			\[
			0\le B_j^{\rm bulk}\le C\sigma_\lambda\rho_j^{1-\frac n2},
			\]
			\[
			-\Delta B_j^{\rm bulk}
			\ge c\zeta_j^{\rm bulk}\sigma_\lambda\rho_j^{-\frac n2-1},
			\]
			and the analogous flat-boundary Robin inequality holds.

			For the collar part, radial cutoffs in \(|y|\) are not
			allowed. Set
			\[
			\omega:=\frac{y-x_0}{r},
			\qquad
			p_\lambda(\omega):=x_0+\lambda\omega.
			\]
			In the collar \(\delta\le2\eta|y|\), if \(\eta\) is sufficiently small then
			\begin{equation}\label{eq:sec3_projection_comparable}
				|p_\lambda(\omega)|\sim |y|.
			\end{equation}
			Indeed, \(y=p_\lambda(\omega)+\delta\omega\), so the upper bound is immediate,
			and the lower bound follows from
			\[
			|p_\lambda(\omega)|\ge |y|-\delta\ge(1-2\eta)|y|.
			\]
			Let \(\{\chi_j^A\}\) be a dyadic partition in
			\(|p_\lambda(\omega)|\), and put
			\[
			\zeta_j^{\rm col}(y)
			:=\chi_j^A\bigl(|p_\lambda(\omega)|\bigr)
			\Theta_{\rm col}\!\left(\frac{\delta}{|y|}\right).
			\]
			The angular cutoffs are constant along the \(x_0\)-centered radial rays. Hence
			\begin{equation*}
				\nabla\chi_j^A\cdot\nabla\sigma_\lambda=0.
			\end{equation*}
			This removes the only cutoff commutator that could lack a factor
			\(\sigma_\lambda\). The derivatives of \(\Theta_{\rm col}\) are supported only in
			the overlap region, where \eqref{eq:sec3_overlap_sigma_large} applies.

			The flat-boundary compatibility of the angular cutoffs is also needed. If
			\(y_n=0\), then \(p_{\lambda,n}=0\) and
			\[
			\partial_n p_\lambda
			=\frac{\lambda}{r}e_n.
			\]
			Therefore
			\[
			\partial_n|p_\lambda|
			=\frac{p_\lambda}{|p_\lambda|}\cdot\partial_np_\lambda
			=\frac{p_\lambda}{|p_\lambda|}\cdot\frac{\lambda}{r}e_n=0.
			\]
			Thus \(\partial_n\chi_j^A=0\) on the flat boundary, and no additional
			normal cutoff term appears in the Robin operator. Consequently
			\(\zeta_j^{\rm col}\) is also admissible for the local construction above. Let
			\(B_j^{\rm col}:=B_{\rho_j,\zeta_j^{\rm col}}\), where
			\(\rho_j\sim |p_\lambda(\omega)|\). By \eqref{eq:sec3_projection_comparable},
			\(\rho_j\sim |y|\) on the support of \(\zeta_j^{\rm col}\), and
			\[
			0\le B_j^{\rm col}\le C\sigma_\lambda\rho_j^{1-\frac n2},
			\]
			\[
			-\Delta B_j^{\rm col}
			\ge c\zeta_j^{\rm col}\sigma_\lambda\rho_j^{-\frac n2-1},
			\]
			with the analogous flat-boundary Robin inequality.

			Define
			\[
			\Psi_\lambda
			:=\sum_j B_j^{\rm bulk}+\sum_j B_j^{\rm col}.
			\]
			The overlap number is uniformly bounded. Since
			\(\Theta_{\rm bulk}+\Theta_{\rm col}=1\), the weights
			\(\{\zeta_j^{\rm bulk}\}\cup\{\zeta_j^{\rm col}\}\) cover the transition core.
			Using the local construction above and
			\(\rho_j\sim |y|\) on each patch, one obtains
			\[
			-\Delta\Psi_\lambda
			\ge c\sigma_\lambda |y|^{-\frac n2-1},
			\]
			and on the flat boundary,
			\[
			-\partial_n\Psi_\lambda-c_\lambda\Psi_\lambda
			\ge c\sigma_\lambda |y|^{-\frac n2}.
			\]
			The transition error estimates from Lemma~\ref{lem:sec3_kelvin_error_estimates} give
			\[
			|\mathcal R_{\rm int}|
			\le C(r-\lambda)|y|^{-\frac n2-1},
			\qquad
			|\mathcal R_{\rm bdy}|
			\le C(r-\lambda)|y|^{-\frac n2}.
			\]
			Since \(\sigma_\lambda\sim r-\lambda\), the preceding two lower bounds dominate
			\(|\mathcal R_{\rm int}|\) and \(|\mathcal R_{\rm bdy}|\) after fixing the
			constants in the local scale-one construction once and for all.

			Finally, the local size bound and bounded overlap imply
			\[
			0\le\Psi_\lambda(y)
			\le C\sigma_\lambda(r)|y|^{1-\frac n2}
			\le C(r-\lambda)|y|^{1-\frac n2}.
			\]
			In the outer cutoff shell \(c_0a\le |y|\le2c_0a\), every active dyadic patch has
			\(\rho_j\sim a\). The same scale-one bounds for the localized pieces give
			\[
			|\Delta B_j|\le C\sigma_\lambda a^{-\frac n2-1}
			\le Ca^{-\frac n2},
			\]
			and, on the flat boundary,
			\[
			\bigl|-\partial_nB_j-c_\lambda B_j\bigr|
			\le C\sigma_\lambda a^{-\frac n2}
			\le C(r-\lambda)a^{-\frac n2}.
			\]
			By bounded overlap, the same estimates hold for \(\Psi_\lambda\). Also every
			local piece contains the factor \(\sigma_\lambda\), hence \(\Psi_\lambda=0\) on
			\(r=\lambda\). This proves the lemma.
		\end{proof}

						\bigskip

					\noindent Y. Liao

					\noindent School of Mathematical Sciences, Beijing Normal University\\
					Beijing 100875, China\\[1mm]
					Email: \textsf{yxliao@mail.bnu.edu.cn}
						\bigskip

						\noindent Y. Ma

						\noindent School of Mathematical Sciences, Beijing Normal University\\
						Beijing 100875, China\\[1mm]
						Email: \textsf{24yxma@mail.bnu.edu.cn}
					\end{document}